\documentclass[a4paper,11pt,reqno]{amsart}
\usepackage[utf8x]{inputenc}
\usepackage{amsmath,amsthm,amssymb}
\usepackage{amsaddr}
\usepackage{bbm}
\usepackage{tikz}
\usetikzlibrary{patterns}
\usepackage{array}
\usepackage[noadjust]{cite}
\usepackage[labelformat=empty]{subcaption}
\usepackage{graphicx}
\usepackage{hyperref}
\usepackage{fullpage}
\usepackage{times,soul}

\usepackage[titletoc,title]{appendix}

\def\MR#1{}

\makeatletter
\let\@fnsymbol\@arabic
\makeatother 

\newtheorem{thm}{Theorem}[section]
\newtheorem{lem}[thm]{Lemma}
\newtheorem{prop}[thm]{Proposition}
\newtheorem{cor}[thm]{Corollary}
\newtheorem{claim}[thm]{Claim}

\theoremstyle{definition}
\newtheorem{definition}[thm]{Definition}
\newtheorem{remark}[thm]{Remark}
\newtheorem{example}[thm]{Example}

\newcommand{\FF}{\mathbb{F}}
\newcommand{\ZZ}{\mathbb{Z}}
\newcommand{\cG}{\mathcal{G}}

\newcommand{\cS}{\mathcal{S}}
\newcommand{\cT}{\mathcal{T}}
\newcommand{\cZ}{\mathcal{Z}}
\newcommand{\cB}{\mathcal{B}}
\newcommand{\cN}{\mathcal{N}}
\newcommand{\cJ}{\mathcal{J}}
\newcommand{\cD}{\mathcal{D}}

\newcommand{\cF}{\mathcal{F}}

\newcommand{\EE}{\mathbb{E}}

\newcommand{\Bi}{\mathrm{Bi}}
\newcommand{\Po}{\mathrm{Po}}

\newcommand{\dist}{\mathrm{d}}
\newcommand{\ord}{\mathrm{Ord}}

\newcommand{\Mjk}{\ensuremath{M_{j,k}}}
\newcommand{\Njk}{\ensuremath{M_{j,k}}}
\newcommand{\Mjkhat}{\ensuremath{\hat M_{j,k}}}

\newcommand{\Njkbar}{\ensuremath{M_{j,\bar k}}}

\newcommand{\Mjkzero}{\ensuremath{M_{j,k_0}}}
\newcommand{\Njkzero}{\ensuremath{M_{j,k_0}}}
\newcommand{\Mjkzerohat}{\ensuremath{\hat M_{j,k_0}}}

\newcommand{\Mjlhat}{\ensuremath{\hat M_{j,\ell}}}

\newcommand{\Njj}{\ensuremath{M_{j,j}}}
\newcommand{\Mjjhat}{\ensuremath{\hat M_{j,j}}}

\newcommand{\tim}{\ensuremath{\tau}}
\newcommand{\bq}{\ensuremath{\bar q}}

\newcommand{\critm}{\ensuremath{\mathcal{E}}}
\newcommand{\critset}{\ensuremath{\mathcal{C}}}

\newcommand{\pj}{\ensuremath{p_j}}
\newcommand{\allp}{\mathbf{p}}
\newcommand{\barallp}{\mathbf{\bar p}}

\newcommand{\varisol}{\ensuremath{X_{j,j}}}
\newcommand{\varNjk}{\ensuremath{X_{j,k}}}
\newcommand{\barvarNjk}{\ensuremath{\bar X_{j,k}}}
\newcommand{\varMjkhat}{\ensuremath{\hat X_{j,k}}}
\newcommand{\varNjkbar}{\ensuremath{X_{j,\bar k}}}

\newcommand{\varMjkzerohat}{\ensuremath{\hat X_{j,k_0}}}

\newcommand{\logp}{\lambda}
\newcommand{\sublogp}{\mu}
\newcommand{\constp}{\nu}

\newcommand{\allt}{\mathbf{t}}
\newcommand{\allu}{\mathbf{u}}

\newcommand{\formalconnected}{$R$-cohomologically $j$-connected}
\newcommand{\connected}{$j$-cohom-connected}

\newcommand{\connectedness}{$j$-cohom-connectedness}

\newcommand{\eps}{\varepsilon}

\newcommand{\largecon}{\ensuremath{h}}

\renewcommand{\Pr}{\mathbb{P}}

\renewcommand{\theenumi}{(\alph{enumi})}

\DeclareMathOperator{\im}{im}

\title{Phase transition in cohomology groups of non-uniform random simplicial complexes}
\author{Oliver Cooley\textsuperscript{$*$}, Nicola Del Giudice\textsuperscript{$*$}, Mihyun Kang\textsuperscript{$*$}, Philipp Spr{\"u}ssel\textsuperscript{$*$}}
\address{Institute of Discrete Mathematics, Graz University of Technology, Steyrergasse 30, 8010 Graz, Austria}
\email{\{cooley,delgiudice,kang,spruessel\}@math.tugraz.at}
\thanks{\textsuperscript{$*$}Supported by Austrian Science Fund (FWF): I3747, W1230.}


\begin{document}

\begin{abstract}
  We consider a generalised model of a random simplicial complex,
  which arises from a random hypergraph. Our model is generated by
  taking the downward-closure of a non-uniform binomial random 
  hypergraph, in which for each $k$, each set of $k+1$ vertices forms an edge with
  some probability $p_k$ independently. As a special case, this contains an
  extensively studied model of a (uniform) random simplicial complex,
  introduced by Meshulam and Wallach [Random Structures \& Algorithms
  34 (2009), no.\ 3, pp.\ 408--417].

  We consider a higher-dimensional notion of connectedness on this
  new model according to the vanishing of cohomology groups over an
  arbitrary abelian group $R$. We prove that this notion of
  connectedness displays a phase transition and determine the
  threshold. We also prove a hitting time result for a natural process
  interpretation, in which simplices and their downward-closure are
  added one by one. 
  In addition, we determine the asymptotic behaviour of 
  cohomology groups inside the critical window around the time of
  the phase transition.
\end{abstract}

\maketitle

\section{Introduction}\label{sec:intro}

\subsection{Motivation}\label{sec:motivation}
One of the first and most famous results in the theory of random graphs, due
to Erd\H{o}s and R\'enyi~\cite{ErdosRenyi59}, states that the uniform random graph
$G(n,m)$ displays a phase transition threshold for the property
of being connected at about $m=\frac12n\log n$ edges (where $\log$ denotes the natural logarithm). Almost equivalently,
in modern terminology, with high probability
the binomial random graph $G(n,p)$ becomes
connected around $p=\frac{\log n}{n}$ (see~\cite{Stepanov70}).

The result was subsequently strengthened by Bollob\'as and Thomason~\cite{BollobasThomason85}
to a \emph{hitting time} result---the random graph process, in which
edges are added to an initially empty graph one by one in a uniformly random order,
is very likely to become connected at exactly the moment at which
the last isolated vertex disappears (i.e.\ acquires an edge).

More recently, there has been a focus on generalising graphs
to higher-dimensional structures. One very well-studied
higher-dimensional analogue of graphs is hypergraphs
(most often uniform hypergraphs),
in which one may consider vertex-connectedness
(see e.g.~\cite{BehrischCojaOghlanKang10b,BCOK14,BollobasRiordan12c,
BollobasRiordan16,BollobasRiordan17,BollobasRiordan18,DarlingNorris05,KaronskiLuczak96,Poole15,
SchmidtShamir85})
or high-order connectedness (also known as $j$-tuple-connectedness,
e.g.~\cite{CooleyKangKoch16,CKKgiant,CooleyKangPerson18,KahlePittel16}),
as well as the
appearance of spanning structures such as Hamilton cycles
(see e.g.~\cite{AllenBoettcherKohayakawaPerson15,AllenKochParczykPerson18,NenadovSkoric19,ParczykPerson16}).

Simplicial complexes have also seen a great deal of attention
as higher-dimensional analogues of graphs.
The study of random simplicial complexes was initiated by
Linial and Meshulam~\cite{LinialMeshulam06}, who studied
a model on vertex set $[n]$ in which each $2$-simplex is present 
with probability $p=p(n)$ independently, and all $1$-simplices
are always present.
The notion of connectedness they studied involved the
vanishing of the first homology group over $\FF_2$
(or equivalently the first cohomology group over $\FF_2$),
and they proved that this property undergoes a phase transition
at threshold $p=\frac{2\log n}{n}$. This threshold is
related to the disappearance of the last isolated $1$-simplex
(i.e.\ a $1$-simplex that does not lie in any $2$-simplex)
as was subsequently proved by Kahle and Pittel~\cite{KahlePittel16}.

Meshulam and Wallach~\cite{MeshulamWallach08} extended the result
of~\cite{LinialMeshulam06} to random simplicial $k$-complexes
with full $(k-1)$-skeleton
(for any $k\ge 2$), proving that the threshold
for the vanishing of the $(k-1)$-th (co)homology group over $\FF_2$,
or indeed over any finite abelian group $R$,
undergoes a phase transition at threshold $p=\frac{k\log n}{n}$.
In~\cite{CooleyDelGiudiceKangSpruessel20}, we proved the corresponding hitting time
result for cohomology over $\FF_2$, relating
cohomological connectedness to the disappearance of the last
isolated $(k-1)$-simplex,
as a corollary of results about a slightly different model
of random simplicial $k$-complexes generated from a
random binomial $(k+1)$-uniform hypergraph by taking the downward-closure
(so in particular, the complex does not necessarily
have a full $(k-1)$-skeleton).
A similar hitting time result in the Linial-Meshulam model and for homology groups over $\ZZ$
was proved by {\L}uczak and Peled~\cite{LuczakPeled16} in the case when $k=2$ and recently by
Newman and Paquette~\cite{NewmanPaquette18} for general $k\ge 2$.

Since the work of Linial and Meshulam, many different models of random simplicial
complexes have been introduced 
(see e.g.~\cite{CostaFarberHorak15,FountoulakisIyerMaillerSulzbach19,Kahle07,Kahle09,Kahle14,KahleSurvey14,
LinialPeled16}),
and several notions of connectedness
have been analysed 
(see e.g.\ \cite{AronshtamLinial15,AronshtamLinialLuczakMeshulam12,HoffmanKahlePaquette17,LuczakPeled16,NewmanPaquette18}),
as well as related concepts such as expansion~\cite{FountoulakisPrzykucki20}
and bootstrap percolation~\cite{FountoulakisPrzykucki19}.
In this paper, we consider
a model of random simplicial complexes generated from
\emph{non-uniform} random hypergraphs, and study
cohomology groups over an arbitrary (not necessarily
finite) abelian group $R$. We note that our model includes
both the model introduced by Linial and Meshulam, which was
extended by Meshulam and Wallach, and the model we introduced
in~\cite{CooleyDelGiudiceKangSpruessel20} as special cases, and therefore our
main result extends and unifies the results of~\cite{CooleyDelGiudiceKangSpruessel20},~\cite{LinialMeshulam06},
and~\cite{MeshulamWallach08}. We also note that our model is equivalent to the `upper model'
which was recently introduced independently by Farber, Mead, and Nowik~\cite{FMN19},
although they considered different properties and different ranges of probabilities
to the ones we focus on in this paper.

\bigskip
\subsection{Model}\label{sec:model}

Throughout 
the paper let $d\geq 2$ be a fixed integer and let $R$
be an abelian group with at least two elements. We use additive notation for the group
operation of $R$ and denote the identity element by $0_R$. For an integer
$k \ge 1$, we write $[k]\ :=\ \{1,\dotsc,k\}$ and $[k]_0 \ :=\ 
\{0,\dotsc,k\}$. If $A$ is a set with at least $k$
elements, we denote by $\binom{A}{k}$ the family of $k$-element subsets of $A$ and we call $K \in \binom{A}{k}$ a \emph{$k$-set} of $A$.

\begin{definition}
  A family $\cG$ of non-empty finite subsets of a vertex set $V$ is
  called a \emph{simplicial complex} on $V$ if it is downward-closed, i.e.\
  if every non-empty set $A$ that is contained in a set $B\in \cG$
  also lies in $\cG$, and if furthermore the singleton $\{v\}$ is in
  $\cG$ for every $v\in V$.
  
  The elements of a simplicial complex $\cG$ which have cardinality $i+1$ are
  called \emph{$i$-simplices} of $\cG$. If $\cG$ has no
  $(d+1)$-simplices, then we call it \emph{$d$-dimensional},
  or a
  \emph{$d$-complex}.\footnote{Note that we do not require $\cG$ to
    contain any $d$-simplices in order to be $d$-dimensional. This is
    in contrast to the usual terminology, but we adopt this convention
    for technical convenience.}
  If $\cG$ is a $d$-complex, then for each
  $j \in [d-1]_0$ the \emph{$j$-skeleton} of $\cG$ is the $j$-complex
  formed by \emph{all} $i$-simplices in $\cG$ with $i \in [j]_0$.
\end{definition}

We define a model of a random $d$-complex generated from a \emph{non-uniform} random hypergraph,
in which sets of vertices have different probabilities of forming an edge
depending on their size.

\begin{definition}\label{def:complexGp}
  For each $k\in[d]$, let $p_k=p_k(n)\in[0,1] \subset \mathbb{R}$ be
  given and write $\allp \ :=\  (p_1,\dotsc,p_d)$. Denote by
  $G(n,\allp)$ the (non-uniform) binomial random hypergraph on vertex set $[n]$ in
  which, for all $k\in[d]$, each element of $\binom{[n]}{k+1}$ forms an 
  edge with probability $p_k$ independently. By
  $\cG(n,\allp)$, we denote the random $d$-dimensional simplicial
  complex on $[n]$ such that
  \begin{itemize}
  \item the $0$-simplices of $\cG(n,\allp)$ are the singletons of $[n]$
    and
  \item for each $i\in[d]$, the $i$-simplices are precisely the
    $(i+1)$-sets which are contained in edges of $G(n,\allp)$.
  \end{itemize}
  In other words, $\cG(n,\allp)$ is the downward-closure of the set of
  edges of $G(n,\allp)$, together with all singletons of $[n]$ (if
  these are not already in the downward-closure).\footnote{Note that if $\binom{n}{d+1}p_d$ is small, then it is likely that
there are no $d$-simplices---it is for this reason that we
slightly abuse terminology by referring to a $d$-complex even
if there may not be any $d$-simplices.}
\end{definition}

Denote by $H^i(\cG;R)$ the $i$-th cohomology group of a simplicial
complex $\cG$ with coefficients in $R$ (see~\eqref{eq:cohomdef} in
Section~\ref{sec:cohom} for a formal definition). It is
well-known that $H^0(\cG;R)=R$ if and only if $\cG$ is
connected in the topological sense (see e.g.\ \cite[Theorem~42.1]{Munkres84}), which we call \emph{topologically
connected} in order to distinguish it from other notions of
connectedness. Observe that  
topological connectedness of $\cG$ is equivalent
to vertex-connectedness
of the underlying hypergraph.
For any integer $i \ge 1$, the  \emph{vanishing} of $H^i(\cG;R)$ can
be viewed as a `higher-order connectedness' of $\cG$.

\begin{definition} \label{def:cohomconn}
  Given a non-negative integer $j$, a simplicial complex $\cG$ is called
  \emph{\formalconnected} (\emph{\connected} for short) if
  \begin{enumerate}
  \item $H^0(\cG;R) = R\,$;
  \item $H^{i}(\cG;R)= 0$ for all $i \in [j]$.
  \end{enumerate}
\end{definition}

We note that the analogous definition of connectedness considered
by Meshulam and Wallach in~\cite{MeshulamWallach08} was only for
the case $j=d-1$, and only demanded the vanishing of the
$(d-1)$-th cohomology group---this was reasonable for their model since with the
complete $(d-1)$-dimensional skeleton, the $i$-th cohomology group
must always vanish for all $i\in[d-2]$
(or equal $R$ if $i=0$).

\bigskip
\subsection{Main results}\label{sec:main}

We will consider asymptotic properties of $\cG(n,\allp)$ as the
number of vertices $n$ tends to infinity, hence all asymptotics
in the paper are with respect to $n$. In particular, we say that
a property or an event holds \emph{with high probability},
abbreviated to \emph{whp}, if the probability tends to $1$ as
$n$ tends to infinity.

Our first main theorem will relate the \connectedness\ of $\cG(n,\allp)$
to the absence of any \emph{minimal obstructions} to this property. 
We call these obstructions \emph{copies of
$\Mjkhat$} for any $k$ with $j\le k\le d$ 
(these will be defined later,
see Definitions~\ref{def:Mjkhat} 
and~\ref{def:Mjjhat}), and   
we will see in Section~\ref{sec:minimal} that the 
presence of any of these configurations in 
$\cG(n,\allp)$ is a witness for the 
\emph{non-vanishing} of
$H^j(\cG(n,\allp);R)$ (Corollary~\ref{cor:obstruction}), 
which is `minimal' in a natural sense (Lemma~\ref{lem:minobst}).

In particular, the strongest relation between 
\connectedness\ and the absence of copies of $\Mjkhat$ will be a hitting time result,
analogous to the result of Bollob\'as and Thomason~\cite{BollobasThomason85}
for graphs, for which we will need to turn the 
random $d$-complex $\cG(n,\allp)$ into a \emph{process}.
We do this by assigning a \emph{birth time} to each $k$-simplex:
more precisely,
for each $k\in[d]$ and each $(k+1)$-set $K \in \binom{[n]}{k+1}$
independently, sample a birth time uniformly at random from $[0,1]$.
Then $\cG(n,\allp)$ is exactly the complex generated by the
$(k+1)$-sets with birth times at most $p_k$, for all $k\in[d]$, by
taking the downward-closure. If we fix a `direction' $\barallp =
(\bar p_1,\dotsc,\bar p_d)$ of non-negative real numbers (not necessarily
less than $1$) with $\bar p_d \not= 0$, set
\begin{equation*}
  \allp = \tim\barallp \ :=\  (\min\{\tim \bar p_1,1\},\dotsc,\min\{\tim \bar p_d,1\}),
\end{equation*}
and gradually increase $\tim$ from $0$ to
$$
\tim_{\max}\ :=\ 1/\bar p_d, \label{prob:timmax}
$$
then $\cG(n,\allp)$ becomes
a process in which simplices (together with their downward-closure)
arrive one by one.\footnote{Observe that by time $\tim = \tim_{\max}$, all $d$-simplices will be present
deterministically, and therefore also all simplices of dimension $k\le d$
will be present as part of their downward-closure.}
We will denote this process by
$(\cG(n,\tim \barallp))_{\tim\in [0,\tim_{\max }]}$, or sometimes just by $(\cG_\tim)$ when
the direction $\barallp$ is clear from the context.
In this way, $\tim$ may be thought of as a `time' parameter. 
Let us note that if we consider
a \emph{snapshot} of the process
$(\cG_\tim)$ at time $\tim=\tim_0$, then it has the same distribution
as $\cG(n,\tim_0\barallp)$. Therefore we will often give definitions or
state and prove results for
the random complex $\cG_\tim := \cG(n,\tim\barallp)$ for some appropriate value of $\tim$,
and subsequently apply them to the process at that time, meaning in particular that
$\cG_{\tim_0} \subseteq \cG_{\tim_1}$ if $\tim_0\le \tim_1$,
i.e.\ we have a natural coupling of the random complexes rather
than sampling them independently.
In other words, for the rest of the paper we take one sample of
random birth times uniformly from $[0,1]$ and independently for all simplices,
and whenever we refer to $\cG_\tim$, we mean the complex
generated by the simplices with \emph{scaled birth times} (scaled according to $\barallp$)
at most $\tim$ (see~\eqref{eq:scaledbirth} in 
Section~\ref{sec:prelim} for 
the formal definition of scaled birth time).

Note that the evolution of the process $(\cG_{\tim})$
is unchanged if the direction $\barallp$ is scaled by a multiplicative factor.
Therefore
we would like to scale $\barallp$ so that
we expect the last copy of $\Mjkhat$ to disappear when $\tim$ is close
to $1$. Indeed, 
our first main result (Theorem~\ref{thm:main}) in particular states that this happens for
a specific type of direction that we call 
\emph{$j$-critical} and that will be formally defined 
in Section~\ref{sec:parametrisation} (Definition~\ref{def:criticaldirection}).

\begin{thm}[Hitting time]\label{thm:main}
  For $j\in[d-1]$ and a $j$-critical direction 
  $\barallp = (\bar p_1,\bar p_2,\ldots, \bar p_d) $ with $\bar p_d \neq 0$,
  let $\tim_{\max}= 1/\bar p_d$ and consider the process
  $(\cG_\tim)=(\cG(n,\tim \barallp))_{\tim \in [0,\tim_{\max}]}$. 
  Let
\label{prob:timstar}
  \begin{equation*}
   \tim_j^*\ :=\ \sup\{ \tim \in \mathbb{R}_{\ge0} \mid \cG_\tim \text{ contains a copy of }\Mjkhat\text{ for some } k \text{ with }j \le k \le d \}.
  \end{equation*}
  Then for every function $\omega$ of $n$ which tends to infinity as
  $n\to\infty$, the following statements hold with high probability.
  \begin{enumerate}
  \item\label{main:hitting}
    $\tim_j^* = 1+o\left(\frac{\omega}{\log n}\right)$.
  \item\label{main:subcritical}
    For all $\tim \in [0,\tim_j^*)$, the random $d$-complex process $(\cG_\tim)$
    is not \formalconnected, i.e.
    \begin{equation*}
      H^0(\cG_\tim;R)\not= R \quad \mbox{ or } \quad H^i(\cG_\tim;R)\not= 0
      \mbox{ for some }i \in [j].
    \end{equation*}
  \item\label{main:supercritical}
    For all $\tim \in [\tim_j^*,\tim_{\max}]$, the random $d$-complex process
    $(\cG_\tim)$
    is \formalconnected, i.e.
    \begin{equation*}
      H^0(\cG_\tim;R) = R \quad \mbox{ and } \quad H^i(\cG_\tim;R) = 0
      \mbox{ for all }i \in [j].
    \end{equation*}
  \end{enumerate}
\end{thm}

Observe that in Theorem~\ref{thm:main} we do not
consider \connectedness\ for the case $j=0$. 
Indeed, the condition 
$H^0(\cG_\tim;R)=R$  corresponds to topological connectedness of  
$\cG_\tim$, i.e.\ 
vertex-connectedness of the underlying (non-uniform)
random hypergraph, which has been extensively  
studied, and for which much stronger
results are known (see e.g.\  \cite{CooleyKangKoch16,Poole15}).
However, topological connectedness is
a necessary condition for the
\connectedness\ of $\cG_\tim$ 
(see~Definition~\ref{def:cohomconn}), 
therefore in order to make this paper 
self-contained, this case is treated separately 
in Lemma~\ref{lem:topconn}.

Furthermore, we observe that neither \connectedness\ nor the presence
of copies of $\Mjkhat$ are necessarily monotone properties (as we will see in  Example~\ref{ex:nonmono}),
which makes the proofs significantly harder.
Indeed, it is not immediately clear that \connectedness\ should
have a single threshold---in principle, the random $d$-complex process
$(\cG_\tim)$ could switch between
being \connected\ or not several times. However,
Theorem~\ref{thm:main} implies that with high probability this does not happen
and there is indeed a single threshold.

Our second main result gives an asymptotic
description of the $j$-th cohomology
group of $\cG_\tim$ for values of $\tim$ in the 
\emph{critical window}, i.e.\ $\tim = 1+O(1/\log n)$.

\begin{thm}[Rank in the critical window]\label{thm:criticalwindow}
Let $c\in \mathbb{R}$ be a constant and suppose 
that $(c_n)_{n\ge 1}$
is a sequence of real numbers with $c_n 
\xrightarrow{n \to \infty} c$. 
Let $j \in [d-1]$, let $\tim=1+\frac{c_n}{\log n}$, and consider $\allp=\tim \bar\allp$ for a $j$-critical
direction $\bar\allp$. 
Then there exists
a constant $\critm = \critm(c,\bar\allp)$ such that
with high probability
\[ H^j(\cG_\tim; R) \quad = \quad R^{Y}, \] 
where $Y$ is a Poisson random variable with mean $\critm$.
\end{thm}
\noindent The constant $\critm$ will be explicitly defined in~\eqref{eq:critm} in Section~\ref{sec:parametrisation}.

As a consequence of Theorems~\ref{thm:main} and~\ref{thm:criticalwindow}, 
we derive an explicit expression for the limiting probability of $\cG_\tim$
being \connected\ within the critical window.
\begin{cor}\label{cor:probcritwind}
Let $(c_n)_{n\ge 1}$, $c$, $j$, $\tim$, $\barallp$, $\allp$, and $\critm$ be given as 
in Theorem~\ref{thm:criticalwindow}. Then
\[ \Pr \left( \cG_\tim \text{ is \formalconnected} \right) \xrightarrow{n\to\infty} \exp(-\critm). \]
\end{cor}

\bigskip
\subsection{Proof techniques}
\label{sec:prooftechniques}

The three statements of the Hitting Time Theorem 
(Theorem~\ref{thm:main}) follow from auxiliary results presented
in Section~\ref{sec:proofmainthm},
which in turn are proved throughout the paper.

More precisely, we show in 
Lemma~\ref{lem:hittingtime_simplevers} that
the choice of a $j$-critical direction $\bar\allp$ (Definition~\ref{def:criticaldirection}) 
implies that the last minimal obstruction disappears
at around time $\tim = 1$, thus 
proving statement~\ref{main:hitting} of 
Theorem~\ref{thm:main}.

The main ingredient in the proof of
Theorem~\ref{thm:main}~\ref{main:subcritical} will be Lemma~\ref{lem:jinterval},  
which states that for every constant $\eps>0$,
whp $H^j(\cG_\tim;R) \neq
0$ for every $\tim$ in the interval
\[I_j(\eps) \ :=\  \left[\frac{\eps}{n}, \tim_j^*\right). \]
To prove this, in Section~\ref{sec:subcritical}
we split $I_j(\eps)$ into three subintervals and show that
whp in each of these there exists a copy of the obstruction $\Mjkhat$ for some $j\le k\le d$
(Lemmas~\ref{lem:firstinterval}, \ref{lem:secondinterval} and~\ref{lem:thirdinterval}), and thus $H^j(\cG_\tim;R) \neq 0$ (Corollary~\ref{cor:obstruction}).

In addition, we show that for every $i \in [j-1]$
there exists an appropriate scaling factor
$\tim$ such that the vector $\tim \barallp$ is
an $i$-critical direction (Lemma~\ref{lem:criticalscaling}).
Thus we can apply Lemma~\ref{lem:jinterval}
with $j$ replaced by $i$ and find intervals
$I_i(\eps)$ where whp 
$H^i(\cG_\tim;R) \neq 0$ (Corollary~\ref{cor:iintervals}).
We further define an interval $I_0(\eps)$
and show that
whp $\cG_\tim$ is not topologically 
connected
for every $\tim \in I_0(\eps)$ (Lemma~\ref{lem:topconn}).
In this way we can complete the proof of Theorem~\ref{thm:main}~\ref{main:subcritical}
by showing that we can 
choose $\eps$ such that $[0, \tim_j^*)
= \bigcup_{i=0}^j I_i(\eps)$ and thus
$\cG_\tim$ is not 
\connected\ throughout the subcritical case.

By definition of $\tim_j^*$, 
whp for any
$\tim \ge \tim_j^*$ there are no copies of
the minimal obstruction $\Mjkhat$ for any $k=j,\ldots,d$, thus in order to prove statement~\ref{main:supercritical} of 
Theorem~\ref{thm:main} we need to show 
that whp
no other `larger' obstructions to the vanishing
of $H^j(\cG_\tim;R)$ appear in the
complex.
This is given by
Lemma~\ref{lem:supercritical}, which we prove
in Section~\ref{sec:supercritical}.
We show that the
smallest support of any non-zero element of
the cohomology group
must be \emph{traversable} (Lemma~\ref{lem:traversability}), a very useful property
that allows us to define a search process,
with which
we can construct such a support. 
By bounding the number of ways this search process can
evolve, we also bound the
number of possible supports
and the probability that such 
a non-zero element of the cohomology
group exists
(Lemmas~\ref{lem:smallsupp} and~\ref{lem:largesupp}).

To prove the Rank Theorem (Theorem~\ref{thm:criticalwindow})
and Corollary~\ref{cor:probcritwind}, 
in Section~\ref{sec:proofofcritwind} we
will use the fact that for values of $\tim$ `close' to~$1$ whp the only obstructions to \connectedness\ are
copies of $\Mjkhat$ (Corollary~\ref{cor:Ntauemptycritwind}) and that indeed they
are a \emph{minimal} set of generators for 
$H^j(\cG_\tim;R)$. We conclude the proof of the Rank Theorem by showing that the number of such obstructions
converges in distribution to a Poisson random variable 
(Lemma~\ref{lem:Njkdistribution}). Finally we prove Corollary~\ref{cor:probcritwind} by applying
Theorem~\ref{thm:criticalwindow} to determine the probability that the $j$-th cohomology group vanishes
and Theorem~\ref{thm:main} to show that whp 
all lower cohomology groups vanish (except the zero-th, which is $R$).

\bigskip
\subsection{Outline of the paper} \label{sec:outlinepaper}
The paper is structured as follows.

In Section~\ref{sec:prelim}, we introduce some preliminary concepts
regarding the parametrisation
of a $j$-critical direction, as well as some standard concepts of cohomology theory.

Section~\ref{sec:proofmainthm} contains
the main auxiliary results that we combine to prove the Hitting Time Theorem (Theorem~\ref{thm:main}).

The proofs of the auxiliary results of 
Section~\ref{sec:proofmainthm} will follow
in Sections~\ref{sec:minimal}--\ref{sec:supercritical}.
In particular, the results of 
Section~\ref{sec:supercritical} will also lay
the foundation of the proof of the
Rank Theorem (Theorem~\ref{thm:criticalwindow}), 
which is presented in Section~\ref{sec:proofofcritwind},
together with the proof of Corollary~\ref{cor:probcritwind}.

In Section~\ref{sec:concrem}
we discuss our main results and present
some open problems.

In Appendix~\ref{sec:parameters}
we explain in more detail 
why with the choice of a $j$-critical
direction, Theorems~\ref{thm:main}
and~\ref{thm:criticalwindow} cover all
interesting cases.
Some standard but technical proofs are omitted 
from the main text,
but included in Appendices~\ref{sec:proofaux} and~\ref{sec:proofaux2} for completeness.
Finally, in Appendix~\ref{ap:glossary} we include a
glossary of some of the most important terminology
and notation used in the paper,
for easy reference.

\bigskip
\section{Preliminaries}\label{sec:prelim}

\subsection{Parametrisation}\label{sec:parametrisation}
In this section we will define the concept of $j$-critical 
directions, which appears in Theorems~\ref{thm:main} and~\ref{thm:criticalwindow}. 

Given a direction $\barallp = 
(\bar p_1, \ldots, \bar p_d)$, 
let $k\in [d]$ be an index such that $\bar p_k\neq 0$,
and let $K$ be a $(k+1)$-set with birth time $t_K$.
The \emph{scaled birth time} of $K$ is defined 
as
\begin{equation} \label{eq:scaledbirth}
\tim_K\ :=\ \frac{t_K}{\bar p_k}.
\end{equation}
(If $\bar p_k=0$ we view all $(k+1)$-sets as having infinite scaled birth time.) 
Thus $\tim_K$ is distributed uniformly in $[0,1/\bar p_k]$,
and $\cG_{\tim}$ consists of all those simplices with
\emph{scaled} birth time at most $\tim$, together with their downward-closure\footnote{With probability 1
no two simplices have the same scaled birth time, which is important for the process interpretation.}.

The motivation
for the following definitions will become apparent later (see Lemma~\ref{lem:expectedNjk} and
Appendix~\ref{sec:parameters}).

\begin{definition}\label{def:criticalp}
Given $j\in [d-1]$,
a vector $\barallp = (\bar p_1,\dotsc,\bar p_d)$\label{prob:barallp}
is called \emph{$j$-admissible} if for each $1\le k\le d$ there are real-valued
constants $\bar\alpha_k$,$\bar\gamma_k$, and a function $\bar \beta_k=\bar \beta_k(n)$\label{param:alphbetgamm}
such that \label{param:barpk}
  \begin{equation}\label{eq:parameters}
    \bar p_k = \frac{\bar\alpha_k\log n + \bar\beta_k}{n^{k-j+\bar\gamma_k}}(k-j)!,
  \end{equation}
and furthermore
  \renewcommand{\theenumi}{(A\arabic{enumi})}
  \begin{enumerate}
  \item\label{parameters:zero}
    at least one of $\bar\alpha_k,\bar\gamma_k$ is zero and neither of
    them is negative;
  \item\label{parameters:subpoly}
    if $\bar\alpha_k = 0$, then either $\bar\beta_k \equiv 0$ or $\bar\beta_k$ is positive and
    \emph{subpolynomial} in the sense that for every constant $\eps >
    0$, we have $\bar\beta_k = o(n^{\eps})$, but $\bar\beta_k =
    \omega(n^{-\eps})$;
  \item\label{parameters:sublog}
    if $\bar\gamma_k = 0$, then $|\bar\beta_k| = o(\log n)$; 
  \item\label{parameters:k0}
    there exists an index $j+1 \le k_0 \le d$ with $\bar\alpha_{k_0} > 0$.
  \end{enumerate}
    \renewcommand{\theenumi}{(\roman{enumi})}
\end{definition}  
The following observation follows immediately from the definition,
and will be used implicitly at many points in the paper.
\begin{remark}\label{rem:pkprob}
If $\barallp$ is $j$-admissible and $k\ge j+1$, then $\bar p_k =O\left(\frac{\log n}{n}\right)=o(1)$.
In particular, if $\allp = \tim \barallp$ for some $\tim=O(1)$, then $p_k = \tim \bar p_k \le 1$.
\end{remark}
This observation means that, for $k\ge j+1$ and for $\tim$ not too large, we have that $p_k=\tim \bar p_k$
is indeed a probability
term and we can use it in calculations without having to replace it by $1$. On the other
hand, for $k=j$ we often need to be slightly more careful.

Note that some of the properties in Definition~\ref{def:criticalp} can be guaranteed simply
by scaling $\barallp$ and choosing $\bar\alpha_k,\bar\gamma_k,\bar \beta_k$ appropriately,
but that some other properties place restrictions on the direction. However, we will see later
(Appendix~\ref{sec:parameters})
that it is reasonable to restrict attention to $j$-admissible vectors $\barallp$.
Indeed, by scaling appropriately we can even go further:
given a $j$-admissible vector $\bar\allp = (
\bar p_1, \ldots, \bar p_d)$,
for every index $k$ with $j \le k \le d$
and $\bar p_k \neq 0$ we define the parameters
  \begin{equation}\label{eq:lmn}
    \begin{split}
      \bar\logp_k &\ :=\  j+1-\bar\gamma_k-(k-j+1)\sum\nolimits_{i=j+1}^{d}\bar\alpha_i,\\
      \bar\sublogp_k &\ :=\  - (k-j+1)\sum\nolimits_{i=j+1}^{d}\frac{\bar\beta_i}{n^{\bar\gamma_i}} +
      \begin{cases}
        0 & \text{if } \bar p_k>1,\\
        \log\log n & \text{if } \bar p_k\le 1 \text{ and } \bar\alpha_k\not=0,\\
        \log(\bar\beta_k) & \text{if } \bar p_k\le 1 \text{ and } \bar\alpha_k=0,
      \end{cases}\\
      \bar\constp_k &\ :=\  
      \begin{cases}
        -\log((j+1)!) & \text{if }k=j,\\
        -\log(j!)-\log(k-j+1) +\log(\bar\alpha_k) & \text{if } k\not=j
        \text{ and } \bar\alpha_k\neq 0,\\
        -\log(j!)-\log(k-j+1) & \text{otherwise.}
      \end{cases}
    \end{split}
  \end{equation} \label{param:lambdmunu}
  Note that all $\bar\logp_k$, $\bar\constp_k$ are constants (since the $\bar \alpha_i$ are constants),
  while the $\bar\sublogp_k$ are
  functions of $n$ with $\bar\sublogp_k = o(\log n)$, by Definition~\ref{def:criticalp}.
  
\begin{definition} \label{def:criticaldirection}
  We say that a $j$-admissible vector $\barallp$ is a \emph{$j$-critical direction} if:
  \renewcommand{\theenumi}{(C\arabic{enumi})}
  \begin{enumerate}
  \item \label{crit:less}
  $\bar\logp_k\log n + \bar\sublogp_k + \bar\constp_k \le 0$, $\quad$ for all indices $k$ with $j \le k \le d$ and $\bar p_k \not= 0$;
  \item \label{crit:equal}
  $\bar\logp_{\bar k}\log n + \bar\sublogp_{\bar k} + \bar\constp_{\bar k} = 0$, $\quad$ for some $\bar k$ with $j \le \bar k \le d$.
  \end{enumerate} \label{param:bark}
    \renewcommand{\theenumi}{(\roman{enumi})}
\end{definition}
More generally, if we have a vector
$\allp=(p_1,\ldots,p_d)$ (where we will usually have $\allp=\tim\barallp$),
we would like to define parameters analogous
to those for $\barallp$.

\begin{definition}\label{def:generalparameters}
Given a vector $\allp=(p_1,\ldots,p_d)$,
for each $k \in [d]$, define
\begin{align*}
\alpha_k & \ :=\  \lim_{n\to\infty}\left(\frac{p_k n^{k-j}}{(k-j)!\log n}\right),\\
\gamma_k & \ :=\  \sup\{\gamma\in\mathbb{R} \mid p_k n^{k-j+\gamma} = o(1)\},\\
\beta_k & \ :=\  \frac{n^{k-j+\gamma_k} p_k}{(k-j)!}-\alpha_k \log n,
\end{align*}
if the limit and the supremum exist.

Furthermore, we define the parameters $\logp_k$, 
$\sublogp_k$, and $\constp_k$ analogously to  
\eqref{eq:lmn}, with $\bar\alpha_k$, $\bar\gamma_k$, and
$\bar\beta_k$ replaced by $\alpha_k$, $\gamma_k$, and
$\beta_k$, respectively.
\end{definition}

The following observation follows directly from the definition.

\begin{remark}\label{rem:scaledparameters:zero}
If $\barallp$ is a $j$-critical direction and $\allp=\tim\barallp$ for some $\tim=O(1)$,
then the analogue of~\ref{parameters:zero} also holds for $\allp$, i.e.\
for all $1\le k \le d$, at least one of $\alpha_k,\gamma_k$ is zero and neither of them is negative.
\end{remark}

In order to prove Theorem~\ref{thm:criticalwindow}, 
we will need to take a closer look at how the 
process behaves within the critical window, 
which is the range where whp the complex
$\cG_\tim$ switches from being not \connected\ 
to being \connected. 
More precisely, we consider
$\tim = 1 + O(1/\log n)$ (cf. Theorem~\ref{thm:criticalwindow}).
We also need the following concepts.
\begin{definition} \label{def:criticaldim}
Given a $j$-critical direction $\bar\allp=(\bar p_1,
\ldots , \bar p_d)$, an index $k$
with $j\le k\le d$ 
and $\bar p_k \neq 0$ is called
a \emph{critical dimension} if 
$\bar\logp_k \log n + \bar\sublogp_k + \bar\constp_k = 
O(1)$, i.e.\ $\bar\logp_k=0$ and
$\bar\sublogp_k = O(1)$ (recall that
$\bar\constp_k = O(1)$). 
We denote
by $\critset = \critset(\bar\allp,j)$ the set of all critical dimensions
for the $j$-critical direction $\bar\allp$.
\end{definition}

It will turn out (Lemma~\ref{lem:Njkdistribution}) that
for any $\tim = 1 + O(1/ \log n)$,
the critical dimensions are precisely those indices $k$
for which there is a positive asymptotic probability 
of having copies of a reduced version of $\Mjkhat$, 
called $\Njk$ (Definitions~\ref{def:Njk} and~\ref{def:Njj}), in $\cG_\tim$.
Furthermore, 
if we consider $\tim=1+\frac{c_n}{\log n}$ with
$c_n \xrightarrow{n \to \infty} c \in \mathbb{R}$,
then the constant $\critm$ which appeared in
Theorem~\ref{thm:criticalwindow} is precisely
\begin{equation} \label{eq:critm}
\critm \ :=\  \exp(-c(j+1)) \sum\nolimits_{k\in \critset}
\exp ( \bar\sublogp_k + \bar\constp_k + c \gamma_k ) ,
\end{equation}
as we will see in the proof of
Theorem~\ref{thm:criticalwindow}
(Section~\ref{sec:proofofcritwind}). We will also see that 
for any critical 
dimension $k$, the term
$\exp (\bar\sublogp_k + \bar\constp_k + c(\bar\gamma_k 
- j -1) )$ is
closely related to the number of copies of $\Njk$
(Corollary~\ref{cor:expectatcritwind}).

\bigskip
\subsection{Cohomology}\label{sec:cohom}

Let us review the standard notions of cohomology groups of a
$d$-dimensional simplicial complex $\cG$.

Let $j \in [d]_0$. To define cohomology groups, one considers
\emph{ordered $j$-simplices}, that is, $j$-simplices with an ordering
of their vertices.\footnote{When we consider simplices
  \emph{without} an ordering, we will often simply refer to them as
  `simplices' instead of `unordered simplices'.}
We adopt the notation $[v_0,\dotsc,v_j]$ for a $j$-simplex whose
vertices are ordered $v_0,\dotsc,v_j$. If $\sigma = [v_0,\dotsc,v_j]$
is an ordered $j$-simplex and $i\in[j]_0$, then
$[v_0,\dotsc,\hat{v}_i,\dotsc,v_j]$ denotes the ordered $(j-1)$-simplex
obtained from $\sigma$ by removing $v_i$ (and preserving the order on the remaining vertices).

Recall that we will be considering cohomology groups over an arbitrary (non-trivial) abelian group~$R$.
A function $f$ from the set of ordered $j$-simplices in $\cG$ to $R$
is called a \emph{$j$-cochain}\label{cohom:cochain} if $f(\sigma) = -f(\sigma')$ whenever
$\sigma'$ is obtained from $\sigma$ by exchanging the positions of two
vertices in the ordering of the simplex. For a $j$-cochain $f$, we
define its \emph{support} $\mathrm{supp}(f)$ to be the set of
\emph{unordered} simplices $\sigma$ such that $f$ maps some (and thus
every) ordering of $\sigma$ to a non-zero value.

The set $C^j(\cG;R)$\label{cohom:Cj} of
$j$-cochains in $\cG$ forms a group with respect to  pointwise summation, defined by $(f_1+f_2)(\sigma) \ :=\  f_1(\sigma) + f_2(\sigma)$.
For $j\in[d-1]_0$, we define the \emph{coboundary operator}\label{cohom:coboperator}
$\delta^j\colon C^j(\cG;R) \to C^{j+1}(\cG;R)$ by
\begin{equation*}
  (\delta^j f)([v_0,\dotsc,v_{j+1}]) \ :=\ 
  \sum\nolimits_{i=0}^{j+1}(-1)^i f([v_0,\dotsc,\hat{v}_i,\dotsc,v_{j+1}]).
\end{equation*}
Clearly, $\delta^j$ is a group homomorphism. Furthermore, let
$\delta^{-1}$ and $\delta^d$ denote the unique group homomorphisms
$\delta^{-1}\colon \{0\}\to C^0(\cG;R)$ and $\delta^{d}\colon
C^d(\cG;R)\to \{0\}$. For each $j\in[d]_0$, the $j$-cochains in
$\ker\delta^j$ and in $\im\delta^{j-1}$ are called \emph{$j$-cocycles}\label{cohom:cocycle}
and \emph{$j$-coboundaries}\label{cohom:coboundary}, respectively. A straightforward
calculation shows that every $j$-coboundary is also a $j$-cocycle,
i.e.\ $\im \delta^{j-1} \subseteq \ker \delta^j$.
Thus, we can define the \emph{$j$-th cohomology group of $\cG$ with
coefficients in $R$} as the quotient group\label{cohom:cohomgroup}
\begin{equation}\label{eq:cohomdef}
  H^j(\cG;R) \ :=\  \ker\delta^j / \im\delta^{j-1}.
\end{equation}

\smallskip
\subsection{Non-vanishing of cohomology groups}
\label{sec:non-vanish}

In view of Theorems~\ref{thm:main} 
and~\ref{thm:criticalwindow}, we are particularly interested in
when $H^j(\cG;R)$ vanishes for $j\in[d-1]$, which happens if and only
if every $j$-cocycle is also a $j$-coboundary. 
Hence, we need a criterion for a
$j$-cocycle (or more generally a $j$-cochain) 
\emph{not} to be a $j$-coboundary, which will be provided
by Lemma~\ref{lem:shellobstruction}.
To this end, we need
the following definition.

\begin{definition} \label{def:shell}
  For any $(j+2)$-set $A$ in a complex $\cG$, the collection of all
  $(j+1)$-sets of $A$ is called a \emph{j-shell}\label{comb:shell} if each of them
  forms a $j$-simplex in $\cG$.
\end{definition}
If the collection of all $(j+1)$-subsets of a $(j+2)$-set $A$ forms a
$j$-shell, with a slight abuse of terminology we also refer to the set
$A$ itself as a $j$-shell.

\begin{lem}\label{lem:shellobstruction}
  Let $j\in[d-1]$, let $f$ be a $j$-cochain in a $d$-dimensional
  complex $\cG$ on $[n]$ and suppose that there exists  $A \in \binom{[n]}{j+2}$ such that
  \begin{enumerate}
  \item\label{shell:sides}
    $A$ is a $j$-shell in $\cG$ and
  \item\label{shell:support}
    precisely one $(j+1)$-set of $A$ lies in the support of $f$.
  \end{enumerate}
  Then $f$ is not a $j$-coboundary in $\cG$.
\end{lem}

\begin{proof}
  Let $\cG' \ :=\  \cG\cup \{A\}$ and observe that this is a simplicial
  complex, because all proper non-empty subsets of $A$ were already
  simplices in $\cG$ by condition~\ref{shell:sides}. Denote the
  vertices in $A$ by $v_0,\dotsc,v_{j+1}$ such that
  $\{v_1,\dotsc,v_{j+1}\}\in \mathrm{supp}(f)$. By \ref{shell:support}, this means that
  \begin{equation*}
    (\delta^j f)([v_0,\dotsc,v_{j+1}]) = 
    \sum\nolimits_{i=0}^{j+1}(-1)^if([v_0,\dotsc,\hat{v}_i,\dotsc,v_{j+1}]) = f([v_1,\dotsc,v_{j+1}]) \not= 0_R.
  \end{equation*}
This implies that while $f$ may be a $j$-cocycle in $\cG$, it is certainly not a $j$-cocycle in $\cG'$.
Thus in particular $f$ is not a $j$-coboundary in $\cG'$.
Since $\cG$ and $\cG'$ have the same
  sets of $j$-simplices and of $(j-1)$-simplices, this means that $f$
  is also not a $j$-coboundary in $\cG$.
\end{proof}

\section{Hitting Time Theorem: proof of Theorem~\ref{thm:main}}
\label{sec:proofmainthm}
In this section,
we provide an outline of the most important auxiliary results
of the paper and show how together they prove the Hitting Time Theorem
(Theorem~\ref{thm:main}).
These auxiliary results are proved throughout the rest of the paper.

\smallskip
\subsection{Hitting time and subcritical case}
\label{sec:mainthm:hitandsub}
To prove Theorem~\ref{thm:main}~\ref{main:hitting}, 
recall that $\tim_j^*$ is the birth time
of the simplex whose appearance causes the last
copy of $\Mjkhat$ for any 
$j\le k \le d$ to disappear.
We want to show that 
this happens at around time $\tim=1$.
More precisely, we will prove the following.
\begin{lem}\label{lem:hittingtime_simplevers}
Let $\omega$ be a function of $n$ that tends to infinity as $n \to \infty$. If $\barallp$ is a $j$-critical direction, then
whp
\[ 1- \frac{\omega}{\log n} < \tim_j^* < 1 + \frac{\omega}{\log n}.\] 
\end{lem}
Statement \ref{main:hitting} of
Theorem~\ref{thm:main} will follow directly from 
Lemma~\ref{lem:hittingtime_simplevers}, which is proved in Section~\ref{sec:hitting}.
Indeed, we will prove a slightly stronger result (Lemma~\ref{lem:hitting}).

For the subcritical case (i.e.\ statement~\ref{main:subcritical}) of Theorem~\ref{thm:main}, we first determine
the threshold for topological connectedness of $\cG(n,\allp)$, i.e.\ for when $H^0(\cG(n,\allp);R) = R$. 

\begin{lem}\label{lem:topconn}
  There exist positive constants $c^- = c^-(d)$ and $c^+ = c^+(d)$ such that
  \begin{enumerate}
  \item \label{topconn:no}
  whp $\cG(n,\allp)$ is not topologically connected if $p_k \le
    \frac{c^-\log n}{n^k}$ for all $k \in [d]$;
  \item \label{topconn:yes}
  whp $\cG(n,\allp)$ is topologically connected if $p_k \ge
    \frac{c^+\log n}{n^k}$ for some $k \in [d]$.
  \end{enumerate}
\end{lem}
\noindent The proof of Lemma~\ref{lem:topconn}~\ref{topconn:no} consists of an easy application of the second moment method
to show that whp $\cG(n,\allp)$ contains isolated vertices, while Lemma~\ref{lem:topconn}~\ref{topconn:yes} follows from \cite[Lemma~4.1]{CooleyDelGiudiceKangSpruessel20}. For completeness, we 
include the proof of both parts of  Lemma~\ref{lem:topconn} in 
Appendix~\ref{sec:prooftopconn}.

\begin{remark}
In fact, with a slightly more careful extension of the argument,
one could strengthen Lemma~\ref{lem:topconn} to give the exact 
threshold. More precisely, if $p_k = \frac{c_k\log n}{n^k}$ for 
$k\in [d]$,
where each $c_k$ may now be a function in $n$, then $\cG(n,\allp)$
contains isolated vertices whp provided
$\sum_{k=1}^d \frac{c_k}{k!} = 1-\omega(\frac{1}{\log n})$,
whereas $\cG(n,\allp)$ is topologically connected whp if
$\sum_{k=1}^d \frac{c_k}{k!} = 1+\omega(\frac{1}{\log n})$.
\end{remark}

In particular, Lemma~\ref{lem:topconn} will imply that for every
sufficiently small $\eps>0$, whp the process $(\cG_\tim)$ is not
topologically connected, and thus also not  \connected, for every $\tim\in [0,\frac{\eps}{n^j}]$.

In order to cover the whole interval 
$[0,\tim_j^*)$, the following result, whose
proof is in Section~\ref{sec:subcritical}, will be  key.

\begin{lem}\label{lem:jinterval}
Let $\eps>0$ be a constant and define
\[ I_j (\eps): = \left[ \frac{\eps}{n}, \tim_j^*\right).\]
Then, whp $H^j(\cG_\tim;R)
\neq 0$ for every $\tim \in I_j(\eps)$.
\end{lem}
\noindent
To prove Lemma~\ref{lem:jinterval} we will show that whp,
for any $\tim \in I_j(\eps)$
there is an index $k$ with
$j\le k \le d$ such that
a copy of the minimal obstruction
$\Mjkhat$ exists in $\cG_\tim$. In fact,
we will show that whp just three minimal obstructions
exist in ranges which together
cover $I_j(\eps)$
(Lemmas~\ref{lem:firstinterval}, \ref{lem:secondinterval}, and 
\ref{lem:thirdinterval}).

For the remaining range of the subcritical
interval $[0,\tim_j^*)$, we want to 
consider the cohomology
groups $H^i(\cG_\tim;R)$ with
$i\in[j-1]$ and
determine in which subintervals they
do not vanish, i.e.\ we want
to find an analogue 
of Lemma~\ref{lem:jinterval}
for $H^i(\cG_\tim;R)$.
To do this, we need to show that starting
from a $j$-critical direction $\bar\allp$
we can use an appropriate rescaling
to obtain an $i$-critical direction.

\begin{lem}\label{lem:criticalscaling}
If $\bar\allp$ is a $j$-critical direction,
then for each $i\in[j-1]$ there
exist a constant $\eta=\eta_i>0$ and
a function $\epsilon=\epsilon_i(n)=o(1)$
such that the vector $\frac{\eta + \epsilon}{n^{j-i}} \barallp$ 
is an $i$-critical direction.
\end{lem}
\noindent
Although Lemma~\ref{lem:criticalscaling} is 
intuitively obvious, its proof is rather
technical. 
We therefore delay the proof 
until
Appendix~\ref{sec:proofofcriticalscaling}. 

Using Lemma~\ref{lem:criticalscaling},
for general 
$i \in [j]$ we can consider the hitting time $\tim_i^*$
for the disappearance of the last
minimal obstruction $\hat{M}_{i,k}$.
More precisely, for the $j$-critical
direction $\barallp$ as in Theorem~\ref{thm:main},
consider the process  $(\cG_\tim)=(\cG(n,\tim \barallp))_{\tim \in [0,\tim_{\max}]}$
and for each $i\in[j]$ let
\begin{equation}\label{eq:timistar}
\tim_i^* \ :=\ 
\sup\{ \tim \in \mathbb{R}_{\ge0} \mid \cG_\tim \text{ contains a copy of } \hat{M}_{i,k}
\text{ for some } k \text{ with }i \le k \le d \}.
\end{equation}
Observe that for $i=j$, this matches
the definition of $\tim_j^*$ in 
Theorem~\ref{thm:main}.
We derive the following result
from Lemmas~\ref{lem:hittingtime_simplevers},
\ref{lem:jinterval}, and~\ref{lem:criticalscaling}.
\begin{cor} \label{cor:iintervals}
Let $\eps>0$ be a constant and $i \in [j]$. Define
\[ I_{i}(\eps) \ :=\  \left[ \frac{\eps}{n^{j-i+1}}, \tim_i^*\right).\]
Then,  whp 
\begin{enumerate}
\item $H^i(\cG_\tim;R)
\neq 0$ for every $\tim \in I_i(\eps)$; 
\label{iintervals:cohomgroup}
\item there exists a positive constant $\eta=\eta_i$ such that
$\tim_i^* = \frac{\eta + o(1)}{n^{j-i}}$.
\label{iintervals:timi*}
\end{enumerate}
\end{cor}
\noindent Indeed, Lemma~\ref{lem:hittingtime_simplevers} implies
that $\eta=\eta_j = 1$ for $i=j$ in statement~\ref{iintervals:timi*}
of Corollary~\ref{cor:iintervals}.

\begin{proof}
For any $i\in[j-1]$,
by Lemma~\ref{lem:criticalscaling} we can 
appropriately scale $\barallp$ to obtain
an $i$-critical direction 
$\frac{\eta+\epsilon}{n^{j-i}} \barallp$.
Thus, by 
Lemmas~\ref{lem:jinterval} and~\ref{lem:hittingtime_simplevers}
applied with $j$ replaced by $i$, we obtain
\ref{iintervals:cohomgroup} and 
\ref{iintervals:timi*}, respectively.

For $i=j$, Lemmas~\ref{lem:jinterval} and~\ref{lem:hittingtime_simplevers}
apply directly to $\barallp$.
\end{proof}

\medskip
\subsection{Supercritical case}
\label{sec:mainthm:sup}
In Theorem~\ref{thm:main}~\ref{main:supercritical} we consider
$\tim \in [\tim_j^*, \tim_{\max}]$.
By definition of $\tim_j^*$, we know that
whp in this range
there is no copy of the minimal obstruction
$\Mjkhat$ to \connectedness\
for any $j\le k \le d$, but
we also have to  
exclude other types of obstructions.
Indeed, in Section~\ref{sec:supercritical}
we prove the following.
\begin{lem} \label{lem:supercritical}
Whp 
for every $\tim \in [\tim_j^*,\tim_{\max}]$,
we have $H^j(\cG_\tim;R) = 0$.
\end{lem}
Observe that since the choice of $j$ was arbitrary, Lemma~\ref{lem:supercritical} 
also holds when $j$ is replaced 
by any $i\in [j-1]$. 

\bigskip
\subsection{Proof of Theorem~\ref{thm:main}}
We now apply the auxiliary results
of Sections~\ref{sec:mainthm:hitandsub} 
and~\ref{sec:mainthm:sup} to
prove the Hitting Time Theorem
(Theorem~\ref{thm:main}).

\smallskip
\noindent
\ref{main:hitting}  
Fix a function $\omega$ of $n$ which tends to
infinity as $n \to \infty$. To show that
whp $\tim_j^* = 1 +o ( \omega / \log n)$, it
suffices to apply Lemma~\ref{lem:hittingtime_simplevers} with any
function $\omega'$ which tends to infinity
but satisfies $\omega' = o(\omega)$,
e.g.\ picking $\omega' = \sqrt{\omega}$
will suffice. 

\smallskip
\noindent
\ref{main:subcritical}
For $\eps>0$, define 
\[I_0(\eps)\ :=\ \left[0, \frac{\eps}{n^j} \right].  \]
By definition of $j$-admissibility 
(Definition~\ref{def:criticalp}),
for every $k \in [d]$ we have
$\bar p_k = O( \log n / n^{k-j})$.
Thus, by Lemma~\ref{lem:topconn}
we can choose $\eps$ small enough such that 
whp 
\begin{equation}\label{eq:H0not}
\text{for every } 
\tim \in  I_0(\eps),
\qquad H^0(\cG_\tim;R) \neq R .
\end{equation}

Now consider the intervals $I_i(\eps)$.
By Corollary~\ref{cor:iintervals}~\ref{iintervals:timi*}, we can choose
$\eps$ small enough (namely
$\eps< \eta_i$ for every $i \in [j-1]$) 
such that whp for each such $i$
\[ I_i(\eps) \cap I_{i+1}(\eps) =
\left[\frac{\eps}{n^{j-i}}, \tim_i^* \right) 
\neq \emptyset, \]
and thus
\begin{equation} \label{eq:intervals}
\bigcup\nolimits_{i=0}^j I_i(\eps) = [0, \tim_j^*).
\end{equation}

By Corollary~\ref{cor:iintervals}~\ref{iintervals:cohomgroup}, for any
 $\eps>0$ whp
$H^i(\cG_\tim;R) \neq 0$ for every $\tim \in
I_i(\eps)$. Thus, by choosing $\eps$ such that
conditions
\eqref{eq:H0not} and \eqref{eq:intervals}
hold simultaneously, whp
the process $(\cG_\tim)$ is not
\connected\ for all $\tim \in [0,\tim_j^*)$, as required.

\smallskip
\noindent
\ref{main:supercritical}
Recalling that $\bar p_k = O(\log n/ n^{k-j})$
for any $k\in [d]$ by Definition~\ref{def:criticalp},
Lemma~\ref{lem:topconn} implies that we can find
a positive constant $\vartheta$ such that
whp $H^0(\cG_\tim;R) = R$ 
for every $\tim \in [\frac{\vartheta}{n^j}
,\tim_{\max}]$, which whp contains 
the interval $[\tim_j^*,\tim_{\max}]$
by Lemma~\ref{lem:hittingtime_simplevers}.

Furthermore, by Lemma~\ref{lem:supercritical}
applied for every $i\in [j]$, whp $H^i(\cG_\tim;R) = 0$ for 
every $\tim \in [\tim_i^*,\tim_{\max}]$,
which contains $[\tim_j^*,\tim_{\max}]$
by Corollary~\ref{cor:iintervals}~\ref{iintervals:timi*}. Thus,
whp the process $(\cG_\tim)$ is 
\connected\ for every 
$\tim \in [\tim_j^*,\tim_{\max}]$, as required.
\qed

\bigskip
\section{Minimal obstructions}\label{sec:minimal}

In this section we define copies of 
$\Mjkhat$ (Definitions~\ref{def:Mjkhat} 
and~\ref{def:Mjjhat}) and we explain why these objects 
can be interpreted as \emph{minimal obstrucions}
to \connectedness.

For the rest of the paper, let $j\in[d-1]$ be fixed. We first introduce the following necessary
concepts.

\begin{definition}\label{def:flower}
  Let $k$ be an integer with 
  $j \le k \le d$.
  Given a $k$-simplex $K$ in a $d$-dimensional simplicial complex
  $\cG$, we say that a collection $\mathcal{F} = \{P_0, \ldots,
  P_{k-j}\}$ of $j$-simplices forms a \emph{$j$-flower in $K$}\label{comb:flower} (see
  Figure~\ref{fig:flower}) if
  \renewcommand{\theenumi}{(F\arabic{enumi})}
  \begin{enumerate}
  \item\label{flower:union}
    $K = \bigcup_{i=0}^{k-j} P_i\,$;
  \item\label{flower:centre}
   there exists  $C$ with $|C|=j$ that is contained in
   $P_i$ for every $ i \in [k-j]_0$.
  \end{enumerate}
  \renewcommand{\theenumi}{(\roman{enumi})}
  We call the $j$-simplices $P_i$ the \emph{petals} and the set $C$
  the \emph{centre} of the $j$-flower $\mathcal{F}$.
 When $j$ is clear from the context, we often refer to the $j$-flower $\mathcal{F}$ simply as a \emph{flower}.
\end{definition}
\begin{figure}[htbp] 
    \centering
    \begin{tikzpicture}[scale=0.72]
      \draw[thick] [opacity=0.3, fill=gray] (-7.5,1.5) -- (-6,1.1);
      \draw[thick] [opacity=0.3, fill=gray] (-7.5,1.5) -- (-9.3,1);
      \draw[thick] [opacity=0.3, fill=gray] (-7.5,1.5) -- (-6.3,3.5);
              
      \draw[fill] (-7.5,1.5) circle [radius=0.1] node [above left]    
        {$c_1$};
      \draw[fill] (-6,1.1)  circle [radius=0.04] node [right]
        {$w_2$};
      \draw[fill] (-6.3,3.5) circle [radius=0.04]  node [below right]
        {$w_1$}; 
      \draw[fill] (-9.3,1) circle [radius=0.04] node [above]
        {$w_0$};

      \draw[thick,opacity=0.4] (-9.8,1.9) to [out=80,in=180] (-7.4,4);
      \draw[thick,opacity=0.4] (-7.4,4) to [out=0,in=100] (-4.65,1.9);
      \draw[thick,opacity=0.4] (-4.65,1.9) to [out=280,in=10]
        (-5.65,0.6);
      \draw[thick,opacity=0.4] (-5.65,0.6) to [out=200,in=0]
        (-7.6,0.7);
      \draw[thick,opacity=0.4] (-7.6,0.7) to [out=180,in=290]
        (-9.8,0.8);
      \draw[thick,opacity=0.4] (-9.8,0.8) to [out=110,in=260]
        (-9.8,1.9);
   
      \node at (-5.4,3.65) {$K$};

      \node at (-7.25,-0.3) {(a)};

      \filldraw [opacity=0.5, fill=gray] (-1.5,1.5) -- (0,1.1) --
        (-3.3,1) -- cycle; 
      \filldraw [opacity=0.5, fill=gray] (-1.5,1.5) -- (0,1.1) --
        (-0.3,3.5) -- cycle; 

      \draw[ultra thick] (-1.5,1.5) -- (0,1.1);  
                
      \draw[fill] (-1.5,1.5) circle [radius=0.1] node [above left]
        {$c_1$};
      \draw[fill] (0,1.1)  circle [radius=0.1] node [right]
        {$c_2$};
      \draw[fill] (-0.3,3.5) circle [radius=0.04]  node [below right]
        {$w_1$}; 
      \draw[fill] (-3.3,1) circle [radius=0.04] node [above]
        {$w_0$}; 
       
      \draw[thick,opacity=0.4] (-3.8,1.9) to [out=80,in=180] (-1.4,4);
      \draw[thick,opacity=0.4] (-1.4,4) to [out=0,in=100] (1.35,1.9);
      \draw[thick,opacity=0.4] (1.35,1.9) to [out=280,in=10]
        (0.35,0.6);
      \draw[thick,opacity=0.4] (0.35,0.6) to [out=200,in=0]
        (-1.6,0.7);
      \draw[thick,opacity=0.4] (-1.6,0.7) to [out=180,in=290]
        (-3.8,0.8);
      \draw[thick,opacity=0.4] (-3.8,0.8) to [out=110,in=260]
        (-3.8,1.9);
   
      \node at (0.6,3.65) {$K$};

      \node at (-1.25,-0.3) {(b)};

      \filldraw [opacity=0.5, fill=gray] (5.7,3.5) -- (6,1.1) --
        (2.7,1) -- cycle;
      \filldraw [opacity=0.5, fill=black] (4.5,1.5) -- (6,1.1) -- (5.7,3.5) --
        cycle; 
      \draw[very thick] (4.5,1.5) -- (6,1.1) -- (5.7,3.5) -- cycle; 
                
      \draw[fill] (4.5,1.5) circle [radius=0.1] node [above left]
        {$c_1$};
      \draw[fill] (6,1.1)  circle [radius=0.1] node [right]
        {$c_2$};
      \draw[fill] (5.7,3.5) circle [radius=0.1]  node [below right]
        {$c_3$}; 
      \draw[fill] (2.7,1) circle [radius=0.04] node [above] {$w_0$};

      \draw[dashed,opacity=0.5] (4.5,1.5) -- (2.7,1);     
       
      \draw[thick,opacity=0.4] (2.2,1.9) to [out=80,in=180] (4.6,4);
      \draw[thick,opacity=0.4] (4.6,4) to [out=0,in=100] (7.35,1.9);
      \draw[thick,opacity=0.4] (7.35,1.9) to [out=280,in=10]
        (6.35,0.6);
      \draw[thick,opacity=0.4] (6.35,0.6) to [out=200,in=0] (4.4,0.7);
      \draw[thick,opacity=0.4] (4.4,0.7) to [out=180,in=290]
        (2.2,0.8);
      \draw[thick,opacity=0.4] (2.2,0.8) to [out=110,in=260]
        (2.2,1.9);
   
      \node at (6.6,3.65) {$K$};
    
      \node at (4.75,-0.3) {(c)};    
    \end{tikzpicture}
    \caption{Examples of $j$-flowers in a $k$-simplex $K$, for $k=3$
      and $j=1,2,3$.\newline
      (a) The $1$-flower in $K$ with centre $C=\{c_1\}$ (bold black)
      and petals $P_i = C \cup \{w_i\}$, $i=0,1,2$ (grey).\newline
      (b) The $2$-flower in $K$ with centre $C=\{c_1,c_2\}$ (bold
      black) and petals $P_i = C \cup \{w_i\}$, $i=0,1$
      (grey).\newline
      (c) The $3$-flower in $K$ with centre $C=\{c_1,c_2,c_3\}$
      (bold black) and the unique petal $P_0 = C \cup \{w_0\} = K$ (grey).}
    \label{fig:flower}
  \end{figure}
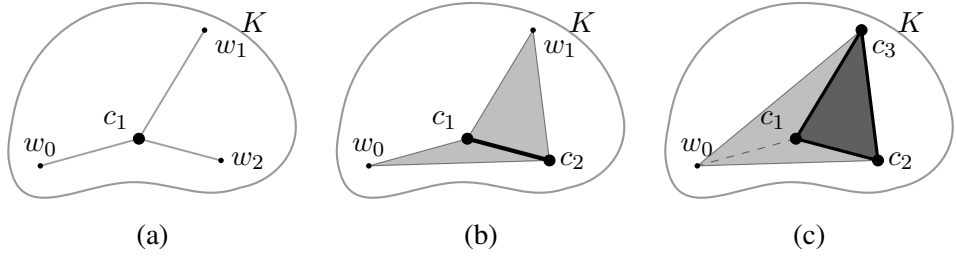

  Observe that for each $k$-simplex $K$ and each $(j-1)$-simplex $C
  \subset K$, there is a unique $j$-flower in $K$ with centre $C$,
  namely
  \begin{equation}\label{eq:flowerinKwithcentreC}
    \cF(K,C) \ :=\  \{C \cup \{w\} \mid w \in K \setminus C\}.
  \end{equation}
Note that if $k=j$, then any choice of a centre
$C\subset K$ produces the same flower
$\cF(K,C) = \{K\}$.

\begin{definition}\label{def:Mjkhat}
 Let $k$ be an integer with $j+1 \le k \le d$. 
 We say that a
  $4$-tuple $(K,C,w,a)$
  forms
  a \emph{copy of $\Mjkhat$}\label{comb:Mjkhat}
  (see Figure~\ref{fig:Mjkhat}) 
  in a simplicial complex $\cG$ if
  \renewcommand{\theenumi}{(M\arabic{enumi})}
  \begin{enumerate}    
  \item\label{Mjkhat:simplex}
    $K$ is a $k$-simplex in $\cG$;
  \item\label{Mjkhat:flower}
    $C$ is a $(j-1)$-simplex in $K$ such that every simplex of
    $\cG$ that contains a petal of the flower $\mathcal{F} =
    \mathcal{F}(K,C)$ is itself contained in $K$;
  \item\label{Mjkhat:shell}
    $w \in K\setminus C$ and $a\in[n]\setminus K$ are such that
    $C \cup \{w\} \cup \{a\}$ is a $j$-shell in $\cG$.
  \end{enumerate}
  \renewcommand{\theenumi}{(\roman{enumi})}
  We call the $j$-simplex $C \cup \{w\}$ the \emph{base} and $a$ the
  \emph{apex vertex} of the $j$-shell $C\cup \{w\} \cup \{a\}$. Every
  other $j$-simplex in the $j$-shell $C\cup\{w\}\cup\{a\}$ is called a \emph{side}
  of the $j$-shell.
\end{definition}

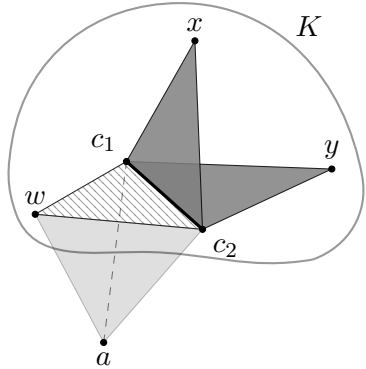
\begin{figure}[htbp]
\centering
  \begin{tikzpicture}[scale=1.0]
    \filldraw [opacity=0.85,fill=gray] (1.8,2.7) -- (4.5,2.6) --
      (2.8,1.8) -- cycle;
    \filldraw [opacity=0.85,fill=gray] (1.8,2.7) -- (2.7,4.3) --
      (2.8,1.8) -- cycle;
    \draw[fill,opacity=0.85,pattern=north west lines,pattern color=gray]
      (1.8,2.7) -- (0.6,2) -- (2.8,1.8) -- cycle;

    \filldraw [opacity=0.25,fill=gray] (0.6,2) -- (1.5,0.3) --
      (2.8,1.8) -- cycle;
    \draw[dashed,opacity=0.5] (1.5,0.3)  -- (1.8,2.7);

    \draw[very thick] (1.8,2.7) -- (2.8,1.8);
    \draw[fill] (1.8,2.7) circle [radius=0.04] node [above left]
      {$c_1$};
    \draw[fill] (2.8,1.8) circle [radius=0.04] node [below right]
      {$c_2$};

    \draw[fill] (4.5,2.6) circle [radius=0.04]
    node [above] {$y$};
    \draw[fill] (2.7,4.3) circle [radius=0.04]
    node [above] {$x$};
    \draw[fill] (0.6,2) circle [radius=0.04]
     node [above] {$w$};
    \draw[fill] (1.5,0.3) circle [radius=0.04]  node [below] {$a$};

    \draw[thick,opacity=0.4] (0.3,3.0) to [out=80,in=180] (2.5,4.8);
    \draw[thick,opacity=0.4] (2.5,4.8) to [out=0,in=100] (4.9,2.5);
    \draw[thick,opacity=0.4] (4.9,2.5) to [out=280,in=20] (4.25,1.4);
    \draw[thick,opacity=0.4] (4.25,1.4) to [out=190,in=0] (2.0,1.5);
    \draw[thick,opacity=0.4] (2.0,1.5) to [out=180,in=290] (0.35,2.0);
    \draw[thick,opacity=0.4] (0.35,2.0) to [out=110,in=260] (0.3,3.0);
    \node at (4.2,4.5) {$K$};
  \end{tikzpicture}
  \caption{A copy of $\Mjkhat$, for $k=4$ 
  and $j=2$.\newline
  (a) The $k$-simplex $K$ contains 
  the flower $\cF (K,C)$
    with centre $C=\{c_1,c_2\}$, whose
    petals $C \cup \{w\}$, $C \cup \{x\}$, and
    $C \cup \{y\}$  are not present 
    in any simplex 
    which is not contained in $K$.\newline
   (b) The $(j+2)$-set $C\cup\{w\}\cup\{a\}$ 
   is a $j$-shell with apex vertex $a 
    \not\in K$ and whose
    base is the petal
    $C\cup \{w\}$.}
  \label{fig:Mjkhat}
\end{figure}
We aim to give an analogous definition for the case 
$k=j$ and it will be convenient to use unified terminology.
However, as observed before, if $K$ is a $(j+1)$-set, then a
$j$-flower $\cF(K,C)$ is always equal to $K$ itself, 
independently of the choice of the centre $C$ in $K$.
In particular, in this case condition~\ref{Mjkhat:flower} 
simply says that $K$ is an \emph{isolated $j$-simplex}
in $\cG$, i.e.\ a $j$-simplex that is not contained 
in any other simplex of $\cG$.
This means that given $K$,
the sets that would be required to be 
simplices or not in $\cG$ do not change
for different choices of the centre $C$, and 
therefore we do not want to consider
two copies of $\Mjjhat$ to be distinct if they
share the same $j$-simplex but have
different centres.
For this reason, to define $\Mjjhat$ we will use 
the following `canonical' choice for the centre.
\begin{definition}\label{def:Mjjhat}
We say that a $4$-tuple $(K,C,w,a)$ forms a 
\emph{copy of $\Mjjhat$} in $\cG$ if
  \begin{itemize}    
  \item
    $K$ is an isolated $j$-simplex in $\cG$;
  \item
    $a\in[n]\setminus K$ is such that
    $K \cup \{a\}$ is a $j$-shell in $\cG$;
  \item
    $C$ consists of the first $j$ vertices of $K$
 in the increasing order on $[n]$, 
    and $w$ is the last vertex of $K$ in this order.
  \end{itemize}
The notions of \emph{base}, \emph{apex vertex}, and 
  \emph{side} are analogous to Definition~\ref{def:Mjkhat}.
\end{definition}
It is easy to see that a copy of $\Mjjhat$ in 
Definition~\ref{def:Mjjhat} satisfies conditions~\ref{Mjkhat:simplex}--\ref{Mjkhat:shell}
of Definition~\ref{def:Mjkhat}.

Let us now define a `reduced' version of $\Mjkhat$, denoted by $\Mjk$,
by omitting the condition~\ref{Mjkhat:shell} on the 
$j$-shell $C\cup\{w\}\cup\{a\}$ in Definitions~\ref{def:Mjkhat} and~\ref{def:Mjjhat}.

\begin{definition}\label{def:Njk}
  Let $k$ be an integer with $j+1 \le k \le d$. A pair $(K,C)$ is
  called 
  a \emph{copy of $\Njk$}\label{comb:Mjk} if it satisfies the following
  conditions.
  \renewcommand{\theenumi}{(M\arabic{enumi})}
  \begin{enumerate}
  \item\label{Njk:simplex}
    $K$ is a $k$-simplex in $\cG$;
  \item\label{Njk:flower}
    $C$ is a $(j-1)$-simplex in $K$ such that every simplex of
    $\cG$ that contains a petal of the flower $\mathcal{F} =
    \mathcal{F}(K,C)$ is contained in $K$.
  \end{enumerate}
  \renewcommand{\theenumi}{(\roman{enumi})}
\end{definition}

Similarly to
$\Mjkhat$, we also define an analogous concept for the case $k=j$. 

\begin{definition} \label{def:Njj}
A pair $(K,C)$ is called a \emph{copy of $\Njj$} if
\begin{itemize}
\item $K$ is an isolated $j$-simplex;
\item $C$ consists of
    the first $j$ vertices in $K$ in the increasing
    order on $[n]$.
\end{itemize} 
\end{definition}
We will see later
(Corollary~\ref{cor:NjkMjk}) that the shell
required for \ref{Mjkhat:shell} in 
Definition~\ref{def:Mjkhat} (and the analogous condition in Definition~\ref{def:Mjjhat})
is very likely to exist if
$\tim$ is `large enough', which will be the case well before the critical
range for the disappearance of $\Mjk$. Thus the presence of $\Mjkhat$ and of 
$\Mjk$ are essentially equivalent events for sufficiently large $\tim$, allowing us to switch our focus
to the simpler $\Mjk$.

We also define the following random variables, 
which we
will later use to count the 
number of minimal obstructions in the complex
(e.g.\ Lemma~\ref{lem:expectedNjk}).

\begin{definition}\label{def:variables}
 For $j \le k \le d$, let
\[ \varNjk = \varNjk(\tim) \ :=\  | \{ \text{ copies of } \Njk \text{ in } \cG_\tim \text{ } \} |\] \label{var:Njk}
and
\[ \varMjkhat = \varMjkhat(\tim)\ :=\  | \{ \text{ copies of } \Mjkhat \text{ in } \cG_\tim \text{ } \}|.\] \label{var:Mjkhat}
\end{definition}

We now justify our interpretation of $\Mjkhat$ as a minimal obstruction
to \connectedness, first observing that it is certainly an obstruction (Corollary~\ref{cor:obstruction}).
To show this, we define a $j$-cocycle
which is not a $j$-coboundary---the function
we choose will depend only on the underlying
copy $(K,C)$ of $\Njk$. 

\begin{definition} \label{def:f_M,r}
Let $M=(K,C)$ be a copy of $\Njk$ in a simplicial complex.
\begin{enumerate}
\item We denote by $\ord (K,C)$ the
(unique)
  ordering $v_0,\dotsc,v_{n-1}$ of all vertices in $[n]$ such that $C = \{v_0,\dotsc,v_{j-1}\}$,
  $K
  = \{v_0,\dotsc,v_{k}\}$, 
  and furthermore the vertices within $C$, 
  within $K\setminus C$, and within 
  $[n]\setminus K$ are ordered
  according to the  
  increasing order in $[n]$.
  \item Given $\ord (K,C)$, for any $r \in R$ 
  we define the following $j$-cochain
 $f_{M,r}$. 
 For every
  ordered $j$-simplex $\sigma = [v_{i_0},\dotsc,v_{i_j}]$ with $i_0 <
  \dotsb < i_j$, we set
  \begin{equation*}
    f_{M,r}(\sigma) \ :=\ 
    \begin{cases}
      r & \text{if }\sigma\in\mathcal{F}(K,C), \text{ i.e.\ } i_s=s \text{ for }0\le s \le j-1
           \text{ and }j\le i_j\le k,\\
      0_R & \text{otherwise},
    \end{cases}
  \end{equation*}
  and we extend this function to all $j$-simplices with different
  orderings so as to obtain a $j$-cochain.
\end{enumerate} 
\end{definition}

\begin{prop}\label{prop:fMrarisesfromMjk}
Let $M=(K,C)$ be a copy of $\Njk$ 
in a simplicial complex $\cG$ and let $f$ be 
a $j$-cochain whose support is contained within
the flower $\cF(K,C)$. Then the following hold.
\begin{enumerate}
\item \label{fMrari:cocycle} 
The $j$-cochain $f$ is a $j$-cocycle if and only if $f=f_{M,r}$ 
for some $r\in R$.
\item \label{fMrari:bad} 
Suppose that there exist $w \in K 
\setminus C$ and $a\in [n]\setminus K$ such that
$(K,C,w,a)$ is a copy of $\Mjkhat$ in $\cG$. Then
$f$ is a $j$-cocycle but 
not a $j$-coboundary if and only if
$f=f_{M,r}$ for some $r\in R\setminus \{ 0_R\}$.
\end{enumerate}
\end{prop}

\begin{proof}
\ref{fMrari:cocycle} First observe that if $k=j$ then 
$\cF(K,C)=\{K\}$ and $K$ is an isolated $j$-simplex by
Definition~\ref{def:Njj}.
Hence a $j$-cochain with support contained in $K$ 
is necessarily of the form $f_{M,r}$, where $r \in R$
is the value it assigns to (the appropriate
ordering of) $K$, and is a
$j$-cocycle since no $(j+1)$-simplex contains $K$.

Now consider $k$ with $k\ge j+1$.
 Let $K=\{v_0,\ldots,v_k\}$ and 
$C=\{v_0,\ldots,v_{j-1}\}$ according to $\ord(K,C)$,
and let $\rho$ be a $(j+1)$-simplex.
By~\ref{Njk:flower},
$(\delta^j f)(\rho) = 0_R$ 
follows immediately
unless $C \subset \rho \subseteq K$.
We may therefore assume that $\rho = \{v_0, \ldots, v_{j-1},
v_{i_1},v_{i_2}\}$, with $j\le i_1 \le i_2 \le k$. 
Then we have
\[
  (\delta^j f)(\rho) = (-1)^j f([v_0,\dotsc,v_{j-1},v_{i_2}])
    + (-1)^{j+1} f([v_0,\dotsc,v_{j-1},v_{i_1}]).
\]
This implies that $f$ is a $j$-cocycle if and only if
on each petal $[v_0,\ldots,v_{j-1},v_i]$ with $j \le i \le k$,
it takes the same value $r \in R$, i.e.\ $f=f_{M,r}$.

\noindent \ref{fMrari:bad} If $f$ is a $j$-cocycle 
and not a $j$-coboundary, by \ref{fMrari:cocycle}
we already know that there exists $r \in R$ such that 
$f=f_{M,r}$. If $r=0_R$, then $f\equiv
0_R$ and thus 
$f$ is a $j$-coboundary, 
a contradiction.

Conversely, if $f=f_{M,r}$ for some $r\in R\setminus \{0_R\}$, 
then $f$ is a $j$-cocycle by \ref{fMrari:cocycle}. 
Furthermore, property~\ref{Mjkhat:shell} 
in Definition~\ref{def:Mjkhat}
implies that $C\cup\{w\}\cup\{a\}$
  is a $j$-shell that meets the support of $f$ in precisely one
  petal, namely in $C\cup\{w\}$. 
  Thus Lemma~\ref{lem:shellobstruction}
  yields that $f$ is not a $j$-coboundary.
\end{proof}

\begin{cor}\label{cor:obstruction}
Suppose that in a simplicial complex $\cG$ the $4$-tuple
$M=(K,C,w,a)$ forms a copy of $\Mjkhat$. Then 
$H^j(\cG;R) \neq 0$.
\end{cor}

\begin{proof}
For any $r\in R\setminus \{ 0_R\}$, the function $f_{M,r}$ defined in Definition~\ref{def:f_M,r} is a
$j$-cocyle but not a $j$-coboundary 
by Proposition~\ref{prop:fMrarisesfromMjk}~\ref{fMrari:bad}, 
i.e.\ the cohomology class of $f_{M,r}$ is a
non-zero element of $H^j(\cG;R)$.
\end{proof}

The next lemma shows that copies of $\Mjkhat$ are also (in a natural sense)
\emph{minimal} obstructions. Given a $k$-simplex $K$ and a collection
$\cS$ of $j$-simplices, define $\cS_K$ to be the set of $j$-simplices of
$\cS$ contained in $K$.

\begin{lem} \label{lem:minobst} 
  Let $\cS$ be the support of a $j$-cocycle $f$ in a $d$-complex $\cG$.
  Then for each $k$ with $j+1 \le k \le d$ 
  and each $k$-simplex $K$,
  \begin{enumerate}
  \item\label{minobst:size}
    either $\cS_K = \emptyset$ or both $|\cS_K|\geq k-j+1$
    and $\bigcup_{\sigma \in \cS_K} \sigma = K$;
  \item\label{minobst:flower}
    if $|\cS_K| = k-j+1$, then $\cS_K$ forms a $j$-flower in $K$.
  \end{enumerate}
\end{lem} 

Note in particular that Lemma~\ref{lem:minobst} implies that the support $\cS$ of any non-trivial
$j$-cocycle satisfies at least one of the following three properties:
\begin{itemize}
\item $\cS_K$ is empty for every $k$-simplex $K$;
\item $|\cS|\ge k-j+2$;
\item $|\cS|= k-j+1$ and $\cS$ forms a $j$-flower in some $k$-simplex $K$.
\end{itemize}
Since in the third case a $j$-shell containing a petal, with its single additional apex vertex,
is the simplest (though by no means the only) way of ensuring that
the corresponding $j$-cocycle is not a $j$-coboundary, this justifies
why a copy of $\Mjkhat$ may be considered a minimal obstruction
to the vanishing of the $j$-th cohomology group.

\begin{proof}[Proof of Lemma~\ref{lem:minobst}]
  \ref{minobst:size} Suppose $\cS_K \ne \emptyset$ and let $\sigma_0 \in
  \cS_K$. Denote the vertices of $\sigma_0$ and of $K\setminus\sigma_0$
  by $u_0,\dotsc,u_j$ and by $v_1,\dotsc,v_{k-j}$, respectively. For
  each $i\in[k-j]$, the ordered $(j+1)$-simplex $[u_0,\dotsc,u_j,v_i]$
  has to be mapped to $0_R$ by $\delta^jf$ and thus the underlying
  unordered simplex $\sigma_0 \cup \{v_i\}$ contains some $j$-simplex
  $\sigma_i \in \cS_K \setminus \{ \sigma_0\}$,
  which therefore contains $v_i$. The simplices $\sigma_0, \ldots,
  \sigma_{k-j}$ are distinct, because each $v_i$ lies in $\sigma_i$ but in
  no other $\sigma_{i'}$. Therefore  $|\cS_K| \ge k-j+1$ and 
   \[ K \supseteq \bigcup\nolimits_{i=0}^{k-j} \sigma_i \supseteq \sigma_0 \cup \{v_1,\ldots, v_{k-j}\} = K .\]
  \noindent
  \ref{minobst:flower} Suppose now that
  $\cS_K=\{\sigma_0,\ldots,\sigma_{k-j}\}$ with
  $\sigma_0,\dotsc,\sigma_{k-j}$ defined as above. For $2 \le i \le 
  k-j$ (if such indices exist), the $(j+1)$-simplex $\sigma \ :=\  \sigma_1\cup\{v_i\}$ contains
  $\sigma_1$, but no $\sigma_{i'}$ with $i' \notin \{1,i\}$. By the
  choice of $f$ as a $j$-cocycle, $\delta^j f$ maps each ordering of
  $\sigma$ to $0_R$ and thus $\sigma$ has to contain at least two
  elements of $\cS_K$, implying that $\sigma_i \subset \sigma$. This
  means that
  \begin{equation*}
    \sigma_1\cap\sigma_i = \sigma\setminus\{v_1,v_i\} =
    \sigma_0\cap\sigma_1.
  \end{equation*}
  As this holds for all $i$, $\cS_K$ forms a flower in $K$ with centre
  $C = \sigma_0\cap\sigma_1$.
\end{proof}

The proofs of our main results 
(Theorems~\ref{thm:main} and~\ref{thm:criticalwindow})
are significantly more difficult than might naively be expected due to the fact that 
both the presence of copies of $\Mjkhat$ and 
\connectedness\ in $\cG_\tim$  
are \emph{not} monotone 
properties. 
Indeed,
we observe that in Definition~\ref{def:Mjkhat} while \ref{Mjkhat:simplex} and 
\ref{Mjkhat:shell} are monotone increasing
properties, property~\ref{Mjkhat:flower}
is monotone decreasing.
Thus, in principle the random process
$(\cG_\tim)$ could oscillate between
being \connected\ or not, as the following example shows.

\begin{example}\label{ex:nonmono}
We consider the case $j=1$. 
  Let $\cG$ be the simplicial complex 
  on vertex set $\{1,2,3,4\}$ generated
  by the hypergraph with edges $\{1,2\}$,
  $\{2,3\}$, and $\{3,4\}$, as in 
  Figure~\ref{fig:nonmon}. It is easy to see
  that $\cG$ is
  $1$-cohom-connected and thus
  contains no copies of $\hat M_{1,k}$ 
  for any $k\ge 1$. 
  If we add the $2$-simplex $\{1,3,4\}$ 
  (and its downward-closure) to $\cG$,
  the $4$-tuple 
  $(\{1,3,4\},\{3\},1,2)$ creates
  a copy of $\hat M_{1,2}$ and 
  thus we obtain a complex $\cG'$ which is
  not $1$-cohom-connected. 
   Adding the $2$-simplex
  $\{1,2,3\}$ to $\cG'$ yields the 
  complex $\cG''$ which is
  again $1$-cohom-connected and 
  thus contains no copies of $\hat M_{1,k}$ 
  for any $k\ge 1$.
\end{example}

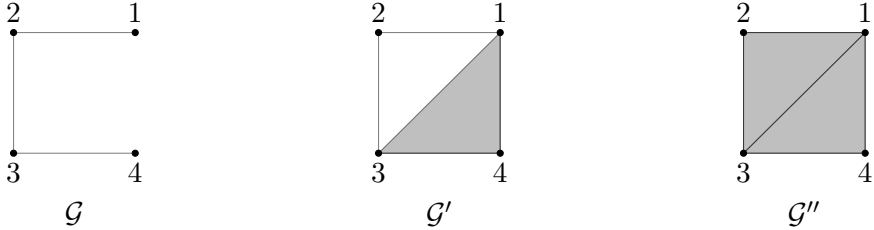
\begin{figure}[htbp]
  \centering
  \begin{tikzpicture}[scale=0.8]
  
    \draw[opacity=0.5]  (3.5,0) -- (1.5,0) -- (1.5,2) -- (3.5,2);
    
    \draw[fill] (3.5,2) circle [radius=0.05] node [above] {$1$};  
    \draw[fill] (1.5,2) circle [radius=0.05] node [above] {$2$};
    \draw[fill] (1.5,0) circle [radius=0.05] node [below] {$3$};
    \draw[fill] (3.5,0) circle [radius=0.05] node [below] {$4$};
            
    \node at (2.5,-1) {$\cG$};

    \begin{scope}[xshift=6cm]
    
    \filldraw [opacity=0.5,fill=gray] (1.5,0) -- (3.5,0) -- (3.5,2) -- cycle;

    \draw[opacity=0.5] (3.5,0) -- (1.5,0) -- (1.5,2) -- (3.5,2) -- cycle;

    \draw[fill] (3.5,2) circle [radius=0.05] node [above] {$1$};  
    \draw[fill] (1.5,2) circle [radius=0.05] node [above] {$2$};
    \draw[fill] (1.5,0) circle [radius=0.05] node [below] {$3$};
    \draw[fill] (3.5,0) circle [radius=0.05] node [below] {$4$};

    \node at (2.5,-1) {$\cG'$};

    \end{scope}

    \begin{scope}[xshift=12cm]
      
    \filldraw [opacity=0.5,fill=gray] (1.5,0) -- (3.5,0) -- (3.5,2) -- cycle;
    \filldraw [opacity=0.5,fill=gray] (1.5,0) -- (1.5,2) -- (3.5,2) -- cycle;
    
    \draw[opacity=0.5] (3.5,0) -- (1.5,0) -- (1.5,2) -- (3.5,2) -- cycle;

    \draw[fill] (3.5,2) circle [radius=0.05] node [above] {$1$};  
    \draw[fill] (1.5,2) circle [radius=0.05] node [above] {$2$};
    \draw[fill] (1.5,0) circle [radius=0.05] node [below] {$3$};
    \draw[fill] (3.5,0) circle [radius=0.05] node [below] {$4$};

    \node at (2.5,-1) {$\cG''$};
    
    \end{scope}
  \end{tikzpicture}
  \caption{Adding simplices might create new copies of $\Mjkhat$ or
    destroy existing ones.}
  \label{fig:nonmon}
\end{figure}

In order to determine the critical range for the disappearance of copies
of $\Njk$, in Lemma~\ref{lem:expectedNjk} 
we will calculate the expectation of $\varNjk$,
i.e.\ the number of copies of $\Njk$
(Definition~\ref{def:variables}).
We first estimate the probability of~\ref{Njk:flower}. Define
  \begin{equation}\label{eq:q}
    \bq = \bq(\bar\allp,n,j) \ :=\  \prod\nolimits_{k=j+1}^{d}(1- \bar p_k)^{\binom{n-j-1}{k-j}}. 
  \end{equation}
Observe that $\bq$ is the probability that a given set of $j+1$ vertices
(which may or may not form a $j$-simplex)
is not in any $k$-simplex of $\cG(n,\barallp)$ (i.e.\
at time $\tim=1$) for any $k\ge j+1$.
Moreover if $\bar\allp$ is a $j$-admissible 
direction, by Definition~\ref{def:criticalp} we have 
\begin{equation} \label{eq:q_computing}
\bq  = (1+o(1))
\exp\left(-\sum\nolimits_{k=j+1}^d
\left(\bar\alpha_k \log n+ \frac{\bar \beta_k}{n^{\bar \gamma_k}}\right)\right) ,
\end{equation}
  because by~\ref{parameters:zero} 
   at least one of
  $\bar\alpha_k,\bar\gamma_k$ is zero and thus $\frac{\bar\alpha_k}{n^{\bar\gamma_k}}
  = \bar\alpha_k$.

The next lemma implies that
for any $\tim = O(1)$, the probability of~\ref{Njk:flower} in
 $\cG_{\tim}$ is
approximately $\bq^{\; \tim (k-j+1)}$---we state the lemma in a slightly more general setting, since
we will need to apply it in different situations (for example when calculating the second moment
of $\varNjk$).

\begin{prop}\label{prop:localisedprob}
Let $\tim=O(1)$ and let $\allp = (p_1,\ldots,
p_d) = \tim \bar\allp$ for
a $j$-admissible direction $\bar\allp$.
Let $\cJ$ be a collection of $O(1)$ many $(j+1)$-sets in $[n]$ and let $\cS$ be a collection of
$O(1)$ many sets of vertices of size between $j+2$ and $d+1$.
Let $A$ be the event that no $(j+1)$-set of $\cJ$ lies in any $k$-simplex $K$ of $\cG_\tim$
with $j+1\le k \le d$ and $K\notin \cS$. Then
$$\Pr (A)= (1+o(1))\bq^{\; \tim|\cJ|}.$$
\end{prop}

The proof of this proposition is straightforward and appears in Appendix~\ref{sec:proofoflocalisedprob}.
We now apply Proposition~\ref{prop:localisedprob} to calculate the 
expectation of $\varNjk$, for $\tim =O(1)$.

Suppose first that $k \ge j+1$. There are $\binom{n}{k+1}\binom{k+1}{j}
  = (1+o(1))\frac{n^{k+1}}{j!(k-j+1)!}$ ways to choose a pair $(K,C)$ that might
  form a copy of $\Njk$. The $(k+1)$-set $K$ forms a $k$-simplex
  in $\cG_\tim$ with probability $p_k$ (recall that $p_k\le 1$ by Remark~\ref{rem:pkprob}).
  By Proposition~\ref{prop:localisedprob} applied with $\cJ=\cF(K,C)$ being the
  set of petals and with $\cS = \{\sigma\subseteq K \,:\, |\sigma| \ge 
  j+2\}$, the probability that~\ref{Njk:flower} holds
  is $(1+o(1))\bq^{\; \tim(k-j+1)}$. Therefore,
  \begin{equation}\label{eq:expXjk}
    \EE(\varNjk) = (1+o(1))\frac{n^{k+1}p_k}{j!(k-j+1)!}\bq^{\; \tim (k-j+1)}.
  \end{equation}

The case $k=j$ is very similar, but for a pair $(K,C)$ 
  that forms a copy of $\Njj$ we only require that the 
  set $K$ is an isolated $(j+1)$-simplex,
  since the centre $C$ is uniquely defined 
  (see Definition~\ref{def:Njj}).
   On the other hand, we need to be careful if $p_j>1$,
  since then $p_j$ must be replaced by $1$ in any probability calculations.
We have
\begin{equation}\label{eq:expXjj}
\EE(\varisol) = (1+o(1)) 
\frac{n^{j+1} \min\{p_j,1\}}{(j+1)!} \bq^{\; \tim}.
\end{equation}
In the next lemma we use \eqref{eq:expXjk} and
\eqref{eq:expXjj} to obtain an explicit
expression for $\log \left( \EE(\varNjk) \right)$, which we 
will need in Section~\ref{sec:findingminimal}.
Recall that given a vector $\allp$, the parameters $\logp_k$, $\sublogp_k$, and 
$\constp_k$
are as defined in Definition~\ref{def:generalparameters}.

\begin{lem}\label{lem:expectedNjk}
 Let $\allp = (p_1,\ldots,p_d) = \tim \bar\allp$ for a $j$-admissible
  direction $\bar\allp$, where $\tim=O(1)$, but $\tim=\Omega(n^{-c})$
  for some positive constant $c$.
  Then the number
  $\varNjk$ of copies of $\Njk$ in $\cG_{\tim}$ satisfies
  \begin{equation*}
    \log \left(\EE(\varNjk) \right) = \logp_k\log n + \sublogp_k + \constp_k + o(1)
  \end{equation*}
  for all $j \le k \le d$ with $\bar p_k \not= 0$.
\end{lem}

The proof of Lemma~\ref{lem:expectedNjk} consists of (standard, but involved)
technical calculations and therefore
is deferred to Appendix~\ref{sec:proofofexpectedNjk}.

Recall that in our main Theorems~\ref{thm:main} 
and~\ref{thm:criticalwindow} we consider a
$j$-critical direction $\bar\allp$
which in particular is a $j$-admissible direction (cf.\ 
Definitions~\ref{def:criticalp} and~\ref{def:criticaldirection}).
Thus Proposition~\ref{prop:localisedprob} and 
Lemma~\ref{lem:expectedNjk} are applicable.
By Lemma~\ref{lem:expectedNjk}, heuristically the critical range for the disappearance
of copies of $\Njk$ is when $\logp_k\log n + \sublogp_k = O(1)$, that is,
$\logp_k = 0$ and $\sublogp_k = O(1)$. This justifies the conditions~\ref{crit:less} and~\ref{crit:equal}
in Definition~\ref{def:criticaldirection},
which together with Lemma~\ref{lem:expectedNjk}
yield that for $\allp = \barallp$,
\begin{equation}\label{eq:expMjkbar}
\begin{array}{lcr}
\EE(\varNjkbar) = 1+o(1) &\text{and } &\EE(\varNjk) \le 1+o(1) \text{ for all other indices } k.
\end{array}
\end{equation}
In other words, heuristically
$\barallp$ is in a critical range for the disappearance of copies of
$\Njkbar$, while for all other $k$, $\barallp$ is either in or already
beyond the critical range for the disappearance of copies of $\Njk$. We
will see later (Corollary~\ref{cor:NjkMjk}) that in this range, whp all
copies $(K,C)$ of $\Njk$ can be extended to copies $(K,C,w,a)$ of $\Mjkhat$.
Thus $\barallp$ is also in the critical range for the disappearance of minimal obstructions.

Recall that in Theorem~\ref{thm:main}, we consider 
\begin{equation*}
  \tim_j^* \ :=\  \sup\{ \tim \in \mathbb{R}_{\ge 0} \mid \cG_{\tim} \text{ contains a copy of }\Mjkhat \text{ for some } k \text{ with } j \le k \le d \}.
\end{equation*}
In other words, $\tim_j^*$ is the scaled birth time of a
simplex whose appearance causes the last minimal obstruction to
disappear. We denote the dimension of this obstruction by $\ell$
(i.e.\ let $\ell$ be the index such that this obstruction is a copy of $\Mjlhat$). \label{param:ell}
For future reference, we collect the definitions of the special indices
$\bar k,k_0,\ell$ which we have fixed so far.

\begin{definition}\label{def:indices}
  Let $\bar k,k_0,\ell$ be integers such that
  \begin{enumerate}
  \item \label{def:indices:bark} $j \le \bar k \le d$ and $\bar\logp_{\bar k}\log n + \bar\sublogp_{\bar k} + \bar\constp_{\bar k} = 0$
  (see~\ref{crit:equal});
  \item \label{def:indices:kzero} $j+1 \le k_0 \le d$ and $\bar\alpha_{k_0} \not= 0$ (see~\ref{parameters:k0});
  \item \label{def:indices:ell} at time $\tim_j^*$, a (last) copy of $\Mjlhat$ vanishes.
  \end{enumerate}
\end{definition}

\bigskip
\section{Finding minimal obstructions}\label{sec:findingminimal}

To prove Lemma~\ref{lem:jinterval},
the strategy is to show that whp a copy of 
$\Mjkhat$ (for some $j \le k \le d$) exists 
in $\cG_\tim$ for every $\tim \in I_j(\eps) = 
[\eps/n, \tim_j^*)$.
Therefore in this section we study the behaviour
of the minimal obstructions $\Mjkhat$.

We start by showing that
at the beginning of the interval
$I_j(\eps)$
we will already
have a growing number of copies of $\Mjkzerohat$, where $k_0$ is as in
Definition~\ref{def:indices}~\ref{def:indices:kzero}.

\begin{lem}\label{lem:firstMjk}
  Let $\eps>0$ be constant. If 
  $\tim =
  \frac{\eps}{n}$, then whp $\cG_{\tim}=\cG(n,\tim \barallp)$ contains $\Theta((\log n)^{j+2})$ copies of
  $\Mjkzerohat$ whose associated copies of $\Njkzero$ are all distinct.
\end{lem}
The proof of Lemma~\ref{lem:firstMjk} is a standard but 
slightly technical application of the second moment method, and is therefore postponed 
to Appendix~\ref{sec:prooffirstMjk}.

In Lemma~\ref{lem:shells} we will 
show that in a range closer to criticality
(i.e.\ for $\tim$ closer to $1$)
$j$-shells are very likely to exist. In order to formulate the statement,
we first define the operation of `adding a simplex'.

\begin{definition} \label{def:addingsimplex}
  Given a complex $\cG$ on vertex set $V$ and a non-empty set $B
  \subseteq V$, we define $\cG + B$ to be the complex obtained by
  adding the set $B$ and its downward-closure to $\cG$, i.e.\
  \begin{equation*}
    \cG + B \ :=\  \cG \cup \{ 2^B \setminus \{\emptyset\} \}.
  \end{equation*}
\end{definition} 
\noindent
Observe that if $B$ is already a simplex of $\cG$, then $\cG + B =
\cG$.

\begin{lem}\label{lem:shells}
  For every $\eps>0$ there exists a constant $\zeta>0$ such that if 
  $\tim \ge \frac{\eps}{\log n}$, then whp for every
  $(j+1)$-set $B$, the complex $\cG_{\tim}+B$ contains at least
  $\zeta n$ many $j$-shells that contain $B$.
\end{lem}

\begin{proof}
Let $L_1,L_2,\ldots,L_{j+1}$ be the $j$-sets contained in $B$. 
We are interested in the vertices $a \in [n]\setminus B$ such that
$B \cup \{ a \}$ forms a $j$-shell in $\cG_{\tim}+B$, i.e.\ such that
$L_i \cup \{a\}$ is a $j$-simplex in $\cG_{\tim}$ for every $i \in [j+1]$. 
We only consider a certain type of such $j$-shells, obtaining a lower bound on their total number.

Let $A,D \subset [n]$ be disjoint sets, both of size $\lceil n/3 \rceil$
and disjoint from $B$.
Recall (Definition~\ref{def:indices}~\ref{def:indices:kzero}) that $k_0$
is an index with $j+1\le k_0\le d$ such that $\bar\alpha_{k_0}=0$.
We consider (potential) $j$-shells $B \cup \{a\}$ formed in the following way:
\begin{itemize}
\item the (apex) vertex $a$ is in $A$;
\item for each $i \in [j+1]$ there exists a set $R_i \subset D$, with $|R_i| =k_0-j$,
such that $L_i \cup \{a\}$ forms a $j$-simplex in $\cG_{\tim}$ (and thus also in $\cG_{\tim} + B$)
as a subset of the $k_0$-simplex $R_i'\ :=\   L_i \cup \{a\} \cup R_i$
with scaled birth time at most $\tim$ (i.e.\ with birth time at most $p_{k_0}=\tim \bar p_{k_0}$).
\end{itemize}
Since a different choice of the triple $(L_i, a, R_i)$ never gives the same simplex $R_i'$,
we have independence in the following calculations.

For fixed $L_i$ and $a$, the probability that no such $R_i$ exists is
\begin{align} \label{eq:probnoRi}
(1-p_{k_0})^{\binom{|D|}{k_0-j}}
& \le \exp \left( - \left( p_{k_0} \frac{n^{k_0-j}}{4^{k_0-j}(k_0-j)!} \right) \right) \nonumber \\
&\le (1+o(1)) \exp \left(- \frac{\bar{\alpha}_{k_0} \eps}{4^{k_0-j}}\right) \le \exp \left(- \frac{\bar{\alpha}_{k_0} \eps}{4^{k_0}}\right),
\end{align}
where we used that by~\ref{parameters:sublog} we have
$\bar{\beta}_{k_0} = o(\log n)$ since $\bar{\gamma}_{k_0} = 0$. 

For any $a \in A$, let $E_a$ be the event that $B \cup \{a\}$ is a $j$-shell
in $\cG_{\tim} + B$. Using~\eqref{eq:probnoRi}, we obtain
\[ \Pr(E_a) \ge \left( 1 - \exp \left( \frac{\bar{\alpha}_{k_0} \eps}{4^{k_0}}\right)\right)^{j+1} =: c >0. \]

Since the events $E_a$ are independent, the number of such $j$-shells dominates $\mbox{Bi}(\lceil n/3\rceil, c)$. By Chernoff's bound, we have
\[ \Pr \Big( \mbox{Bi}\left( \lceil n/3 \rceil, c \right) < \zeta n \Big) \le \exp \left( - 
\frac{\left( \frac{cn}{3}  - \zeta n\right)^2}{\frac{2cn}{3}} \right) = \exp \left(- \frac{n}{6 c} \left( c - 3 \zeta\right)^2\right). \]

Choosing $0<\zeta<c/3$ and by taking a union bound over all possible choices for the set $B$, we obtain that the probability that there are less than $\zeta n$ many $j$-shells containing $B$ is bounded above by
\[ \binom{n}{j+1} \exp \left(- \frac{n}{6 c} \left( c - 3 \zeta\right)^2\right) = o(1),\]
as required.
\end{proof}

As an immediate corollary, we obtain that in this range, whp every copy of $\Njk$ can be extended
to a copy of $\Mjkhat$, allowing us to consider just copies of $\Njk$ as obstructions to \connectedness.

\begin{cor}\label{cor:NjkMjk}
Let $\varepsilon>0$ be constant and consider the process $(\cG_\tim)$. Then whp any copy $(K,C)$ of $\Njk$
which exists in $\cG_\tim$
for any $\tim \ge \frac{\varepsilon}{\log n}$
can be extended to a copy $(K,C,w,a)$ of $\Mjkhat$ in $\cG_\tim$.
\end{cor}

\begin{proof}
Let $(K,C)$ be any pair of sets that \emph{could} form a copy of
$\Njk$, i.e. $K$ is a $(k+1)$-subset of $[n]$ and $C$ is a $j$-subset
of $K$ (and, if $k=j$, then $C$ consists of the first $j$ vertices of
$K$ in the increasing order on $[n]$). Fix any vertex
$w = w_{K,C} \in K\setminus C$. Lemma~\ref{lem:shells} implies that at
time $\tim = \frac{\varepsilon}{\log n}$, whp for \emph{all} such 
pairs $(K,C)$, there are linearly many $j$-shells in $\cG_\tim + 
(C\cup\{w\})$ that contain $C\cup\{w\}$. For each $(K,C)$, only
$O(1)$ many of these $j$-shells can be subsets of $K$. Therefore
there exist vertices $a = a_{K,C} \in[n]\setminus K$ and $w=w_{K,C}\in K \setminus C$ such that whp,
for every pair $(K,C)$, 
the sides of the $j$-shell $C\cup \{w\} \cup \{a\}$ are all present as $j$-simplices
in $\cG_{\frac{\eps}{\log n}}$, and therefore for any $\tim \ge \frac{\eps}{\log n}$
such that the pair $(K,C)$ forms a copy of $\Njk$ in $\cG_\tim$, also $(K,C,w,a)$ forms a copy of
$\Mjkhat$ in $\cG_\tim$.
\end{proof}

The following proposition describes 
more precisely the 
parameters in Definition~\ref{def:generalparameters} for
$\allp = \tim \bar\allp$ and
$\tim$ `close' to $1$, in terms of the 
analogous parameters defined 
in~\eqref{eq:parameters} and~\eqref{eq:lmn} 
for $\bar\allp$. 

\begin{prop}\label{prop:scaledparameters}
Let $\tim=1+\xi$ with $\xi=\xi(n)=o(1)$  and let $\allp = \tim \barallp$.
Then for all $j\le k \le d$ with $\bar p_k\not=0$,
\begin{align*}
\alpha_k & =\bar \alpha_k, &
\beta_k & = \left(1+\xi\right)\bar \beta_k 
+ \bar \alpha_k  \xi \log n,
& \gamma_k & = \bar \gamma_k,\\
\logp_k & = \bar \logp_k, &
\sublogp_k & = \bar \sublogp_k - (1+o(1))
(k-j+1) \xi \sum\nolimits_{i=j+1}^d \bar \alpha_i \log n, &
\constp_k & =\bar\constp_k.
\end{align*}
\end{prop}
\noindent The easy but technical proof of 
Proposition~\ref{prop:scaledparameters} appears in
Appendix~\ref{sec:proofofscaledparam}.

We derive the following corollary about the expected number of
copies of $\Njk$ in the critical window, which will be crucial for the proof
of the Rank Theorem (Theorem~\ref{thm:criticalwindow}).

\begin{cor}\label{cor:expectatcritwind}
Let $c\in \mathbb{R}$ be a constant and suppose
$(c_n)_{n\ge 1}$ is a sequence of real numbers such 
that $c_n \xrightarrow{n \to \infty} c$. 
Let $\tim=\left(1 + \frac{c_n}{\log n}\right)$  and  
$\allp = \tim \barallp$. 
Then for any $k$ with $j \le k \le d$,
\[ \EE(\varNjk) =  \begin{cases}
       (1+o(1)) 
\exp (\bar\sublogp_k + \bar\constp_k + c(\bar\gamma_k 
- j -1) ) & \text{if }k \text{ is a critical dimension,}\\
  o(1) & \text{otherwise.}
   \end{cases}\]
\end{cor}
\noindent
We delay the proof of Corollary~\ref{cor:expectatcritwind} until
Appendix~\ref{sec:proofexpectcritwind}.

As the last result of this section, we show that for $\tim$ slightly
less than $1$, whp we have many copies of $\Njk$.

\begin{lem}\label{lem:manyNjk}
  Let $\omega_0=\omega_0(n)=o(\log n)$ be a function that tends to infinity
  and let $\tim = \left(1-\frac{1}{\omega_0}\right)$.
  Then there exists a constant $c>0$ such that for any critical dimension
  $k\ge j$,  
  whp there are
  at least $\exp (\frac{c\log n}{\omega_0})$ many copies of $\Njk$ in $\cG_{\tim}$.
\end{lem}

The proof is similar to the proof of Lemma~\ref{lem:firstMjk},
although the second moment calculation is significantly simpler
without the $j$-shell of $\Mjkhat$, and can be found 
in Appendix~\ref{sec:proofmanyNjk}.

\bigskip
\section{Determining the hitting time: proof of Lemma~\ref{lem:hittingtime_simplevers}}\label{sec:hitting}
In this section we consider the hitting time $\tim_j^*$ for the disappearance of the last minimal obstruction, i.e.\
  \begin{equation*}
    \tim_j^* \ :=\  \sup\{ \tim \in \mathbb{R}_{\ge0} \mid \cG_{\tim} \text{ contains a copy of }\Mjkhat\text{ for some } k \text{ with } j \le k \le d \},
  \end{equation*}
as defined in Theorem~\ref{thm:main}.
We will show that whp this happens at around the claimed threshold $\tim=1$ (Lemma~\ref{lem:hittingtime_simplevers}).

Consider the time \label{prob:timprime}
\[ \tim'\ :=\ 1-\frac{\log \log n}{10 d \log n}, \]
and let $\tim''$\label{prob:timsecond}
be the first scaled birth time larger than $\tim'$ such that there are no copies of $\Mjkhat$
in $\cG_{\tim''}$.
Lemmas~\ref{lem:shells} and \ref{lem:manyNjk} tell us that whp $\cG_{\tim'}$
contains a growing number of copies of $\Mjkhat$, thus by definition of $\tim_j^*$
we have $\tim'' \le \tim_j^*$.
The following main result of this section says that they are in fact
equal whp, and indeed whp both are close to $1$.

\begin{lem}\label{lem:hitting}
Whp $\tim_j^*=\tim''$.
Furthermore, suppose
$\omega$ is a function of $n$ that tends to infinity as $n \to \infty$. Then, whp
\[ 1- \frac{\omega}{\log n} < \tim_j^* < 1 + \frac{\omega}{\log n}.\] 
\end{lem}

Observe that Lemma~\ref{lem:hittingtime_simplevers} is an immediate
corollary of Lemma~\ref{lem:hitting}. 
To prove Lemma~\ref{lem:hitting}, we will need some further concepts and some auxiliary results.

\begin{definition}
Given $k\ge j$ and a $(k+1)$-set $K$, a $(j+1)$-set $J\subseteq K$ is
\emph{$K$-localised}\label{comb:Klocal} if every simplex $\sigma$ with $J \subseteq \sigma$
is such that $\sigma \subseteq K$.
\end{definition}

Note that we do not demand that $J$ is a $j$-simplex---if it is not, then it is trivially
$K$-localised for any $K\supseteq J$ since there is no simplex $\sigma \supseteq J$.

\begin{definition}
Given an integer $k$ with $j \le k \le d$, a $k$-simplex $K$ is called a \emph{local $j$-obstacle}\label{comb:localobst} if it contains at least $k-j+1$ many $j$-simplices that are $K$-localised.
\end{definition}

In particular a $j$-simplex is a local $j$-obstacle if and only if it is isolated.
More generally, any copy of $\Mjk$ for $j\le k \le d$ is certainly a local $j$-obstacle,
although a local $j$-obstacle is not necessarily an obstruction to \connectedness.

\begin{lem}\label{lem:obstacles}
  Whp, for all $\tim \ge \tim'$, every local $j$-obstacle in $\cG_{\tim}$
  also exists in $\cG_{\tim'}$.
\end{lem}

\begin{proof}
We will prove the statement for local obstacles of size $k+1$, for some $j\le k \le d$.
The lemma then follows by applying a union bound over all $k$.

We first note that by Remark~\ref{rem:pkprob}, $\bar p_k <1$ if $k\ge j+1$. 
On the other hand, if $k=j$ and $\bar p_j \ge \frac{1}{\tim'}=1+o(1)$,
every $j$-simplex is present in $\cG_{\tim'}$ deterministically,
and therefore the statement of the lemma trivially holds.

Thus in the following we may assume that
\begin{equation} \label{eq:condlocalobst}
\text{either } \quad k\ge j+1 \qquad \quad  \text{ or }  \quad k=j \text{ and } \bar p_j < \frac{1}{\tim'} <  1+ \frac{\log \log n}{9d \log n}.
\end{equation}
Although in the second case we indeed have
$\bar p_j < \frac{1}{\tim'}$, we would 
incur some technical difficulties if the 
probability $\bar p_j$ is very `close'
to $\frac{1}{\tim'}$. Hence, in the following calculations we need to replace
$\tim'$ by a slightly smaller value. 
More precisely,
we consider
\begin{equation*} \label{eq:timminus}
\tim^- \ :=\  1 - \frac{\log \log n}{5 d \log n}.
\end{equation*}
We will show that whp for any $\tim \ge \tim^-$, 
a local $j$-obstacle in $\cG_\tim$ also exists 
in $\cG_{\tim^-}$, thus obtaining the statement for any $\tim \ge \tim'>\tim^-$ as well.
In particular, observe that
\[  \tim^- \bar p_k < 1 \qquad \text{ for every } k\ge j, \]
by Remark~\ref{rem:pkprob} and \eqref{eq:condlocalobst}.

Fix $k\ge j$, let $K$ be a $(k+1)$-set,
and recall that $\tim_K \in [0,1/\bar p_k]$ denotes its scaled birth time, i.e.\ if $t_K$ is the birth time
of $K$ as a $k$-simplex in $\cG_\tim$, then $\tim_K = t_K/\bar p_k$ 
(see~\eqref{eq:scaledbirth}). 
In order to \emph{become} a local $j$-obstacle in $\cG_{\tim}$ for some $\tim> \tim^-$,
$K$ must contain a collection
$\cJ$ of $k-j+1$ many $(j+1)$-sets such that the following conditions are satisfied:
 \renewcommand{\theenumi}{(L\arabic{enumi})}
\begin{enumerate}
\item\label{locobst:notsimpl}
$\tim_K > \tim^-$;
\item\label{locobst:Klocalised} 
every $J\in \cJ$ is $K$-localised in $\cG_{\tim^-}$;
\item\label{locobst:bornbefore}
$K$ is born as a $k$-simplex before any other simplex $I$ that contains
some $J\in \cJ$, but which is not contained in $K$, i.e.\ $\tim_K < \tim_I$ 
for all such $I$.
\end{enumerate}
\renewcommand{\theenumi}{(\roman{enumi})}

Fix the $(k+1)$-set $K$ and the collection 
$\cJ$ of $(j+1)$-sets in $K$.
For this choice of $K$ and $\cJ$,
we denote by 
$L_1$, $L_2$, and $L_3$ the
events that conditions \ref{locobst:notsimpl}, \ref{locobst:Klocalised}, and \ref{locobst:bornbefore} hold, respectively.
 
By definition of our model, we have that 
\begin{equation}\label{eq:L1}
\Pr(L_1) = (1- \tim^- \bar p_k  ).
\end{equation}

In order to compute $\Pr(L_2 \; |\; L_1)$, first observe that $L_2$ is independent of $L_1$. By Proposition~\ref{prop:localisedprob} applied with $\allp=\tim^-\barallp$, we have
\begin{align}
\Pr(L_2 \; | \; L_1) &\stackrel{\phantom{\eqref{eq:expXjk},\eqref{eq:expXjj}}}{=}  
\Pr(L_2) \nonumber \\
&\stackrel{\phantom{\eqref{eq:expXjk},\eqref{eq:expXjj}}}{=} (1+o(1))\bq^{\; \tim^- (k-j+1)} \nonumber \\
&\stackrel{\eqref{eq:expXjk},\eqref{eq:expXjj}}{=} \left( \frac{\Theta(1) \cdot \EE(\barvarNjk)}{n^{k+1} \min\{\bar{p}_k,1\} } \right)^{\tim^-}
\stackrel{(\bar p_k < 1+o(1))}{=} \left( \frac{\Theta(1) \cdot \EE(\barvarNjk)}{n^{k+1} \bar{p}_k } \right)^{\tim^-}, \label{eq:L2}
\end{align}
where $\barvarNjk$ denotes the number of copies of $\Njk$ in $\cG_{1}=\cG(n,\barallp)$ (i.e.
$\tim=1$),
and thus $\EE(\barvarNjk) \le 1+o(1)$ (see \eqref{eq:expMjkbar}).
 
We now want to bound 
$\Pr\big( L_3 \; | \; (L_1 \land L_2) \big)$. 
For any $i$ such that $j+1 \le i \le d$ and 
for any $J\in\cJ$, there are $\binom{n-k-1}{i-j}$ 
many $(i+1)$-sets which contain $J$ and whose  remaining vertices are outside $K$.
In order for $L_3$
to hold,
 all these $(i+1)$-sets (among others) must be born as simplices after $K$
and observe that all of these $(i+1)$-sets are distinct for different choices of $J$. 
 It will be convenient to pick $i=k_0$,
recalling from Definition~\ref{def:indices}~\ref{def:indices:kzero} that $k_0\ge j+1$ is such that
$\bar \alpha_{k_0}\neq 0$ and $\bar \gamma_{k_0}=0$.
 Thus we have a family $\cZ$ of
 \begin{equation}\label{eq:badk0+1sets}
 z\ :=\ |\cZ|=(k-j+1) \binom{n-k-1}{k_0-j} = 
 \Theta\left( \frac{\log n}{\bar p_{k_0}}\right)
 \end{equation}
 many $\emph{bad}$ $(k_0+1)$-sets whose scaled  
 birth times
 are 
 uniformly distributed in the interval $[\tim^-, \frac{1}{\bar{p}_{k_0}}]$
 (since the corresponding simplices are 
 not present in $\cG_{\tim^-}$
 by $L_2$),
 but must all be larger than $\tim_K$,
 in order for $L_3$ to hold.
Similarly, conditioned on 
$L_1$, the scaled birth time $\tim_K$ is uniformly distributed in $[\tim^-,\frac{1}{\bar{p}_k}]$.  
This allows us to prove the following.

\begin{claim}\label{claim:firstbirthtime}
Let $L_3'$ be the event that
$K$ is born as a $k$-simplex before any 
of the bad $(k_0+1)$-sets in $\cZ$.
Then
\[ \Pr\big( L_3' \;|\; (L_1 \land L_2)\big) = 
(1+o(1))  \frac{\bar p_k}{z \bar p_{k_0}
(1 -  \tim^- \bar  p_k  )}
 \stackrel{\eqref{eq:badk0+1sets}}{=} 
 \Theta\left( 
\frac{\bar p_k}{(1- \tim^- \bar p_k )\log n }\right).\]
\end{claim}

We will delay the proof of Claim~\ref{claim:firstbirthtime} until Appendix~\ref{sec:proofoffirstbirthtime}.
We now complete the proof of Lemma~\ref{lem:obstacles}.
Note that
$L_3 \subseteq L_3'$, thus
Claim~\ref{claim:firstbirthtime} in particular
implies that
\begin{equation}\label{eq:L3givenL1L2}
\Pr\big( L_3 \;|\; (L_1 \land L_2)\big) 
\le \Pr\big( L_3' \;|\; (L_1 \land L_2)\big) 
= \Theta\left( 
\frac{\bar p_k}{(1- \tim^- \bar p_k )\log n }\right).
\end{equation}

Putting \eqref{eq:L1}, \eqref{eq:L2},
and \eqref{eq:L3givenL1L2} together
we have
\begin{align*}
\Pr(L_1 \land L_2 \land L_3) & = 
\Pr(L_1) \Pr(L_2 \; | \; L_1)
\Pr\big(L_3 \; | \; (L_1 \land L_2) \big) \\
& =  O \left((1-  \tim^- \bar p_k )
\left( \frac{\EE(\barvarNjk)}{n^{k+1} 
\bar{p}_k } \right)^{\tim^-}  
\frac{\bar p_k}{ (1- \tim^- \bar p_k )\log n}  \right).
\end{align*}
Recalling that $\EE(\barvarNjk)\le 1+o(1)$
by \eqref{eq:expMjkbar}, we thus deduce that
\begin{equation}
\Pr(L_1 \land L_2 \land L_3)
= O\left(\frac{{\bar p_k}^{1- \tim^-}}{
n^{\tim^- (k+1)}\log n}\right).\label{eq:probL1L2L3}
\end{equation}

There are $\Theta(n^{k+1})$ choices for the 
$(k+1)$-set $K$ and, once $K$ is fixed, there 
are $\Theta(1)$ choices for the collection 
$\cJ$ of $k-j+1$ many $(j+1)$-sets in $K$.
Since $\barallp$ is a $j$-admissible
direction (cf.\ Definition~\ref{def:criticalp})
we know that $\bar p_k = O\left( \frac{\log n}{n^{k-j}}\right)$, hence
the expected numbers of 
pairs $(K,\cJ)$ satisfying
\ref{locobst:notsimpl}, \ref{locobst:Klocalised}, and
\ref{locobst:bornbefore} is 
\begin{align*}
\Theta(n^{k+1}) \Pr(L_1 \land L_2 \land L_3)
&\stackrel{\eqref{eq:probL1L2L3}}{=}
O\left(\frac{\left(n^{k+1}\bar p_k\right)^{1- \tim^-}}{\log n}\right)\\
& \stackrel{\phantom{\eqref{eq:probL1L2L3}}}{=}
O\left(\frac{n^{(j+1)(1-\tim^-)}}{(\log n)^{
\tim^-}} \right)\\
& \stackrel{\phantom{\eqref{eq:probL1L2L3}}}{=}
 O \left( \exp\left( 
\left(\frac{j+1}{5d} - (1-o(1)) \right) 
\log \log n\right) \right)\\
& \stackrel{\phantom{\eqref{eq:probL1L2L3}}}{=} O \left( \exp \left(
\frac{-\log \log n}{2} \right) \right) 
\xrightarrow{n \to \infty} 0.
\end{align*}
Therefore by Markov's inequality, whp there are no such pairs
$(K,\cJ)$, as required.
\end{proof}

We are now ready to prove the main result of this section.

\begin{proof}[Proof of Lemma~\ref{lem:hitting}]
Observe that in particular a copy of $\Njk$ (or, more precisely, the associated $k$-simplex)
is a local $j$-obstacle.
Lemma~\ref{lem:obstacles} shows that if a local $j$-obstacle is present
in $\cG_\tim$ for some
$\tim \ge \tim'$, then whp it already existed in $\cG_{\tim'}$. 
Therefore, if $\tim'' < \tim_j^*$ then a copy of $\Mjkhat$ appears in between these two times,
but whp the associated copy of $\Njk$ would already exist at time $\tim''$ and thus whp would form a copy
of $\Mjkhat$ at that time too, by Lemma~\ref{lem:shells}. 
This cannot happen by definition of $\tim''$, which gives $\tim'' = \tim_j^*$ whp, as required.  

To prove the second statement, observe that
by Lemmas~\ref{lem:shells} and \ref{lem:manyNjk}, whp we have $\tim_j^* > 1 - \frac{\omega}{\log n}$,
proving the lower bound. 

In the proof of the upper bound it will be convenient to assume that 
$\omega = o(\log n)$---this assumption is permissible since the statement becomes stronger for smaller $\omega$.

For $\tim= 1+ \frac{\omega}{\log n}$ and any $j\le k \le d$, applying
\eqref{eq:expXjk} (for $k \ge j+1$) or \eqref{eq:expXjj} (for $k=j$) to
$\varNjk$ and to the number $\barvarNjk$ of copies
of $\Njk$ at time $1$, we deduce that
\begin{align*}
  \EE(\varNjk)
  &\stackrel{\phantom{\eqref{eq:q_computing}}}{=}
   (1+o(1)) \EE(\barvarNjk) \cdot \bq^{(\tim-1)(k-j+1)}\\
  &\stackrel{\eqref{eq:q_computing}}{=}
   (1+o(1)) \EE(\barvarNjk) \cdot \exp\left(-\frac{\omega}{\log n}(k-j+1)\sum\nolimits_{k=j+1}^d
\left(\bar\alpha_k \log n+ \frac{\bar \beta_k}{n^{\bar \gamma_k}}\right)\right)\\
&\stackrel{\phantom{\eqref{eq:q_computing}}}{\le} (1+o(1)) \EE(\barvarNjk) \exp(-\bar\alpha_{k_0} \omega) =o(1),
\end{align*}
where we are using that $\EE(\barvarNjk) \le 1 + o(1)$ by \eqref{eq:expMjkbar} and that the index $k_0$ is such that $\bar\alpha_{k_0} \neq 0$ and $\bar\gamma_{k_0} = 0$.

Hence, by Markov's inequality whp there are no copies of $\Njk$
and thus also no copies of $\Mjkhat$.
This means that whp $\tim'' < 1+\frac{\omega}{\log n}$, and we have already shown that whp $\tim_j^* = \tim''$.
\end{proof}

\bigskip
\section{Subcritical case: proof of Lemma~\ref{lem:jinterval}}
\label{sec:subcritical}

In this section we first derive some auxiliary results
and combine them to prove
Lemma~\ref{lem:jinterval}, which plays
a crucial role in the proof of
the subcritical case
(i.e. statement~\ref{main:subcritical}) of Theorem~\ref{thm:main}.

Given a constant $\eps>0$,
in order to show that whp
$H^j(\cG_\tim;R)\neq 0$ for every
$\tim \in I_j(\eps)= [\eps/n, \tim_j^*)$,
we split this range into three separate intervals,
$$
\left[\frac{\eps}{n},\tim_j^*\right)=
\left[\frac{\eps}{n},\frac{\delta}{\log n}\right]
\cup \left[\frac{\delta}{\log n},1-\frac{1}{(\log n)^{1/3}}\right]
\cup \left[1-\frac{1}{(\log n)^{1/3}},\tim_j^*\right)
$$
for some constant $\delta>0$,
and prove that
for each of these ranges there is some $k$ and one copy of $\Mjkhat$ which exists throughout
the subinterval (Lemmas~\ref{lem:firstinterval},~\ref{lem:secondinterval}, 
and~\ref{lem:thirdinterval}).

\begin{lem}\label{lem:firstinterval}
  For every constant $\eps>0$, there exists a constant $\delta>0$
  such that whp
  there is at least one copy of $\Mjkzerohat$ that is present
  in the process $(\cG_{\tim}) = (\cG(n,\tim\bar\allp))_\tim$ for all values
  \begin{equation*}
    \tim \in \left[\frac{\eps}{n},\frac{\delta}{\log n}\right].
  \end{equation*}
\end{lem}

\begin{remark}
Indeed, Lemma~\ref{lem:firstinterval} would
also hold with $k_0$ replaced by any index $k$ with $j+1 \le k \le d$ such that $\bar\alpha_{k}\neq 0$.
\end{remark}

\begin{proof}[Proof of Lemma~\ref{lem:firstinterval}]
  By Lemma~\ref{lem:firstMjk}, there exist constants $0<c_1<c_2$
  (depending on $\eps$) such that whp the number $\varMjkzerohat$ of
  copies of $\Mjkzerohat$ in $\cG_{\frac{\eps}{n}}$
  satisfies
  \begin{equation*}
    c_1(\log n)^{j+2} \le \varMjkzerohat \le c_2(\log n)^{j+2},
  \end{equation*}
  and all these copies of $\Mjkzerohat$ originate from distinct
  copies of $\Njkzero$.
  We will show that whp at least one of these copies survives
  (i.e.\ remains a copy of $\Mjkzerohat$) until time $\tim = \frac{\delta}{\log n}$, for
  a suitable constant $\delta >0$.

  For each index $k$ with $j+1 \le k \le d$, call a $(k+1)$-set
  \emph{dangerous} if it is \emph{not} a $k$-simplex in
  $\cG_{\frac{\eps}{n}}$ and contains a petal of at least one
  copy of $\Mjkzerohat$.
  Since there are at most $c_2(k-j+1)(\log n)^{j+2}$ petals,
  and each is contained in at most $\binom{n-j-1}{k-j}\le \frac{n^{k-j}}{(k-j)!}$ many $(k+1)$-sets,
  setting $c_3= \max_k\frac{(k-j+1)c_2}{(k-j)!}$,
  for
  all $k$ the number of dangerous $(k+1)$-sets is at most
  \begin{equation*}
    c_3(\log n)^{j+2}n^{k-j}.
  \end{equation*}
  For each dangerous $(k+1)$-set, the probability that it becomes a simplex
  by time $\tim=\delta/(\log n)$ is the probability that its scaled birth time
  is at most $\delta/(\log n)$ conditioned on the event that it is at least $\eps /n$,
  which is
  $$
  \frac{\left(\frac{\delta}{\log n}-\frac{\eps}{n}\right)\bar p_k}{1-\frac{\eps }{n}\bar p_k}
  \le \frac{\delta \bar p_k}{\log n} \le \frac{(k-j)!(\bar \alpha_k +1) \delta}{n^{k-j}}.
  $$
  Setting $c_4\ :=\ \max_k (k-j)!(\bar \alpha_k+1)$,
  the number of dangerous $(k+1)$-sets that turn into $k$-simplices
  in the time interval we are considering is dominated by
  \begin{equation*}
    \Bi\left(c_3(\log n)^{j+2}n^{k-j},\frac{c_4\delta}{n^{k-j}}\right),
  \end{equation*}
  so by a Chernoff bound, we deduce
  that the number of dangerous sets of \emph{any} size that turn into
  simplices by time $\tim=\delta/(\log n)$ is whp smaller than
  \begin{equation*}
    2(d-j)c_3c_4\delta (\log n)^{j+2}.
  \end{equation*}
  Note that each $(k+1)$-set can contain at most $\binom{k+1}{j+1}\le \binom{d+1}{j+1}=:c_5$
  petals, and therefore each of these dangerous
  sets makes at most $c_5$ copies of $\Mjkzerohat$ disappear by becoming a simplex. If we choose
  $\delta<\frac{c_1}{2(d-j)c_3c_4c_5}$, then whp the number of copies of $\Mjkzerohat$
  that disappear by time $\tim=\delta/(\log n)$ is at most
  \begin{equation*}
    2(d-j)c_3c_4c_5\delta (\log n)^{j+2} < c_1(\log n)^{j+2} \le \varMjkzerohat.
  \end{equation*}
  In other words, at least one copy of $\Mjkzerohat$ that exists
  at the beginning of the interval survives until the end of the interval.
\end{proof}

\begin{lem}\label{lem:secondinterval}
  For every constant $\delta>0$ and every critical dimension $k$ with
  $j\le k\le d$, whp there is a
  copy of $\Mjkhat$ that is present in the process $(\cG_{\tim}) =
  (\cG(n,\tim\barallp))_\tim$ for all values
  \begin{equation*}
    \tim \in \left[\frac{\delta}{\log n},1-\frac{1}{(\log n)^{1/3}}\right].
  \end{equation*}
\end{lem}

\begin{proof}
  By Lemma~\ref{lem:manyNjk} with $\omega_0=(\log n)^{1/3}$,
  whp there are more than $\exp(\sqrt{\log n})$
  many copies of $\Njk$ in $\cG_{\tim}$ at the upper end $\tim =
  1-\frac{1}{(\log n)^{1/3}}$ of the interval. Observe that only
  $\Theta(1)$ such copies can share the same $k$-simplex $K$, thus
  we have $\Omega(\exp(\sqrt{\log n}))$ many copies with distinct
  $k$-simplices. For each such copy $(K,C)$, the scaled birth time  
  of $K$  is uniformly distributed within
  $$\left[0,\min\left\{1-\frac{1}{(\log n)^{1/3}},\frac{1}{\bar p_k}\right\}\right],$$
  meaning that
  $(K,C)$ formed a copy of $\Njk$ at time $\tim = \frac{\delta}{\log n}$
  with probability
  $$
  \frac{\delta/(\log n)}{\min\left\{1-\frac{1}{(\log n)^{1/3}},\frac{1}{\bar p_k}\right\}}
  \ge \frac{\delta}{\log n}.
  $$

  The birth times of the simplices $K$ are independent, thus the
  probability that at least one of them was present at time $\tim = 
  \frac{\delta}{\log n}$ is at least
  \begin{equation*}
    1-\left(1-\frac{\delta}{\log n}\right)^{\Theta\left(\exp(\sqrt{\log n})\right)}
    \ge 1-\exp\left(-\Theta\left(\frac{\exp(\sqrt{\log n})}{\log n}\right)\right)
    = 1-o(1).
  \end{equation*}
  In other words, whp some copy $(K,C)$ of $\Njk$ that exists at time $\tim =
  1-\frac{1}{(\log n)^{1/3}}$ already existed at time $\tim =
  \frac{\delta}{\log n}$. 
  By Corollary~\ref{cor:NjkMjk}
  applied at time 
  $\tim = \frac{\delta}{\log n}$, whp there exist $w\in K\setminus C$ and $a\in [n]\setminus K$
  such that $(K,C,w,a)$ is a copy of $\Mjkhat$
  in $\cG_\tim$ and therefore 
  throughout the interval
  $\left[\frac{\delta}{\log n},
  1 - \frac{1}{(\log n)^{1/3}} \right]$, as
  claimed.
\end{proof}

\begin{lem}\label{lem:thirdinterval}
  Whp the minimal obstruction which vanishes
  at time $\tim_j^*$ (defined in Theorem~\ref{thm:main}) was already present in $(\cG_{\tim}) =
  (\cG(n,\tim\barallp))_\tim$ for all
  values $\tim$ with
  \begin{equation*}
    \tim \in \left[1-\frac{1}{(\log n)^{1/3}},\tim_j^*\right).
  \end{equation*}
\end{lem}

\begin{proof}
  Recall that by Definition~\ref{def:indices}, the last minimal
  obstruction to vanish is a copy $(K,C,w,a)$ of $\Mjlhat$.
  Similar to the proof of Lemma~\ref{lem:secondinterval},
  the birth time of $K$ is uniformly
  distributed within $\left[0,\min \{\tim_j^*,1/\bar p_\ell\}\right)$ and
  by Lemma~\ref{lem:hitting}, whp $\tim_j^*\le 1+\frac{1}{(\log n)^{1/3}}$.
  Conditioned on this high probability event, the probability that
  $K$ already existed as a simplex in $\cG_\tim$ at time $\tim=1-\frac{1}{(\log n)^{1/3}}$ is at least
  $$
  \frac{1-\frac{1}{(\log n)^{1/3}}}{1+\frac{1}{(\log n)^{1/3}}}=1-o(1),
  $$
  i.e.\ whp $(K,C)$ formed a copy of $\Njk$ already at time
  $\tim = 1-\frac{1}{(\log n)^{1/3}}$. By Corollary~\ref{cor:NjkMjk}, this
  means that whp there exist $\tilde w,\tilde a$ such that
  $(K,C,\tilde w,\tilde a)$ forms a copy of $\Mjlhat$ throughout the
  interval.
\end{proof}

\begin{proof}[Proof of
Lemma~\ref{lem:jinterval}]
Lemmas~\ref{lem:firstinterval},
\ref{lem:secondinterval}, and~\ref{lem:thirdinterval} imply that
whp for any $\tim \in [\eps/n, \tim_j^*)$ a copy
of $\Mjkhat$ (for some $j\le k\le d$)
exists in $\cG_\tim$. Therefore,
for any $\tim$ in this range
$H^j(\cG_\tim;R) \neq 0$
by Corollary~\ref{cor:obstruction}.
\end{proof}

\bigskip
\section{Critical and supercritical cases: proof of Lemma~\ref{lem:supercritical}}\label{sec:supercritical}

In this section we present some auxiliary results
and prove
Lemma~\ref{lem:supercritical}, which
we used in Section~\ref{sec:proofmainthm} to show that
whp the process $(\cG_{\tim}) =
  (\cG(n,\tim\barallp))_\tim$ is
\connected\ for all $\tim\ge \tim_j^*$
(Theorem~\ref{thm:main}~\ref{main:supercritical}).
Furthermore, the results of this section
will be fundamental for the proof
of Theorem~\ref{thm:criticalwindow} (Section~\ref{sec:proofofcritwind}).

Recall that in order to have
$H^j(\cG_\tim;R)$ not vanishing, $\cG_\tim$ would have
to admit a \emph{bad function}, i.e.\ 
 a $j$-cocycle that is not
a $j$-coboundary.
We aim to show that no bad function exists by considering what
such a function with smallest possible support might look like, if it exists.
We show that the support must be \emph{traversable}
(Definition~\ref{def:traversable}, Lemma~\ref{lem:traversability}),
and then use this property to show that whp the support cannot be small (Lemma~\ref{lem:smallsupp}).
Subsequently, we use traversability and a result of Meshulam and Wallach~\cite{MeshulamWallach08}
to show that whp the support cannot be large (Lemma~\ref{lem:largesupp}), which is a contradiction.

However, so far this only proves that for any $\tim \ge \tim_j^*$, whp 
$H^j(\cG_\tim;R)=0$.
We need to know that whp, for any $\tim \ge \tim_j^*$ the group
$H^j(\cG_\tim;R)$ vanishes
(i.e.\ with a different order of quantifiers).
We achieve this by observing that 
at time $\tim_j^*$ the $j$-th
cohomology group is zero,
and proving that whp no new bad functions can appear (Lemma~\ref{lem:nonewobstructions}).

Slightly more generally than described above, we will actually prove that
for $\tim$ large enough, but slightly smaller than $\tim_j^*$, the only bad
functions that exist are the result of copies of $\Njk$ existing.

\begin{definition}
Let $(K,C)$ be a copy of $\Njk$ in a $d$-complex $\cG$. We say that a $j$-cochain $f \in C^j(\cG)$ \emph{arises from $(K,C)$} if its support $\cS=\mathrm{supp}(f)$ is such that
\[ \cS = \mathcal{F}(K,C).  \] 

We say that a $j$-cocycle $f$ (i.e.\ $f \in \ker(\delta^j)\subseteq C^j(\cG)$) is
\emph{generated by copies of $\Njk$} if it belongs to the same cohomology class
as some $f_1 + f_2 + \ldots + f_m$, where each $f_i$ is a $j$-cocycle that arises
from a copy of $M_{j,k_i}$.
We denote by $\cN_{\cG}$ the set of $j$-cocycles in $\cG$ that are
\emph{not} generated by copies of $\Njk$. If $\cG=\cG_\tim=\cG(n,\tim\barallp)$,
we will ease notation by defining $\cN_\tim \ :=\ \cN_{\cG_\tim}$.
\end{definition}

The goal is to prove that whp for all $\tim \ge \tim_j^*$ we have $\cN_\tim = \emptyset$. 
Since in this range there are no copies of $\Mjkhat$ and thus by Corollary~\ref{cor:NjkMjk} whp also no copies of $\Njk$, 
this will imply that whp $H^j(\cG_{\tim};R)=0$.  To this end, we need the following notation.
\begin{definition}
For every $\tim$, we denote by $f_\tim$ a function in $\mathcal{N}_\tim$ with smallest support $\cS_\tim$, if such a function exists.
\end{definition}

In order to bound the number of possible such supports $\cS_\tim$, we first show
(Lemma~\ref{lem:traversability}) that $\cS_\tim$ must satisfy the following
concept of \emph{traversability}.
\begin{definition}\label{def:traversable}
Let $(\cS,\cT)$ be a pair where $\cS$ is a collection of $j$-simplices in $\cG_{\tim}$ and $\cT$ is a collection of simplices in $\cG_\tim$ of dimensions $j+1,\ldots, d$. 

We say that $\cS$ is \emph{$\cT$-traversable} if it 
\emph{cannot} be partitioned into two non-empty subsets such that every simplex of $\cT$
contains elements of $\cS$ in at most one of the two subsets.
Equivalently, $\cS$ is $\cT$-traversable if for every $J,J'\in \cS$ there exists a sequence
$J=J_0,J_1,\ldots,J_m=J'$ of $j$-simplices in $\cS$ and a sequence $K_1,\ldots,K_m$ of
simplices in $\cT$ (not necessarily all of the same dimension) such that
$(J_{i-1}\cup J_i) \subseteq K_i$ for all $i \in [m]$.
We may think of these sequences of simplices as a generalisation of a path between $J$ and $J'$,
and thus traversability may be considered a form of connectedness.

We say that $\cS$ is \emph{traversable}\label{comb:traversability} (in $\cG_\tim$)
if it is $\cT$-traversable with $\cT$ consisting of \emph{all} $k$-simplices of
$\cG_{\tim}$ for every $j+1\le k \le d$.
\end{definition}

\begin{lem}\label{lem:traversability}
For every $\tim$, the support $\cS_\tim$, if it exists, is traversable.
\end{lem}

\begin{proof}
Suppose $\cS_\tim$ is not traversable and let $\cS_\tim = \cS_{(1)} \mathop{\dot{\cup}} \cS_{(2)}$
be a partition
into non-empty parts such that (in particular) each $(j+1)$-simplex of $\cG_\tim$
contains elements of $\cS_\tim$ in at most one of the two parts. 

For $i=1,2$, let $f_{(i)}$ be the $j$-cochain  defined by
\[
f_{(i)}(\sigma) = 
\begin{cases}
f_\tim(\sigma) & \mbox{if } \sigma \in \cS_{(i)}, \\
0_R & \mbox{otherwise}.
\end{cases}
\]
Suppose that $\rho$ is a $(j+1)$-simplex that contains $j$-simplices from only one $\cS_{(i)}$,
without loss of generality from $\cS_{(1)}$ and not $\cS_{(2)}$. 
Then trivially $(\delta^jf_{(2)})(\rho) = 0$ and $(\delta^jf_{(1)})(\rho) = (\delta^j f_\tim)(\rho) = 0$,
because $f_\tim \in \ker \delta^j$. 
Thus both functions $f_{(i)}$ are $j$-cocycles, and neither of them lies in $\cN_\tim$
by the minimality of $\cS_\tim$. 
Hence $f_\tim = f_{(1)} + f_{(2)}$ is generated by copies of $\Njk$, since this property is closed under summation, a contradiction to  $f_\tim \in \cN_\tim$.
\end{proof}

It is clear that given a traversable $\cS$ in $\cG_\tim$,
there exists a minimal collection $\cT$ of simplices of $\cG_\tim$
such that $\cS$ is $\cT$-traversable and every $\sigma \in \cT$
has scaled birth time at most $\tim$.
We fix some such minimal collection and denote it by  $\cT(\cS)$. We also define the sequence 
$\allt(\cS)=(t_{j+1},\ldots,t_d)$ 
where $t_k\ge 0$ is the number of $k$-simplices in $\cT(\cS)$, for every $j+1\le k \le d$.

\begin{remark}\label{rem:exploration}
  We note that if $\cS$ is traversable, $(\cS,\cT(\cS))$ can be explored
  in a natural way using
  a breadth-first search process: start from some $j$-simplex in $\cS$
  and reveal all $k$-simplices of $\cT(\cS)$ containing it, for every
  $k \in \{j+1,\ldots,d\}$. We thus `discover' any further
  $j$-simplices of $\cS$ contained in these simplices of $\cT(\cS)$,
  and from each of these $j$-simplices in turn we repeat the process.
  The $\cT(\cS)$-traversability of $\cS$ implies that all
  $j$-simplices of $\cS$ (and also all simplices of $\cT(\cS)$) are
  discovered in this process.
\end{remark}

This viewpoint allows us to observe some important properties.
First, note that by the minimality of $\cT(\cS)$, every simplex of $\cT(\cS)$ must
contain a previously undiscovered $j$-simplex of $\cS$, and therefore
\begin{equation}\label{eq:slowerbound}
|\cT(\cS)| = \sum\nolimits_{k=j+1}^d t_k \le |\cS|.
\end{equation}
On the other hand, each $k$-simplex of $\cT(\cS)$ is discovered
from a $j$-simplex it contains, and therefore contains at most
$k-j$ previously undiscovered vertices. Thus if $v$ is the number of vertices that
are contained in some $j$-simplex of $\cS$, we have
\begin{equation}\label{eq:verttrav}
v \leq (j+1) + \sum\nolimits_{k=j+1}^d (k-j)t_k.
\end{equation}

In the next lemma we show that at around $\tim=1$, while we may have copies of
$\Njk$, whp there are no `small' traversable supports of $j$-cocycles
other than those arising from these $\Njk$.

\begin{lem}\label{lem:smallsupp}
Let $\tim=1+o(1)$ and let $\largecon \in \mathbb{R}^+$ be a constant. 
Then whp there is no $j$-cocycle in $\cG_{\tim}$ with traversable support of size $s\le \largecon$, apart from those arising from copies of $\Njk$.

In particular, whp $|\cS_\tim|>\largecon$, if it exists.
\end{lem}

\begin{proof}
We want to bound the expected number of pairs $(\cS,\cT(\cS))$,
where $\cS$ is a traversable support of a $j$-cocycle not arising from a copy of $\Njk$
and with size $s\le \largecon$.

Recall that $v$ is the number of vertices that are contained in some $j$-simplex of $\cS$
and $\allt(\cS) = (t_{j+1},\ldots,t_d)$ is such that $t_k\ge 0$ indicates
the number of $k$-simplices in $\cT(\cS)$.
By~\eqref{eq:slowerbound} we have 
\begin{equation}\label{eq:simpltrav}
\sum\nolimits_{k=j+1}^d t_k \le s \le \largecon.
\end{equation}
Since $\cS$ is the support of a $j$-cocycle, by Lemma~\ref{lem:minobst}~\ref{minobst:size} if a $k$-simplex
contains an element in $\cS$ then all its $k+1$ vertices are contained in some $j$-simplex of $\cS$. 
This means that the $s\binom{n-v}{k-j}$ many $(k+1)$-sets containing a $j$-simplex in $\cS$
and $k-j$ vertices not in any $j$-simplex of $\cS$ are not allowed to be simplices. 
We thus obtain that the probability that a fixed pair $(\cS,\cT(\cS))$ has all the necessary properties
(in terms of which simplices exist and which do not) is bounded from above by
\begin{align*}
& \prod\nolimits_{k=j+1}^{d} p_k^{t_k} (1-p_k)^{s\binom{n-v}{k-j}}\\
& \hspace{2.5cm} = \prod\nolimits_{k=j+1}^{d} \big((1+o(1))\bar p_k\big)^{t_k} \exp \left(-(1+o(1))\bar p_k\left(s \frac{n^{k-j}}{(k-j)!}\right)\right)\\
& \hspace{2.5cm} =O\left( \prod\nolimits_{k=j+1}^{d} \left( n^{-(k-j+\bar\gamma_k)+o(1)}\right)^{t_k} \exp \left(-(1+o(1))s \left( \bar\alpha_k \log n + \frac{\bar\beta_k}{n^{\bar\gamma_k}}  \right) \right ) \right)\\
& \hspace{2.5cm} = O \left( n^{- \sum_{k=j+1}^{d} \big((k-j+\bar\gamma_k)t_k + s \bar\alpha_k \big) +o(1)}\right),
\end{align*}
where we used the observation that
$\bar \beta_k$ can only be negative if $\bar \alpha_k\neq 0$, in which case $\bar \beta_k=o(\log n)$ and $\bar \gamma_k=0$
(see \ref{parameters:subpoly}).

Let $\allt = (t_{j+1},\ldots,t_d)$ and denote by $E_{s,v,\allt}$ the event that a pair
$(\cS,\cT(\cS))$ with $\cS$ a traversable support of size $s$ on $v$ vertices and $\allt(\cS)=\allt$ exists.
Equations~\eqref{eq:verttrav} and~\eqref{eq:simpltrav} together imply that $v \le j+1 + (d-j)h = O(1)$,
and therefore 
there are $O(n^v)$ ways of choosing such a pair $(\cS,\cT(\cS))$, meaning that
\[\Pr\left( E_{s,v,\allt} \right) = O\left( n^{v - \sum_k(k-j) t_k - \sum_k (\bar\gamma_k t_k + s \bar\alpha_k) + o(1)} \right).\]
 
By \eqref{eq:verttrav}, we have 
\[v - \sum\nolimits_{k=j+1}^d (k-j)t_k \leq j+1.\] 
Moreover, for an index $i$ such that $t_i\ge 1$ (such an index exists,
because otherwise the support would be empty), 
by Lemma~\ref{lem:minobst} and since the considered $j$-cocycle does not arise from a copy of $\Njk$, it holds that $s\ge i-j+2$. 
Recalling that 
\[\bar\logp_i = j+1 - \bar\gamma_i - (i-j+1)\sum\nolimits_{k=j+1}^{d} \bar\alpha_k \stackrel{\text{\ref{crit:less},\ref{crit:equal}}}{\le}0\] 
and that $\sum_{k=j+1}^d \bar\alpha_k >0$, we have
\begin{equation*}
\sum\nolimits_{k=j+1}^d (\bar\gamma_k t_k + s \bar \alpha_k ) \ge \bar\gamma_i + (i-j+2)\sum\nolimits_{k=j+1}^d \bar\alpha_k 
= j+1 - \bar\logp_i + \sum\nolimits_{k=j+1}^d \bar\alpha_k > j+1.
\end{equation*}
Thus, 
\[\Pr\left( E_{s,v,\allt} \right) =o(1).\]
Since by \eqref{eq:verttrav} and \eqref{eq:simpltrav} there are only constantly many choices for the values $s$, $v$, and $\allt$, the probability that any such pair $(\cS,\cT(\cS))$ exists is $o(1)$.
\end{proof}

For supports of larger sizes, we will need a lower bound on the number
of $(j+2)$-sets
that are not allowed to be $(j+1)$-simplices in $\cG_\tim$. Such a bound is given by
Meshulam and Wallach \cite[Proposition~3.1]{MeshulamWallach08},
where it was stated for the case when
the cohomology groups considered are over any \emph{finite} abelian group $R$.
We observe however, that the proof still works without the additional condition that $R$ is finite and
we include this proof in Appendix~\ref{sec:proofmeshwal}
for completeness.

\begin{prop}[{\cite[Proposition~3.1]{MeshulamWallach08}}]\label{prop:meshwal}
  Let $n\ge j+2$ and let $\Delta$ be the downward-closure of the 
  $(n-1)$-simplex on $[n]$. Let $f \in C^j(\Delta;R)$ have support 
  $\cS$ and suppose that any other $j$-cochain of the form 
  $f+\delta^{j-1}g$, where $g\in C^{j-1}(\Delta;R)$, has support of
  size at least $|\cS|$. Denote by $\cD(f)$ the support of 
  $\delta^jf$, i.e.\ all $(j+1)$-simplices in $\Delta$ such that for
  some ordering $[v_0,\ldots,v_j]$ (and thus for all orderings) it 
  holds that $\sum_{i=0}^j (-1)^i f[v_0,\ldots,\hat{v}_i,\ldots,v_j]
  \neq 0_R$. Then
  \begin{equation*}
    |\cD(f)| \geq \frac{n}{j+2} |\cS|.
  \end{equation*}
\end{prop}

The following lemma shows that whp in the supercritical case a smallest support in $\cN_{\tim}$ cannot be `large'.

\begin{lem} \label{lem:largesupp}
There exists a positive constant $\tilde{\largecon} \in \mathbb{R}^+$ such that
whp for all $\tim \geq \tim_0 \ :=\  1 - \frac{1}{(\log n)^{1/3}}$
we have $|\cS_\tim|<\tilde{\largecon}$ (if $\cS_\tim$ exists). 
\end{lem}

\begin{proof}
We first note that, if $\tim$ is large enough to imply that there
exists a $k\in \{j+1,\ldots,d\}$ with $p_k =\tim \bar p_k \ge 1$, then
the result is trivial: all $k$-simplices are present and there is no
bad function, so $\cS_\tim$ does not exist. We will therefore assume
for the remainder of the proof that $\tim_0 \le \tim \le \min_{k\in 
\{j+1,\ldots,d\}} 1/\bar p_k$, and in particular $p_k\le 1$ for each
$k\ge j+1$.

By Lemma~\ref{lem:traversability} if $\cS_\tim$ exists it is
traversable. Consider a pair $(\cS,\cT(\cS))$ where $\cS$ is a
traversable support of size $s$, and thus can be found via the search
process described in Remark~\ref{rem:exploration}. For such a pair, we
define the \emph{exploration matrix} $B=(b_{i,k})$ for $i \in [s]$ and
$k \in \{j+1,\ldots,d\}$, where $b_{i,k}\ge 0$ is the number of
$k$-simplices of $\cT(\cS)$ we discover from the $i$-th $j$-simplex of
$\cS$ in the search process. 

For a fixed exploration matrix $B$, we can bound the number of ways
the exploration process can proceed, thus bounding the number of pairs
$(\cS,\cT(\cS))$ with exploration matrix $B$.
\begin{claim}\label{largesupp:numberpairs}
  The number of pairs $(\cS,\cT(\cS))$ in which $\cS$ is traversable
  and has exploration matrix $B$ is at most
  \begin{equation*}
    n^{j+1} \frac{\prod_k \left( \binom{n}{k-j} 2^{\binom{k+1}{j+1}}\right)^{t_k}}{\prod_{i,k} b_{i,k}!}.
  \end{equation*}
\end{claim}
\noindent
We postpone the proof of Claim~\ref{largesupp:numberpairs} to
Appendix~\ref{sec:proofclaim:numberpairs}.

Claim~\ref{largesupp:numberpairs} provides an upper bound for the
number of sets that are candidates for $\cS_\tim$. Our next aim is
to bound the probability that a fixed function $f$ with traversable
support $\cS$ turns out to be $f_\tim$. To this end, we would like to
know that for $k \in \{j+1, \ldots, d\}$, the number of $(k+1)$-sets that are
not allowed to be $k$-simplices in $\cG_\tim$ in order for $f_\tau$ to
be a $j$-cocycle is `large'. We will prove this by applying
Proposition~\ref{prop:meshwal} to $f_\tim$ (if it exists).

In order to see that $f_\tim$ (if it exists) satisfies the hypothesis
of Proposition~\ref{prop:meshwal}, we will need the following
auxiliary result, which we prove in Appendix~\ref{sec:proofclaim:meshwal}
by a simple first moment argument.

\begin{claim}\label{largesupp:meshwal}
  Whp every $\cG_\tim$ with $\tim\ge\tim_0$ has full $(j-1)$-skeleton.
\end{claim}
\noindent
For the rest of the proof, we condition on the high probability event
in Claim~\ref{largesupp:meshwal}.

This means that, with $\Delta$ as in Proposition~\ref{prop:meshwal},
we have $C^{j-1}(\Delta;R) = C^{j-1}(\cG_\tim;R)$. Now let $g$ be a 
$(j-1)$-cochain in $\Delta$ (and thus also in $\cG_\tim$). Consider 
$f_\tim + \delta^{j-1}g$ as a $j$-cochain $f$ in $\cG_\tim$ and as
a $j$-cochain $f_\Delta$ in $\Delta$. Clearly, the support of
$f_\Delta$ contains the support of $f$, which in turn has size at
least $|\cS_\tim|$, because $\cS_\tim$ is minimal. In particular,
$f_\tim$ satisfies the hypothesis of Proposition~\ref{prop:meshwal}.
This enables us to prove that there are `many' $(k+1)$-sets that are
not allowed to be $k$-simplices in $\cG_\tim$ in order for $f_\tim$
to be a $j$-cocycle.

\begin{claim}\label{largesupp:badsets}
  For $j+1 \le k \le d$ and a $j$-cocycle $f$, denote by $\cD_k(f)$
  the set of $(k+1)$-sets in $[n]$ that contain at least one element
  of $\cD(f)$. There exists a positive constant $\largecon_0$ such
  that, if $f_\tim$ exists, then for all $k$ with $j+1 \le k \le d$,
  \begin{equation*}
    |\cD_k(f_\tim)| \ge \largecon_0 |\cS_\tim| n^{k-j}.
  \end{equation*}
\end{claim}
\noindent
The proof of Claim~\ref{largesupp:badsets} is an easy counting 
argument, which we present in Appendix~\ref{sec:proofclaim:badsets}.

The elements in $\cD_k(f_\tim)$ are not allowed to be $k$-simplices
in $\cG_\tim$ and thus all must have
birth time larger than $\tim$.
We can use this fact, together with the upper bound for the number of
`candidates' for $\cS_\tim$, to prove that $\cS_\tim$ existing and
having a fixed exploration matrix is unlikely.

\begin{claim}\label{largesupp:rB}
  There exist positive constants $\largecon_1,\tilde{\largecon}_1$
  such that the following holds for every $s \ge \tilde{\largecon}_1$.
  For every $(s\times(d-j))$-matrix $B$, the probability $r_B$ that
  $f_\tim$ exists and $\cS_\tim$ has size $s$ and exploration matrix
  $B$ satisfies
  \begin{equation*}
    r_B \le \frac{n^{- \largecon_1 s}}{\prod_{i,k}b_{i,k}!}\,.
  \end{equation*}
\end{claim}
\noindent
We prove Claim~\ref{largesupp:rB} in Appendix~\ref{sec:proofclaim:rB}.

We determine a lower bound for the denominator $\prod\nolimits_{i,k}b_{i,k}!$ and sum over all possible exploration matrices
$B$, so as to deduce that whp $\cS_\tim$ cannot be larger than a
given large constant.

\begin{claim}\label{largesupp:singleh}
  There exists a positive constant $\largecon_2$ such that for every
  $\tilde{\largecon}\ge\tilde{\largecon}_1$, the probability that 
  $\cS_{\tim}$ exists with $|\cS_{\tim}|\ge\tilde{\largecon}$ is at most
  \begin{equation*}
    n^{-\largecon_2 \tilde{\largecon} / 2}\,.
  \end{equation*}
\end{claim}
\noindent 
We delay the proof of this claim until Appendix~\ref{sec:proofclaim:singleh}.

Claim~\ref{largesupp:singleh} implies that for \emph{a fixed time}
$\tim \ge \tim_0$, whp $|\cS_\tim| < \tilde{\largecon}$ (if it exists).
Our aim, however, is to prove that whp $|\cS_\tim| < \tilde{\largecon}$
\emph{simultaneously for all $\tim \ge \tim_0$.} To this end, observe
that there are $O(n^{d+1})$ simplices $\sigma$ with scaled birth times
$\tim_\sigma \ge \tim_0$. Taking a union bound over all these scaled
birth times, we deduce from Claim~\ref{largesupp:singleh} that
\begin{equation*}
  \Pr(\exists\tim\ge\tim_0 \text{ such that } \cS_\tim \text{ exists 
  and } |\cS_\tim| \ge \tilde{\largecon}) =
  O(n^{d+1-\largecon_2 \tilde{\largecon} / 2}).
\end{equation*}
Thus, Lemma~\ref{lem:largesupp} holds for $\tilde{\largecon} > 
\max\left\{\tilde{\largecon}_1,\frac{2(d+1)}{\largecon_2}\right\}$.
\end{proof}

Lemma~\ref{lem:largesupp} implies that whp traversable supports of $j$-cocycles of `large' size do not exist in the whole supercritical range. 
For supports of constant size, this is given by Lemma~\ref{lem:smallsupp} only for $\tim = 1+o(1)$. 
We therefore derive the following result, stating that for $\tim$ 
`close' to $1$ whp every $j$-cocycle is generated by 
copies of $\Mjk$.

\begin{cor}\label{cor:Ntauemptycritwind}
For every $\tim = 1 + O\left(\frac{1}{\log n} \right)$, 
we have $\cN_{\tim} = 
\emptyset$ whp.
\end{cor}
\begin{proof}
For every such $\tim$ Lemma~\ref{lem:largesupp} holds
and thus there exists a constant
$\tilde h$ such that  either $\cS_\tim$
does not exist or $|\cS_\tim|\le \tilde h$. 
But by Lemma~\ref{lem:smallsupp} (with $h=\tilde h$),
whp if $\cS_{\tim}$ exists then $|\cS_\tim|>\tilde h$.
Thus whp $\cS_{\tim}$ does not exist, i.e.\
$\cN_{\tim} = 
\emptyset$.
\end{proof}

To exclude the existence of `small' supports
throughout the entire supercritical case, we show that if a new obstruction appears,
then the simplex whose addition to the complex creates the obstruction
in fact forms a local $j$-obstacle, which whp does not exist in this range by Lemma~\ref{lem:obstacles}.

\begin{lem}\label{lem:nonewobstructions}
Let $K$ be the simplex with smallest scaled birth time $\tim_K \ge \tim_j^*$
such that $\cN_{\tim_K} \neq \emptyset$ (if it exists).
Then whp $K$ forms a local $j$-obstacle in $\cG_{\tim_K}$.
\end{lem}

\begin{proof}
First observe that by Lemma~\ref{lem:hitting} and 
Corollary~\ref{cor:Ntauemptycritwind}, whp 
$\cN_{\tim_j^*} = \emptyset$, and thus
whp $\tim_K > \tim_j^*$. For the rest of
this proof, we condition on this high probability
event.

Suppose now that $|K|=k+1$ and let $\tim\ge \tim_j^*$ be such that $\cG_{\tim_K} = \cG_\tim + K$.
If $\cS_{\tim_K} \cap \cG_\tim \neq \emptyset$, let $\cS$ be a maximal subset of $\cS_{\tim_K}$
which is traversable in $\cG_{\tim}$ and let $f$ be the $j$-cochain in $\cG_{\tim}$ defined by 
\[
f(\sigma) = 
\begin{cases}
f_{\tim_K}(\sigma) & \mbox{if } \sigma \in \cS; \\
0_R & \mbox{otherwise}.
\end{cases}
\]
Then $f$ is a $j$-cocycle in $\cG_\tim$ because every $i$-simplex of $\cG_\tim$,
for $i=j+1,\ldots,d$, containing some element of $\cS$ cannot contain other $j$-simplices
in $\cS_{\tim_k} \setminus \cS$ by the maximality of $\cS$, and because $f_{\tim_K}$ is a $j$-cocycle.

Moreover, by Lemma~\ref{lem:largesupp} there exists a positive constant $\tilde{\largecon}$
such that whp $|\cS|\leq |\cS_{\tim_K}| < \tilde{\largecon}$. 
Lemma~\ref{lem:shells} implies that whp each $j$-simplex of $\cS$ lies in a linear number
of $j$-shells in $\cG_\tim$ and at most $|\cS|-1$ many of them can contain other elements of $\cS$. 
This means that whp there are $j$-shells in $\cG_\tim$ that meet $\cS$ in a single $j$-simplex,
and thus $f$ is a bad function in $\cG_\tim$. 
Since $f$ cannot be generated by copies of $\Njk$,
because $\tim \ge \tim_j^*$ and thus no copies
of $\Njk$ exist,
 this yields $\cN_\tim \neq \emptyset$,
a contradiction to the choice of $K$.

Hence, whp the $j$-simplices of $\cS_\tim$ are all contained in $K$ and are not in other simplices of
$\cG_{\tim_K}$. Then whp $K$ forms a local $j$-obstacle in $\cG_{\tim_K}$,
because $|\cS_{\tim_K}| \ge k-j+1$ by Lemma~\ref{lem:minobst}.
\end{proof}

\begin{cor} \label{cor:N_tauisempty}
Whp, for all $\tim \ge \tim_j^*$ we have
$\cN_{\tim} = \emptyset$.
\end{cor}
\begin{proof}
By Lemma~\ref{lem:hitting} 
and Corollary~\ref{cor:Ntauemptycritwind} whp $\cN_{\tim_j^*}$ is empty. 
If there is $\tim > \tim_j^*$ such that $\cN_{\tim} \neq \emptyset$, then Lemma~\ref{lem:nonewobstructions} tells
us that the simplex whose birth creates a $j$-cocycle
that is
not generated by copies of $\Mjk$ would create a local
$j$-obstacle. But by Lemma~\ref{lem:hitting} whp
$\tim > \tim_j^* \ge \tim'= 1 - \frac{\log \log n}{10 d \log n}$,
thus by Lemma~\ref{lem:obstacles} 
whp no new local $j$-obstacle can appear in $\cG_\tim$.
\end{proof}

We can now use Corollary~\ref{cor:N_tauisempty}
to prove Lemma~\ref{lem:supercritical}
\begin{proof}[Proof of
Lemma~\ref{lem:supercritical}]
By the definition of $\tim_j^*$, there are no copies of $\Mjkhat$ in
$\cG_\tim$ for any $\tim\ge \tim_j^*$, so in order for the $j$-th
cohomology group \emph{not} to vanish, $\cN_\tim$ would have to be
non-empty, but this is excluded by Corollary~\ref{cor:N_tauisempty}.
\end{proof}

\bigskip
 \section{Rank in the critical window: proofs of Theorem~\ref{thm:criticalwindow} and Corollary~\ref{cor:probcritwind}}
 \label{sec:proofofcritwind}

In order to prove the Rank Theorem (Theorem~\ref{thm:criticalwindow}),
we first want to describe
the asymptotic joint distribution
of the number of copies
of $\Mjk$  within the critical window.
To this end, we will make use of 
Lemma~\ref{lem:Njkdistribution},
for which we need the following notation. 
Given a sequence $(X_k)_{k\in I}$ of
random variables (for some finite, ordered index set $I$){
we denote by $\mathcal{L} (X_k)$ the probability distribution of $X_k$
and by $\mathcal{L} \left((X_k)_{k\in I} \right)$ the 
joint probability distribution of the sequence $(X_k)_{k\in I}$.
If $X_k = X_k(n)$ for every $k \in I$, 
we say that the sequence $(X_k)_{k\in I}$ \emph{converges in distribution}
to the sequence of random variables $(Y_k)_{k\in I}$ if 
$\Pr\big((X_k)_{k \in I} = (x_k)_{k\in I}\big) \xrightarrow{n\to \infty}  \Pr\big((Y_k)_{k \in I} = (x_k)_{k\in I}\big)$
for every sequence of values $(x_k)_{k\in I}$,
and we write $(X_k)_{k\in I} \xrightarrow{\; \dist \;} (Y_k)_{k\in I}$.

We adopt the convention that $\Po(0)\equiv 0$
and recall the definition of a critical dimension from Definition~\ref{def:criticaldim}. 

\begin{lem} \label{lem:Njkdistribution}
Let $c\in \mathbb{R}$ be a constant and $(c_n)_{n\ge 1}$ be a sequence of real numbers such that
$c_n \xrightarrow{n \to \infty} c$. 
For any $j\leq k \leq d$ define
\[\critm_k \ :=\  
\begin{cases}
\exp (\bar\sublogp_k + \bar\constp_k + c(\bar\gamma_k 
- j -1) ) & \text{if }
k \text{ is a critical dimension}, \\
0 & otherwise
\end{cases}\]
and let $\tim = 1 + \frac{c_n}{\log n}$. Then, 
setting $\mathbf{X}\ :=\  (\varisol, X_{j,j+1},\ldots,X_{j,d})$,
we have
\[ \mathcal{L}(\mathbf{X}) \xrightarrow{\; \dist \;} (\Po(\critm_j),\ldots,\Po(\critm_d)) .\]
\end{lem}
\noindent To prove Lemma~\ref{lem:Njkdistribution} we 
use a multivariate Poisson approximation technique 
from~\cite{BarbourHolstJansonBook}, which will be presented in Appendix~\ref{sec:multipoisson}. 
The proof of Lemma~\ref{lem:Njkdistribution} then
appears in Appendix~\ref{sec:proofofNjkdistr}.
\begin{proof}[Proof of Theorem~\ref{thm:criticalwindow}]
Consider $\mathcal{X}=\sum_{k=j}^d \varNjk$ and define
\[\critm :  = \exp(-c(j+1)) \sum\nolimits_{k\in \critset}
\exp(\bar\sublogp_k + \bar\constp_k + c\bar\gamma_k ),\]
where $\critset=\critset(\barallp,j)$ is
the set of critical dimensions for 
the $j$-critical direction $\barallp$. 
By Lemma~\ref{lem:Njkdistribution}, we have 
\begin{equation} \label{eq:Xconverges}
\mathcal{L}(\mathcal{X}) 
\xrightarrow{\enskip \dist \enskip} \sum_{k=j}^d \Po(\critm_k)
 =
\Po\left(  \sum_{k\in \critset}
\exp\big(\bar\sublogp_k + \bar\constp_k + c(\bar\gamma_k 
- j -1) \big)
\right)= \Po(
\critm).   
\end{equation}

We first show that 
$H^j(\cG_\tim;R) = R^\mathcal{X}$ whp.
Let $M_1,M_2, \ldots,M_{\mathcal{X}}$ denote
the copies of $\Mjk$, for every $j\le k\le d$ that are present in $\cG_\tim$.
By Corollary~\ref{cor:Ntauemptycritwind} we have 
$\cN_\tim = \emptyset$ whp and by Proposition~\ref{prop:fMrarisesfromMjk}~\ref{fMrari:cocycle} 
we know that
the only $j$-cocycles arising from $M_i$ are 
of the form $f_{M_i,r_i}$ with $r_i \in R$. Thus whp
each cohomology class contains an element of the form
$\sum_{i=1}^{\mathcal{X}}f_{M_i,r_i}$ with  
$r_i \in R$, i.e.\ 
whp the set of cohomology classes of those elements
generates $H^j(\cG_\tim;R)$.

We now need to show that if we take two tuples 
$(r_1,\ldots,r_{\mathcal{X}}) \neq 
(r_1',\ldots,r_{\mathcal{X}}')$ with $r_i,r_i' \in R$
for every $i$, then the cohomology classes of 
$\sum_{i=1}^\mathcal{X} f_{M_i,r_i}$ and of 
$\sum_{i=1}^\mathcal{X} f_{M_i,r_i'}$ are distinct.
Note that this is equivalent to showing that if 
$(r_1,\ldots,r_{\mathcal{X}})$ is not the $0_R$-vector,
then $f=\sum_{i=1}^\mathcal{X} f_{M_i,r_i}$
is not in the same cohomology class as the 
$0_R$-function, 
i.e.\ the $j$-cochain $f$ is not a 
$j$-coboundary. 

We first observe that by
Markov's inequality
$\mathcal{X}=\sum_{k=j}^{d} \varNjk = o(n)$, because 
$\EE(\mathcal{X})=O(1)$ by 
Corollary~\ref{cor:expectatcritwind}. 
We further claim that
for each $j\le k\le d$ whp no two copies
 of $\Njk$ share the same $k$-simplex. 
Indeed, for $k=j$ by 
Definition~\ref{def:Njj} all copies of $\Njj$ 
come from different $j$-simplices.
If $k\ge j+1$, for two copies of $\Mjk$
sharing the same $k$-simplex
there are $\binom{n}{k+1}$ ways to choose the common
$k$-simplex and $O\left(\binom{k+1}{j}^2\right)$ ways to
choose the centres of the two flowers. Moreover,
these two copies are present in $\cG_\tim$ with 
probability $O(p_k \bq^{\; \tim (2k-2j+1)})$, because 
the common $(k+1)$-set is a $k$-simplex 
in $\cG_\tim$ with
probability $p_k$ and the two 
flowers can share at most one petal, thus in total there are at least $(k-j+1)+(k-j) =
2k - 2j +1$ many $j$-simplices that are petals,
and these satisfy \ref{Njk:flower} with 
probability at
most $(1+o(1)) \bq^{\;\tim (2k-2j + 1)}$ by Proposition~\ref{prop:localisedprob}.

Therefore, for $k\ge j+1$, the expected number of pairs of copies of 
$\Mjk$ with the same $k$-simplex is 
\[ O \left( \binom{n}{k+1} \binom{k+1}{j}^2
 p_k  \bq^{\;\tim (2k-2j+1)}\right) 
 \stackrel{\eqref{eq:expXjk}}{=} 
O\left( \EE(\varNjk) \bq^{\;\tim(k-j)} \right) = O(\bq^{\;\tim(k-j)}),\]
because $\EE(\varNjk)=O(1)$. Furthermore, we have that
\begin{equation} \label{eq:qissmall}
\bq \stackrel{\eqref{eq:q_computing}}{=} O \left( \exp\left(- \sum\nolimits_{i=j+1}^d
\left(\bar\alpha_i \log n+ \frac{\bar\beta_i}{n^{\bar\gamma_i}}\right)\right) \right) =
O(n^{-\bar\alpha_{k_0}/2}) = o(1),
\end{equation}
where we are using that $k_0$ is such that
$\bar\alpha_{k_0}\neq 0$ and $\bar\gamma_{k_0}=0$ 
(see \ref{parameters:zero} and \ref{parameters:k0}
in Definition~\ref{def:criticalp}). 
Since $k\ge j+1$ and $\tim = 1 + o(1)$ we also have that $\bq^{\;\tim(k-j)} = o(1)$,
and thus by Markov's inequality, whp
there exist no such pairs of $\Njk$.

Hence, by condition~\ref{Njk:flower} in 
Definition~\ref{def:Njk}, whp the $f_{M_i,r_i}$
have pairwise disjoint supports,
and in particular, for our choice of
the $r_i$, the support $S$ of $f$ is not empty. 
Pick a $j$-simplex $L \in S$. 
Lemma~\ref{lem:shells} yields that in the 
range of $\tim$ we are considering,
whp $L$ is contained in $\Theta(n)$ many
$j$-shells which meet only in $L$, and therefore
at most $|S| \le \sum_{k=j}^d (k-j+1) \varNjk 
\le (d-j+1)\mathcal{X}=o(n)$ of them can contain
another $j$-simplex in $S$. Thus whp
there exists a $j$-shell that meets the support of $f$
only in $L$, i.e.\ $f$ is not a $j$-coboundary
by Lemma~\ref{lem:shellobstruction}.

We therefore have $H^j(\cG_\tim;R) 
= R^{\mathcal{X}}$ whp.
Since $\mathcal{L}(\mathcal{X}) \xrightarrow{\;\text{d}\;} \Po(\critm)$ by \eqref{eq:Xconverges},
there exists a coupling $Y \sim \Po(\critm)$
such that $\mathcal{X} = Y$ whp.
Thus, whp
\[H^j(\cG_\tim;R) 
= R^{Y},\]
as required.
\end{proof}

We now combine Theorems~\ref{thm:main} and~\ref{thm:criticalwindow} 
to prove Corollary~\ref{cor:probcritwind}
\begin{proof}[Proof of Corollary~\ref{cor:probcritwind}]
As in the proof of Theorem~\ref{thm:main},
by applying Lemma~\ref{lem:supercritical} and
Corollary~\ref{cor:iintervals}~\ref{iintervals:timi*} we deduce that for $\tim = 1+\frac{c_n}{\log n}$,
whp $H^i(\cG_\tim;R)=0$ for every $i \in [j-1]$. 
Furthermore,
whp $H^0(\cG_\tim;R)=R$ by Lemma~\ref{lem:topconn}.
This in particular implies that
\[\Pr\left(\cG_\tim \text{ is \connected}\right) = \Pr\left(H^j(\cG_\tim;R) = 0\right) + o(1).\]
Moreover, by Theorem~\ref{thm:criticalwindow} whp 
$H^j(\cG_\tim;R) = R^Y$ with $Y \sim \Po(\critm)$, hence
\[ \Pr\left(\cG_\tim \text{ is \connected}\right) = (1+o(1)) \Pr \big(\Po(\critm)=0\big)
= (1+o(1)) \exp\left( - \critm \right) ,\]
as required.
\end{proof}

\smallskip
\section{Concluding remarks}
\label{sec:concrem}

\subsection{Non-triviality of cohomology groups}
To prove Theorem~\ref{thm:main}~\ref{main:subcritical}, 
our strategy was to show that for every $\eps>0$
and for each $i \in [j]$,
whp $H^i(\cG_\tim;R) \neq 0$ for every $\tim \in I_i(\eps) =
[\eps/n^{j-i+1},\tim_i^*)$, because of the 
existence of copies of $\hat M_{i,k}$ for some
$i \le k \le d$ throughout the interval
$I_i(\eps)$ (Corollary~\ref{cor:iintervals}).
However, it is likely that $H^i(\cG_\tim;R)$
would already be non-trivial
for even smaller $\tim$.
In particular, it would be interesting
to precisely determine from which point
on $H^i(\cG_\tim;R)\neq 0$ whp and in this case
to describe its rank, analogously to
Theorem~\ref{thm:criticalwindow}.

\smallskip
\subsection{Dimension of the last minimal obstruction}
In Theorem~\ref{thm:criticalwindow} we obtain an asymptotic
description of the $j$-th cohomology group
in the critical window. More strongly,
in this regime Lemma~\ref{lem:Njkdistribution} yields
the asymptotic (joint) distribution of the number of
copies of $\Njk$, for every index $k$ with $j\le k\le d$.
This leads to the natural question:
what is (the asymptotic probability distribution of) the dimension 
of the last copy of $\Njk$ that vanishes?

\smallskip
\subsection{Determining the critical dimensions}
For a given $j$-critical direction 
$\barallp = (\bar p_1,\ldots,\bar p_d)$,
it is interesting to determine which
indices $k$ with $j\le k\le d$ 
represent critical dimensions for $\barallp$.
Recall from Definition~\ref{def:criticaldim}
that $k$ is a critical dimension
if $\bar \logp_k \log n+ \bar \sublogp_k + \bar \constp_k
= O(1)$.
The main term of this expression is $\bar \logp_k \log n$;
all other terms are $o(\log n)$. 
The constant $\bar{\lambda}_k$ depends on the parameters
$\bar \gamma_k$, $\bar \alpha_{j+1}$, $\ldots$, $\bar\alpha_d$.
Ignoring the lower order terms, we can therefore 
plot the ranges where $\bar\lambda_k =0$
as a function of those parameters, for
$j\le k\le d$.

We present some examples of these plots
in Figure~\ref{fig:critdim}.
Further examples can be found at:\\
\url{https://www.wolframcloud.com/obj/delgiudice/CriticalDimensions}.

\begin{figure}[htbp] 
  \begin{subfigure}{6cm}
    \centering
    \begin{tikzpicture}[scale=0.611]
     \draw (-4.5,0) node[above]{\tiny $\bar \gamma_1$} -- (4.5,0) node[above] {\tiny $\bar \alpha_1$};
     \draw (0,-4.5) node[right]{\tiny $\bar \gamma_2$} -- (0,4.5) node[right] {\tiny $\bar \alpha_2$};
     
     \draw[domain=-1:0,smooth,variable=\y,gray,opacity=0.6,thick] plot ({-2},{\y});
     \draw[domain=-2:-1,smooth,variable=\x,gray,opacity=0.6,thick] plot ({\x},{2+\x});
     \draw[domain=-1:0,smooth,variable=\x,black,ultra thick] plot ({\x},{2+\x});
     \draw[domain=0:4.5,smooth,variable=\x,black,ultra thick] plot ({\x},{2});
     
     \draw[domain=-1:4.5,smooth,variable=\x,gray,dashed,opacity=0.6,thick] plot ({\x},{-2});
     \draw[domain=-1:4.5,smooth,variable=\x,gray,dashed,opacity=0.6,thick] plot ({\x},{1});
     \draw[domain=-4.5:-1,smooth,variable=\x,black,dashed,ultra thick] plot ({\x},{1});
     
     \fill[opacity=0.35,pattern=north west lines,pattern color=gray]
      (-4.5,-4.5) -- (-4.5,-1) -- (-1,-1) -- (-1,-4.5)-- cycle;
   
     \end{tikzpicture}
    \subcaption{(a)}
  \end{subfigure}
  \begin{subfigure}{6cm}
    \centering
    \begin{tikzpicture}[scale=0.611]
     \draw (-4.5,0) node[above]{\tiny $\bar \gamma_2$} -- (4.5,0) node[above] {\tiny $\bar \alpha_2$};
     \draw (0,-4.5) node[right]{\tiny $\bar \gamma_3$} -- (0,4.5) node[right] {\tiny $\bar \alpha_3$};

     \draw[domain=-1:0,smooth,variable=\y,gray,opacity=0.6,thick] plot ({-2},{\y});
     \draw[domain=-2:-2/3,smooth,variable=\x,gray,opacity=0.6,thick] plot ({\x},{1+\x/2});
     \draw[domain=-2/3:0,smooth,variable=\x,black,ultra thick] plot ({\x},{1+\x/2});
     \draw[domain=0:1,smooth,variable=\x,black,ultra thick] plot ({\x},{1-\x});
     \draw[domain=-4.5:0,smooth,variable=\y,black,ultra thick] plot ({1},{\y});
     
     \draw[domain=-1:0,smooth,variable=\x,gray,dashed,opacity=0.6,thick] plot ({\x},{-2});
     \draw[domain=0:2/3,smooth,variable=\x,gray,dashed,opacity=0.6,thick] plot ({\x},{-2+3*\x});
     \draw[domain=0:2/3,smooth,variable=\x,gray,dashed,opacity=0.6,thick] plot ({\x},{2/3-\x});
     \draw[domain=-2/3:0,smooth,variable=\x,gray,dashed,opacity=0.6,thick] plot ({\x},{2/3});
     \draw[domain=-4.5:-2/3,smooth,variable=\x,black,dashed,ultra thick] plot ({\x},{2/3});
     
     \fill[opacity=0.35,pattern=north west lines,pattern color=gray]
      (-4.5,-4.5) -- (-4.5,-1) -- (-1,-1) -- (-1,-4.5)-- cycle;
    \end{tikzpicture}
    \subcaption{(b)}
  \end{subfigure}
  \caption{Plots of the equations $\bar\logp_k =0$ ($k=1,2$ in~(a)
  and $k=2,3$ in~(b), respectively).
  For each $i\in[k]$, we use the same axis for $\bar\alpha_i$
  and $\bar\gamma_i$, because property~\ref{parameters:zero}
  in Definition~\ref{def:criticalp} states that only one
  of $\bar\alpha_i, \bar\gamma_i$ is non-zero. \newline
  The bold sections denote the ranges for which
  $\barallp$ is a $j$-critical direction and the corresponding $k$ is a critical dimension.
  The striped portions in the lower left
  corners indicate the regions
  in which whp $\cG_\tim$ has
  no simplices of positive dimensions.\newline
  (a) Plots of $\bar\logp_1 =0$ (plain) and
  $\bar\logp_2 =0$ (dashed), for $d=2$, $j=1$.
    Recall that the equation $\bar\logp_1=0$
  refers to copies of $M_{1,1}$,
  i.e.\ isolated $1$-simplices in $\cG_\tim$. \newline
  (b) Plots of $\bar\logp_2 =0$ (plain) and
  $\bar\logp_3 =0$ (dashed), for $d=3$, $j=1$.
  In this case, the plots are under the condition that 
  $\bar p_1=0$.}
  \label{fig:critdim}
\end{figure}
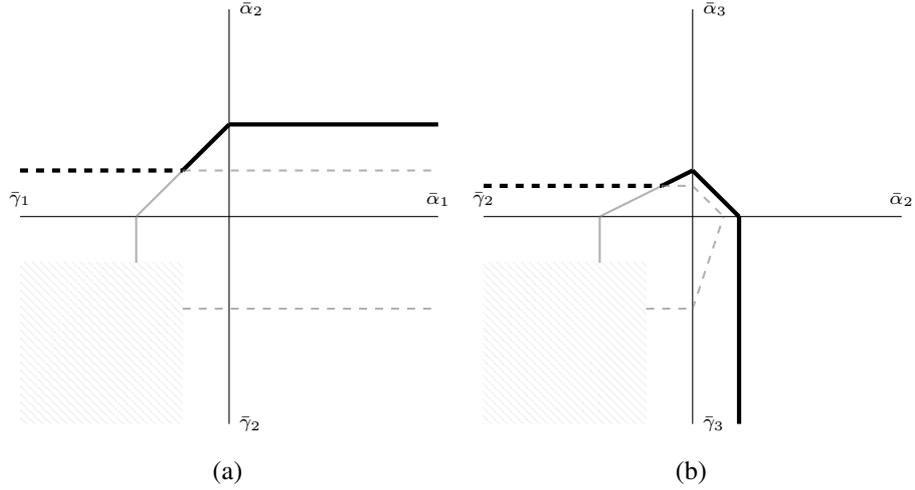

\subsection{Integer homology}
Recently, Newman and Paquette~\cite{NewmanPaquette18} proved a hitting time result for the
$(d-1)$-th homology group over $\mathbb{Z}$ in the Linial-Meshulam model
(the case $d=2$ was previously proved by {\L}uczak and Peled~\cite{LuczakPeled16});
this is a stronger result than for the corresponding cohomology group.
It would be interesting to know whether the analogous result also holds in $\cG_{\tim}$,
i.e.\ does the $j$-th homology group with integer coefficients vanish at the same time as the last copy of $\Mjkhat$ for any $k$ disappears?

\newpage
\bibliographystyle{plain}

\begin{thebibliography}{10}

\bibitem{AllenBoettcherKohayakawaPerson15}
P.~Allen, J.~Böttcher, Y.~Kohayakawa, and Y.~Person.
\newblock Tight Hamilton cycles in random hypergraphs.
\newblock {\em Random Structures Algorithms} 46(3):446--465, 2015.

\bibitem{AllenKochParczykPerson18}
P.~Allen, C.~Koch, O.~Parczyk, and Y.~Person.
\newblock Finding tight {H}amilton cycles in random hypergraphs faster.
\newblock {\em L{ATIN} 2018: {T}heoretical informatics}, 28--36,
Lecture Notes in Comput. Sci., 10807,
{\em Springer, Cham}, 2018.

\bibitem{AronshtamLinial15}
L.~Aronshtam and N.~Linial.
\newblock When does the top homology of a random simplicial complex vanish?
\newblock {\em Random Structures Algorithms}, 46(1):26--35, 2015.

\bibitem{AronshtamLinialLuczakMeshulam12}
L.~Aronshtam, N.~Linial, T.~{\L}uczak, and R.~Meshulam.
\newblock Collapsibility and vanishing of top homology in random simplicial
  complexes.
\newblock {\em Discrete Comput. Geom.}, 49(2):317--334, 2013.

\bibitem{BarbourHolstJansonBook}
A.~D. Barbour, L.~Holst, and S.~Janson.
\newblock {\em Poisson approximation}, volume~2 of {\em Oxford Studies in
  Probability}.
\newblock The Clarendon Press, Oxford University Press, New York, 1992.
\newblock Oxford Science Publications.

\bibitem{BehrischCojaOghlanKang10b}
M.~Behrisch, A.~Coja-Oghlan, and M.~Kang.
\newblock The order of the giant component of random hypergraphs.
\newblock {\em Random Structures Algorithms}, 36(2):149--184, 2010.

\bibitem{BCOK14}
M.~Behrisch, A.~Coja-Oghlan, and M.~Kang.
\newblock Local limit theorems for the giant component of random hypergraphs.
\newblock {\em Combin. Probab. Comput.}, 23(3):331--366, 2014.

\bibitem{BollobasRiordan12c}
B.~Bollob{\'a}s and O.~Riordan.
\newblock Asymptotic normality of the size of the giant component in a random
  hypergraph.
\newblock {\em Random Structures Algorithms}, 41(4):441--450, 2012.

\bibitem{BollobasRiordan16}
B.~Bollob{\'a}s and O.~Riordan.
\newblock Counting connected hypergraphs via the probabilistic method.
\newblock {\em Combin. Probab. Comput.} 25(1):1, 21--75, 2016.

\bibitem{BollobasRiordan17}
B.~Bollob{\'a}s and O.~Riordan.
\newblock Exploring hypergraphs with martingales.
\newblock {\em Random Structures Algorithms}, 50(3):325--352, 2017.

\bibitem{BollobasRiordan18}
B.~Bollob{\'a}s and O.~Riordan.
\newblock Counting dense connected hypergraphs via the probabilistic method.
\newblock {\em Random Structures Algorithms}, 53(2):185--220, 2018.

\bibitem{BollobasThomason85}
B.~Bollob{\'a}s and A.~Thomason.
\newblock Random graphs of small order.
\newblock In {\em Random graphs '83 ({P}ozna\'n, 1983)}, volume 118 of {\em
  North-Holland Math. Stud.}, pages 47--97. North-Holland, Amsterdam, 1985.

\bibitem{CooleyDelGiudiceKangSpruessel20}
O.~{Cooley}, N.~{Del Giudice}, M.~{Kang}, and P.~{Spr\"ussel}.
\newblock Vanishing of cohomology groups of random simplicial complexes.
\newblock {\em Random Structures Algorithms}, 56:461--500, 2020.

\bibitem{CooleyKangKoch16}
O.~Cooley, M.~Kang, and C.~Koch.
\newblock Threshold and hitting time for high-order connectedness in random
  hypergraphs.
\newblock {\em Electron. J. Combin.}, 23(2):Paper 2.48, 14, 2016.

\bibitem{CKKgiant}
O.~Cooley, M.~Kang, and C.~Koch.
\newblock The size of the giant high-order component in random hypergraphs.
\newblock {\em Random Structures Algorithms}, 53(2):238--288, 2018.

\bibitem{CooleyKangPerson18}
O.~{Cooley}, M.~{Kang}, and Y.~{Person}.
\newblock {Largest components in random hypergraphs}.
\newblock {\em Combin.\ Probab.\ Comput.}, pages 1--22, 2018.

\bibitem{CostaFarberHorak15}
A.~Costa, M.~Farber, and D.~Horak.
\newblock Fundamental groups of clique complexes of random graphs.
\newblock {\em Trans. London Math. Soc.}, 2(1):1--32, 2015.

\bibitem{DarlingNorris05}
R.~W.~R. Darling and J.~R. Norris.
\newblock Structure of large random hypergraphs.
\newblock {\em Ann. Appl. Probab.}, 15(1A):125--152, 2005.

\bibitem{ErdosRenyi59}
P.~Erd{\H{o}}s and A.~R{\'e}nyi.
\newblock On random graphs. {I}.
\newblock {\em Publ. Math. Debrecen}, 6:290--297, 1959.

\bibitem{FMN19}
M.~Farber, L.~Mead, and T.~Nowik.
\newblock Random simplicial complexes, duality and the critical dimension.
\newblock to appear in \emph{J.\ Topol.\ Anal.} (2020).

\bibitem{FountoulakisIyerMaillerSulzbach19}
N.~Fountoulakis, T.~Iyer, C.~Mailler, and H.~Sulzbach.
\newblock Dynamical Models for Random Simplicial Complexes.
\newblock ArXiv: 1910.12715.

\bibitem{FountoulakisPrzykucki19}
N.~Fountoulakis and M.~Przykucki.
\newblock High-dimensional bootstrap processes in evolving simplicial complexes.
\newblock ArXiv: 1910.10139.

\bibitem{FountoulakisPrzykucki20}
N.~Fountoulakis and M.~Przykucki.
\newblock Algebraic and combinatorial expansion in random simplicial complexes.
\newblock ArXiv: 2006.09445.

\bibitem{HoffmanKahlePaquette17}
C.~Hoffman, M.~Kahle, and E.~Paquette.
\newblock The threshold for integer homology in random {$d$}-complexes.
\newblock {\em Discrete Comput. Geom.}, 57(4):810--823, 2017.

\bibitem{Kahle07}
M.~Kahle.
\newblock The neighborhood complex of a random graph.
\newblock {\em J. Combin. Theory Ser. A}, 114(2):380--387, 2007.

\bibitem{Kahle09}
M.~Kahle.
\newblock Topology of random clique complexes.
\newblock {\em Discrete Math.}, 309(6):1658--1671, 2009.

\bibitem{Kahle14}
M.~Kahle.
\newblock Sharp vanishing thresholds for cohomology of random flag complexes.
\newblock {\em Ann. of Math.}, 179(3):1085--1107, 2014.

\bibitem{KahleSurvey14}
M.~Kahle.
\newblock Topology of random simplicial complexes: a survey.
\newblock In {\em Algebraic topology: applications and new directions}, volume
  620 of {\em Contemp. Math.}, pages 201--221. Amer. Math. Soc., Providence,
  RI, 2014.

\bibitem{KahlePittel16}
M.~Kahle and B.~Pittel.
\newblock Inside the critical window for cohomology of random {$k$}-complexes.
\newblock {\em Random Structures Algorithms}, 48(1):102--124, 2016.

\bibitem{KaronskiLuczak96}
M.~Karo{\'n}ski and T.~{\L}uczak.
\newblock Random hypergraphs.
\newblock In {\em Combinatorics, {P}aul {E}rd\H{o}s is eighty, {V}ol.\ 2
  ({K}eszthely, 1993)}, volume~2 of {\em Bolyai Soc. Math. Stud.}, pages
  283--293. J\'anos Bolyai Math. Soc., Budapest, 1996.

\bibitem{LinialMeshulam06}
N.~Linial and R.~Meshulam.
\newblock Homological connectivity of random 2-complexes.
\newblock {\em Combinatorica}, 26(4):475--487, 2006.

\bibitem{LinialPeled16}
N.~Linial and Y.~Peled.
\newblock On the phase transition in random simplicial complexes.
\newblock {\em Ann. of Math.}, 184(3):745--773, 2016.

\bibitem{LuczakPeled16}
T.~{\L}uczak and Y.~Peled.
\newblock Integral homology of random simplicial complexes.
\newblock {\em Discrete Comput. Geom.}, 59(1):131--142, 2018.

\bibitem{MeshulamWallach08}
R.~Meshulam and N.~Wallach.
\newblock Homological connectivity of random {$k$}-dimensional complexes.
\newblock {\em Random Structures Algorithms}, 34(3):408--417, 2009.

\bibitem{Munkres84}
J.~Munkres.
\newblock {\em Elements of algebraic topology}.
\newblock Addison-Wesley Publishing Company, Menlo Park, CA, 1984.

\bibitem{NenadovSkoric19}
R.~Nenadov and N.~\v{S}kori\'{c}.
\newblock Powers of Hamilton cycles in random graphs and tight Hamilton cycles in random hypergraphs.
\newblock {\em Random Structures Algorithms} 54(1):187--208, 2019.

\bibitem{NewmanPaquette18}
A.~Newman and E.~Paquette.
\newblock The integer homology threshold in {$Y_d(n,p)$}.
\newblock ArXiv: 1808.10647.

\bibitem{ParczykPerson16}
O.~Parczyk and Y.~Person.
\newblock Spanning structures and universality in sparse hypergraphs.
\newblock {\em Random Structures Algorithms}, 49(4):819--844, 2016.

\bibitem{Poole15}
D.~Poole.
\newblock On the strength of connectedness of a random hypergraph.
\newblock {\em Electron. J. Combin.}, 22(1):Paper 1.69, 16, 2015.

\bibitem{SchmidtShamir85}
J.~Schmidt-Pruzan and E.~Shamir.
\newblock Component structure in the evolution of random hypergraphs.
\newblock {\em Combinatorica}, 5(1):81--94, 1985.

\bibitem{Stepanov70}
V.~E. Stepanov.
\newblock On the probability of the connectedness of a random graph
  {${\mathcal{G}}_{m}(t)$}.
\newblock {\em Theory Probab.\ Appl.}, 15:55--67, 1970.


\end{thebibliography}

\newpage

\appendix

\section{Parametrisation}\label{sec:parameters}
In this appendix
we clarify the parametrisation of $\barallp$
and the assumptions made for $j$-admissibility and $j$-criticality
in Definitions~\ref{def:criticalp} and~\ref{def:criticaldirection}.
Note that the arguments here are independent of the proof of
Theorem~\ref{thm:main}---rather, they justify why the assumptions
made in the theorem are reasonable and cover all interesting cases.

We first justify the parametrisation of $\barallp$ in terms of the 
$\bar\alpha_k,\bar\beta_k,\bar\gamma_k$. We note that scaling
$\barallp$ by a factor $c$ (which may be a function of $n$)
has no effect on the evolution of the process $(\cG_\tim)$,
since $\cG(n,c\tim\barallp) = \cG(n,\tim'\barallp)$,
where $\tim'=c\tim$.
We therefore aim to choose $\barallp$ such that the critical
range for \connectedness\ occurs around time $\tim=1$, i.e.\ when
$\allp=\barallp$.

Observe that the probabilities $p_i$ with $i\in[j-1]$ have no influence
on the $j$-th cohomology group $H^j(\cG_\tim;R)$. To see this,
we note that the $j$-th cohomology group depends only on the set
of $(j-1)$-simplices, the set of $j$-simplices, and the set of
$(j+1)$-simplices of $\cG_{\tim}$. The probabilities $p_i$ with
$i\in[j-2]$ have no influence on any of these sets, while $p_{j-1}$
only affects the set of \emph{isolated} $(j-1)$-simplices. Isolated
$(j-1)$-simplices, however, have no effect on the set of
$j$-coboundaries, and thus do not influence $H^j(\cG_{\tim};R)$.
Therefore, when we consider whether or not the $j$-th cohomology group
vanishes, we will only take the probabilities $\pj,\dotsc,p_d$ into
account.

\smallskip
\subsection{Approximate order: Justifying \texorpdfstring{$\bar p_k=O\left(\frac{\log n}{n^{k-j}}\right)$}{order of pk}}
We first explain why we may assume that $\bar p_k=O\left(\frac{\log n}{n^{k-j}}\right)$.
In particular, this will imply the assumption $\bar \gamma_k\ge 0$, once $\bar \gamma_k$ is defined (see~\eqref{eq:gamma}).

What range of $\allp$ do we expect to be critical for
\connectedness\ of $\cG(n,\allp)$? Let us first look at a single
probability $p_k$, i.e.\ consider
$$\allp = (0,\dotsc,0,p_k,0\dotsc,0).$$
For $R
= \FF_2$ and $j+1 \le k \le d$, \cite[Theorem~1.11]{CooleyDelGiudiceKangSpruessel20}
states that the critical range lies around
\begin{equation*}
  p_k = \frac{(j+1)\log n + \log\log n}{(k-j+1)n^{k-j}}(k-j)!.
\end{equation*}
It is therefore reasonable to expect the critical range for general
coefficient group $R$ and general $\allp = (p_1,\dotsc,p_d)$ to lie
around
\begin{equation}\label{eq:parametersprelim}
  p_k = \frac{\alpha_k\log n + r_k}{n^{k-j}}(k-j)!, \quad\forall j\le k\le d,
\end{equation}
where each $\alpha_k$ is a non-negative constant, at least one $\alpha_k$
is non-zero, and each $r_k = r_k(n)$ is a function of order $o(\log n)$.

To justify this more precisely, note that if
$p_k\ge \frac{c_k\log n}{n^{k-j}}(k-j)!$ for some constant
$c_k>j+1$, a simple first moment calculation shows that whp
$\cG_\tim$ has a complete $j$-skeleton, and therefore if it is
\connected,
adding further $k$-simplices for $j\le k \le d$ will not change this.
Furthermore, it follows from the results of~\cite{CooleyDelGiudiceKangSpruessel20} (for $R=\FF_2$),
and indeed also from Theorem~\ref{thm:main} (for general $R$), that the complex
will in fact be \connected\ whp if $p_k$ is this large, and therefore
it is reasonable to scale the chosen direction $\barallp$ in such a way that
$\bar p_k = O(\frac{\log n}{n^{k-j}})$ for every $j+1\le k \le d$.

Thus, let us suppose that for each $k$ with $j \le k \le d$, the limit
\begin{equation}\label{eq:alpha}
  \bar \alpha_k \ :=\  \lim_{n\to\infty}\left(\frac{\bar p_k n^{k-j}}{(k-j)!\log n}\right)
\end{equation}
exists.
This is a reasonable assumption, because in terms of phase transitions,
we are interested in how the model behaves depending on the asymptotic
behaviour of the probabilities $p_k$, which should not
fluctuate between, say, $\frac{\log n}{n^{k-j}}$ and
$\frac{2\log n}{n^{k-j}}$.
Observe that if $\bar p_k = o(\log n/n^{k-j})$, then $\bar\alpha_k = 0$.
Indeed, we next argue that
we may also assume that
at least one $\bar\alpha_k$ is non-zero.

\smallskip
\subsection{Existence of \texorpdfstring{$k_0$}{k0}: Justifying~\ref{parameters:k0}}\label{sec:parameters:k0}
So far we have only guaranteed certain properties of $\bar\allp$ by scaling
appropriately.
By rescaling once more if necessary, we can certainly
guarantee that $\bar \alpha_k \neq 0$ for some $j\le k \le d$,
but we would like to ensure that this does not \emph{only} hold for $k=j$,
i.e.\ that there is in fact some $j+1 \le k \le d$ such that $\bar \alpha_k \neq 0$.
Note that this cannot necessarily be achieved by a simple rescaling without
potentially violating the condition that $p_j=O(\log n)$.

Instead, we consider the two cases:
\begin{enumerate}
\item \label{param:pjlarge} $\bar p_j\ge 1$ and $\bar\alpha_k=0$ for all $k=j+1,\ldots,d$;
\item \label{param:k0exists} $\bar p_j \ge 1$
and for some $j+1 \le k \le d$ we have $\bar \alpha_k \neq 0$.
\end{enumerate}
We argue that case~\ref{param:pjlarge} can be easily reduced
to case~\ref{param:k0exists}
and therefore we may assume that there exists $k_0\ge j+1$ with
$\bar \alpha_{k_0}\neq 0$ as stated in~\ref{parameters:k0}.

Indeed, suppose we have the slightly more general case than case~\ref{param:pjlarge},
that
$\bar p_j\ge 1$ and $\bar p_k\le \frac{\log n}{C n^{k-j}}$ for all $k=j+1,\ldots,d$
and for some sufficiently large $C$. In this case,
a simple second moment argument shows that there exists a constant $c>0$
such that
for $\tim = cn^{-j}$, whp $\cG_\tim=\cG(n,\tim \barallp)$ contains
an isolated $k$-simplex for some $0 \le k \le j-1$, which guarantees the existence
of an isolated simplex of dimension at most $k$ for any $\tim \le cn^{-j}$,
and therefore $(\cG_\tim)$
is not \connected\ in the interval $[0,cn^{-j}]$. Furthermore, another second moment argument shows
that if $C$ is large enough, whp
$\cG_{1}=\cG(n,\barallp)$ contains $\Omega(n^{j+\frac12})$ isolated $j$-simplices. Conditioned
on its presence in $\cG_{1}$, the probability that an isolated $j$-simplex
was already present in $\cG_{cn^{-j}}$ is at least $cn^{-j}$ independently for each such
simplex, and therefore with high probability one of these was present
throughout the entire range $\tim \in [cn^{-j},1]$. In other words,
either the presence of isolated
$k$-simplices for some $k \le j-1$ or of isolated $j$-simplices
ensure that 
whp the process is certainly not \connected\ until the time when it has a complete $j$-skeleton.
Therefore we may increase $\bar p_k$
for all $k\ge j+1$ by the same factor
(equivalent to decreasing $\bar p_j$)
until $\bar p_k = \frac{\log n}{Cn^{k-j}}$ for some $k$ without affecting
which appearances of simplices cause the process becomes \connected.
In other words, we may assume that $\bar \alpha_{k_0} > 0$ for some $k_0\ge j+1$.

\smallskip
\subsection{Lower bound on \texorpdfstring{$\bar p_k$}{pk}: Justifying \texorpdfstring{$\bar \gamma_k<\infty$}{gammak}}
Furthermore, we may assume that each non-zero probability $\bar p_k$ is not
`too small', or in other words that any $\bar p_k$ which is very small is in fact $0$.
More precisely, we have shown the existence of an index $k_0$
with $\bar \alpha_{k_0}\neq 0$, which implies that $\bar p_{k_0} = \Theta(\frac{\log n}{n^{k_0-j}})$.
Now if $\bar p_k \le n^{-(k+k_0-j+1)}$, then a simple first moment calculation shows
that whp all $k_0$-simplices are born (and so in particular the complex
is \connected) before any $k$-simplices are born. Thus
we may set $\bar p_k=0$ without affecting when the process is \connected.
Therefore we may assume that
\begin{equation}\label{eq:gamma}
  \bar \gamma_k \ :=\  \sup\{\gamma\in\mathbb{R} \mid \bar p_k n^{k-j+\gamma} = o(1)\}
\end{equation}
exists for every $k$ with $j \le k \le d$ and $\bar p_k \not= 0$. By the
existence of the limit in~\eqref{eq:alpha}, we have $\bar \gamma_k \ge 0$.

\smallskip
\subsection{Fine-tuning}
Finally, let $\bar \beta_k$ be the function of $n$ for which
\begin{equation}\label{eq:parametersgeneral}
  \bar p_k = \frac{\bar \alpha_k\log n + \bar \beta_k}{n^{k-j+\bar \gamma_k}}(k-j)!.
\end{equation}
Note that the function $\bar \beta_k$ might be negative if $\bar \alpha_k \not= 0$.

\smallskip
\subsection{\texorpdfstring{$j$}{j}-admissibility}
So far we have only ensured that properties~\ref{parameters:zero}--\ref{parameters:sublog}
hold for $j\le k \le d$ (note that~\ref{parameters:k0} is independent of $k$).
To show that we may assume that these properties
also hold for $1 \le k \le j-1$, we use a similar argument to the one in
Section~\ref{sec:parameters:k0}:
if for some $k \in [j-1]$ we have $\bar p_k \ge Cn^{j-k}\log n$ for some sufficiently large constant $C$,
then whp $(\cG_\tim)$ contains an isolated $i$-simplex for some $i \in [k]_0$ (and is therefore not \connected)
until the moment when it has a complete $k$-skeleton. Therefore we may decrease $\bar p_k$ to $Cn^{j-k}\log n$
without changing when the process becomes \connected. In other words, we may assume
that $\bar p_k = O\left(\frac{\log n}{n^{k-j}}\right)$,
and thus also that $\bar \alpha_k, \bar \beta_k,\bar \gamma_k$ are well-defined,
and properties~\ref{parameters:zero}--\ref{parameters:sublog} follow.

\smallskip
\subsection{\texorpdfstring{$j$}{j}-criticality}
It only remains to justify the assumptions
of Definition~\ref{def:criticaldirection}. These properties
can also be guaranteed by appropriate scaling of $\bar \allp$.

  To see this, observe that scaling $\barallp$ by a 
  constant $C^*$ also scales the $\bar\alpha_k$ by the same factor $C^*$,
  while leaving the $\bar\gamma_k$ unchanged.
  Thus if we let $\barallp=C^*\allp$, where $\allp$ is a $j$-admissible direction,
  since $\alpha_{k_0}>0$, we have
  $$ \bar \lambda_k j+1 - \bar \gamma_k - \Theta(C^*). $$
  Thus by choosing $C^*$ large enough, we can ensure that $\bar \lambda_k <-1$ for all $k\ge j$.
  Since $\bar\logp_k \log n$ is the
  main term in $\bar\logp_k \log n + \bar\sublogp_k + \bar\constp_k$,
  this would mean that~\ref{crit:less} certainly holds if $C^*$ is
  large enough.
  On the other hand, since $\bar \gamma_{k_0}=\gamma_{k_0}=0$,
  if $C^*$ is small enough we have $\bar\logp_{k_0}\ge j+1-\frac{1}{2}>0$,
  i.e.\ at least one of the
  $\bar\logp_k$ is positive.
  By continuity, we may choose $C^*$ such that~\ref{crit:less}
  and~\ref{crit:equal} both hold.

\newpage  

\section{Proofs of auxiliary results}\label{sec:proofaux}

\subsection{Proof of Lemma~\ref{lem:criticalscaling}}
\label{sec:proofofcriticalscaling}

We prove the statement for $i=j-1$; for general $i \in [j-1]$ it
suffices to iterate the procedure $j-i$ times.

We thus need to show that we can choose a positive constant $\eta =
\eta_{j-1}$ and a function $\epsilon = \epsilon_{j-1}(n) = o(1)$ such
that
\begin{equation*}
  \barallp' = \barallp'(\eta,\epsilon)\ :=\
  \frac{\eta+\epsilon}{n} \barallp
\end{equation*}
is a $(j-1)$-critical direction
(Definition~\ref{def:criticaldirection}).

Recall that the $j$-critical direction  $\barallp
= (\bar p_1,\ldots,\bar p_d)$ is in particular a $j$-admissible direction (Definition~\ref{def:criticalp}), i.e.\
for every $k\in [d]$
\begin{equation*}
    \bar p_k = \frac{\bar\alpha_k\log n + \bar\beta_k}{n^{k-j+\bar\gamma_k}}(k-j)!,
  \end{equation*}
with $\bar\alpha_k$, $\bar\beta_k= \bar\beta_k(n)$, and $\bar\gamma_k$ satisfying
conditions \ref{parameters:zero}--\ref{parameters:k0}.

This implies that $\barallp'$ is a $(j-1)$-admissible direction:
for every $k\in [d]$, we have
\begin{align*}
\bar p_k' = \frac{\eta+\epsilon}{n} \bar p_k 
     & = \frac{\bar\alpha_k'\log n + \bar\beta_k'}{n^{k-(j-1)+\bar\gamma_k'}}(k-(j -1))!,
\end{align*} 
where
\begin{equation} \label{eq:paramprime}
\bar\alpha_k' \ :=\  \frac{\eta \bar\alpha_k}{k-j+1},
\qquad \bar\beta_k'\ :=\  \frac{\eta \bar\beta_k
+ \epsilon \bar \alpha_k \log n + \epsilon
\bar\beta_k}{k-j+1}, \qquad \bar\gamma_k' \ :=\ 
\bar\gamma_k,
\end{equation}
and it is easy to check that for any choices of 
$\eta$ and $\epsilon$, the parameters
$\bar\alpha_k'$, $\bar\beta_k'$, and 
$\bar\gamma_k'$ satisfy conditions~\ref{parameters:zero}--\ref{parameters:k0} 
in Definition~\ref{def:criticalp}.

We now want to prove that, for the appropriate choices of $\eta$ and 
$\epsilon$, the vector $\barallp'$ is also a $(j-1)$-critical 
direction (Definition~\ref{def:criticaldirection}), i.e. 
\renewcommand{\theenumi}{(C\arabic{enumi}$'$)}
\begin{enumerate}
\item \label{crit:prime:less}
  $\bar{\logp}'_k\log n + \bar{\sublogp}'_k + \bar{\constp}'_k \le 0$, $\quad$ for all indices $k$ with $j-1 \le k \le d$ and $\bar{p}'_k \not= 0$;
\item \label{crit:prime:equal}
  $\bar{\logp}'_{\bar k}\log n + \bar{\sublogp}'_{\bar k} + \bar{\constp}'_{\bar k} = 0$, $\quad$ for some $\bar k$ with $j-1 \le \bar k \le d$,
\end{enumerate}
\renewcommand{\theenumi}{(\roman{enumi})}
where the parameters $\bar{\logp}'_k$, $\bar{\sublogp}'_k$, and 
$\bar{\constp}'_k$ are defined as in \eqref{eq:lmn}, but with $j$
replaced by $j-1$.

Recall that $\bar{\logp}'_k$ and $\bar{\constp}'_k$ are constants, 
while each $\bar{\sublogp}'_k$ is a function with
$\bar{\sublogp}'_k = o(\log n)$. Thus, \ref{crit:prime:less} 
and~\ref{crit:prime:equal} will both hold if and only if
\begin{align}
\begin{split}
\bar\logp_k' \le 0 
&\qquad \text{ for every  } j-1\le k\le d \text{ with } \bar{p}'_k \not= 0; \\
\bar\logp_{\bar k}' = 0 
&\qquad \text{ for some }
j-1\le \bar k\le d;
\end{split}\label{eq:condlogp'}\\
\text{and}\nonumber\\
\begin{split}
\bar\sublogp_k' \le -\bar\constp_k'& \qquad \text{ for every  } j-1 \le k \le d \text{ such that }
\bar{\logp}'_k = 0; \\
\bar\sublogp_{\bar k}' = - \bar\constp_{\bar k}' & \qquad \text{ for some } \bar k \text{ such that }
\bar{\logp}'_{\bar k} = 0.
\end{split}\label{eq:condsublogp'}
\end{align}
Let us consider~\eqref{eq:condlogp'} first.
For every $k$ with $\bar{p}'_k \not= 0$ (which is equivalent to
$\bar{p}_k \not= 0$), we have
\begin{align*}
\bar\logp_k' &\stackrel{\phantom{\eqref{eq:paramprime}}}{=} j - \bar\gamma_k'
- (k-j+2) \sum\nolimits_{i=j}^d \bar\alpha_i'\\
& \stackrel{\eqref{eq:paramprime}}{=}
j - \bar\gamma_k - \eta (k-j+2) 
\sum\nolimits_{i=j}^d \frac{\bar\alpha_i}{i-j+1}.
\end{align*}
Recall that $k_0$ is an index  with $j+1 \le k_0 \le d$ such
that $\bar\alpha_{k_0} \neq 0$ and $\bar\gamma_{k_0} = 0$ 
(cf.~\ref{parameters:zero} and~\ref{parameters:k0}
for $\barallp$).
We observe that
\begin{equation*}
  \lim_{\eta\to0}\bar{\logp}'_{k_0} = j > 0
  \qquad\text{and}\qquad
  \lim_{\eta\to\infty}\bar{\logp}'_{k} = -\infty
  \quad \text{for every } j-1 \le k \le d \text{ with } \bar{p}'_k \not= 0.
\end{equation*}
Hence, since each $\bar{\logp}'_{k}$ is a continuous function of
$\eta$, we can choose $\eta > 0$ such that \eqref{eq:condlogp'} holds.

For the rest of the proof, let this value of $\eta$ be fixed. We will
now show that we can choose the function $\epsilon$ so that
\eqref{eq:condsublogp'} is satisfied as well. To simplify notation,
whenever we consider $\bar{\sublogp}'_{k}$ or $\bar{\constp}'_k$ in
the following, we will assume that $\bar{\logp}'_k = 0$.

By \eqref{eq:lmn}, with $j$ replaced by $j-1$, we have
\begin{align*}
  \bar\constp_k' &=
      \begin{cases}
        -\log(j!) & \text{if }k=j-1,\\
        -\log((j-1)!)-\log(k-j+2) +\log(\bar{\alpha}'_k) & \text{if } k\not=j-1 \text{ and }\bar\alpha_k'\neq 0,\\
        -\log((j-1)!)-\log(k-j+2) & \text{otherwise},
      \end{cases}
\end{align*}
which by \eqref{eq:paramprime} is independent from $\epsilon$.
Furthermore,
\begin{equation}\label{eq:sublogprime}
\bar\sublogp_k' =
- (k-j+2)\sum\nolimits_{i=j}^{d}\frac{\bar\beta_i'}{n^{\bar\gamma_i'}} +
      \begin{cases}
        0 & \text{if } \bar p_k'>1,\\
        \log\log n & \text{if } \bar p_k'\le 1 \text{ and } \bar\alpha_k'\not=0,\\
        \log(\bar\beta_k') & \text{if } \bar p_k'\le 1 \text{ and } \bar\alpha_k'=0.
      \end{cases}
\end{equation}
By \eqref{eq:paramprime},
\begin{equation*}
  \sum\nolimits_{i=j}^{d}\frac{\bar\beta_i'}{n^{\bar\gamma_i'}}
  = \sum\nolimits_{i=j}^{d}\frac{\eta \bar\beta_i
+ \epsilon \bar \alpha_i \log n + \epsilon
\bar\beta_i}{(i-j+1)n^{\bar\gamma_i}}.
\end{equation*}
Now \ref{parameters:zero}--\ref{parameters:k0} imply that
$c := \sum_{i}\frac{\bar{\alpha}_i}{(i-j+1)n^{\bar\gamma_i}}$ is a
positive constant, while $\sum_{i}\frac{\bar\beta_i}{(i-j+1)n^{\bar\gamma_i}}
= o(\log n)$. Moreover, all possible summands in the case distinction
of~\eqref{eq:sublogprime} are $o(\log n)$ as well. Therefore, there
exists a positive function $f(n) = o(\log n)$, which does not depend
on the choice of $\epsilon$, such that
\begin{equation*}
  |\bar\sublogp_k' + c(k-j+2)\epsilon\log n| \le f(n)
\end{equation*}
for all $j-1 \le k \le d$ with $\bar{p}'_k \not= 0$.
Fix a function $\omega$ of $n$ with $\omega \to \infty$, but
$\omega = o\left(\frac{\log n}{f(n)}\right)$. Then for all indices $k$
as above,
\begin{equation*}
  \bar{\sublogp}'_k \xrightarrow{n\to\infty}
  \begin{cases}
    \infty & \text{for }\epsilon = -\frac{1}{\omega},\\
    -\infty & \text{for }\epsilon = \frac{1}{\omega}.
  \end{cases}
\end{equation*}
Therefore, by continuity we can choose 
$\epsilon = o(1)$ such that 
\eqref{eq:condsublogp'} holds.
Since we have now found $\eta$ and $\epsilon$ such that
\eqref{eq:condlogp'} and \eqref{eq:condsublogp'}
simultaneously hold,
\ref{crit:prime:less} and~\ref{crit:prime:equal} are both satisfied, i.e.\ 
$\barallp'$ is a $(j-1)$-critical direction,
as required.
\qed

\smallskip
\subsection{Proof of Lemma~\ref{lem:expectedNjk}}
\label{sec:proofofexpectedNjk}
Observe that in Definition~\ref{def:generalparameters}, $\alpha_k$
and $\gamma_k$ exist, because $\tim=O(1)$ and $\tim=\Omega(n^{-c})$,
respectively. Thus, $\logp_k$, $\sublogp_k$, and $\constp_k$ can be
defined analogously to \eqref{eq:lmn}. Moreover, since $\bar\alpha_i / n^{\bar\gamma_i} = \bar\alpha_i$ 
for each $j+1\le i\le d$ by \ref{parameters:zero} in Definition~\ref{def:criticalp}, 
we have
\begin{align} 
\bq^{\;\tim} &\stackrel{\eqref{eq:q_computing}}{=}
(1+o(1))
\exp\left(- \tim \sum\nolimits_{i=j+1}^d\left(\bar\alpha_i \log n+ \frac{\bar\beta_i}{n^{\bar\gamma_i}}\right)\right) \nonumber \\
& \stackrel{\;\eqref{eq:parameters}\;}{=}
(1+o(1))
\exp\left(-\sum\nolimits_{i=j+1}^d \tim \bar p_i \frac{n^{i-j}}{(i-j)!}  \right) \nonumber \\
& \stackrel{\phantom{\eqref{eq:q_computing}}}{=}
(1+o(1))
\exp\left(-\sum\nolimits_{i=j+1}^d  p_i \frac{n^{i-j}}{(i-j)!}  \right) \nonumber \\
& \stackrel{\phantom{\eqref{eq:q_computing}}}{=} (1+o(1))
\exp\left(-\sum\nolimits_{i=j+1}^d\left(\alpha_i \log n+ \frac{\beta_i}{n^{\gamma_i}}\right)\right). \label{eq:qtim_computing}
\end{align}

Suppose first that $k\ge j+1$.
Recall that in this case $p_k \le 1$ by Remark~\ref{rem:pkprob}. Furthermore, we observe that
  \begin{align}\label{eq:scaledpk}
    \frac{n^{k+1}p_k}{j!(k-j+1)!} &= \frac{n^{j+1-\gamma_k}(\alpha_k\log n+\beta_k)}{j!(k-j+1)},
    \end{align}
by Definition~\ref{def:generalparameters}. 
We thus have
  \begin{align*}
    \log\left(\EE(\varNjk)\right) &\stackrel{\eqref{eq:expXjk}}{=} \log\left((1+o(1))\frac{n^{k+1}p_k}{j!(k-j+1)!}\bq^{\; \tim (k-j+1)}\right)\\
    &\stackrel{\eqref{eq:scaledpk}}{=} (j+1-\gamma_k)\log n + \log (\alpha_k \log n + \beta_k) - \log (j!)-\log (k-j+1)\\
  & \hspace{2cm} +(k-j+1)\log\left(\bq^{\;\tim}\right) + o(1)\\
  & \stackrel{\eqref{eq:qtim_computing}}{=} \left(j+1-\gamma_k-(k-j+1)\sum\nolimits_{i=j+1}^d\alpha_i\right)\log n \\
  & \hspace{2cm} -(k-j+1)\sum\nolimits_{i=j+1}^d \frac{\beta_i}{n^{\gamma_i}} + \log (\alpha_k \log n + \beta_k)\\
  & \hspace{2cm} -\log(j!)-\log(k-j+1) + o(1).
  \end{align*}
Now \eqref{eq:lmn}, together with the fact that $p_k \le 1$, implies
that
\begin{equation*}
  \log\left(\EE(\varNjk)\right) =
  \logp_k \log n + \sublogp_k + \constp_k + o(1),
\end{equation*}
as required.

If $k=j$, substituting \eqref{eq:qtim_computing} 
into \eqref{eq:expXjj} we have
\begin{equation*}
\EE(\varisol) = (1+o(1)) 
\frac{n^{j+1} \min\{p_j,1\} }{(j+1)!} 
\exp\left(-\sum\nolimits_{i=j+1}^d \left(\alpha_i\log n + \frac{\beta_i}{n^{\gamma_i}} \right)\right).
\end{equation*}
Therefore we obtain
\begin{align}\label{eq:logexpXjj}
\log(\EE(\varisol)) = \left(j+1-\sum\nolimits_{i=j+1}^d \alpha_i\right)\log n & + \sum\nolimits_{i=j+1}^d \frac{\beta_i}{n^{\gamma_i}} - \log((j+1)!)
\nonumber \\
&
+ \log \left( \min\{p_j,1\} \right) + o(1).
\end{align}

We first consider the case when $p_j\le 1$, in which case by
Definition~\ref{def:generalparameters} we have $\alpha_j=0$, and so
$$
\log \left( \min\{p_j,1\} \right) = \log p_j = \log
\left( \frac{\beta_j}{n^{\gamma_j}} \right) 
= -\gamma_j \log n + \log \beta_j.
$$
Substituting this into~\eqref{eq:logexpXjj}, we obtain
\begin{align*}
\log(\EE(\varisol)) & = \left(j+1-\gamma_j-\sum\nolimits_{i=j+1}^d \alpha_i\right)\log n + \sum\nolimits_{i=j+1}^d \frac{\beta_i}{n^{\gamma_i}} +\log \beta_j  - \log((j+1)!) + o(1)\\
& = \logp_j \log n + \sublogp_j + \constp_j + o(1)
\end{align*}
as required.

On the other hand, if $p_j>1$, then we must have $\gamma_j=0$. Furthermore,
$$
\log \left( \min \{p_j,1\} \right)= \log 1 =0,
$$
and therefore~\eqref{eq:logexpXjj} gives
\begin{align*}
\log(\EE(\varisol)) & = \left(j+1-\sum\nolimits_{i=j+1}^d \alpha_i\right)\log n + \sum\nolimits_{i=j+1}^d \frac{\beta_i}{n^{\gamma_i}} - \log((j+1)!) + o(1)\\
& = \logp_j \log n + \sublogp_j + \constp_j + o(1)
\end{align*}
since we are in the case when $k=j$ and $p_j>1$.
\qed

\smallskip
\subsection{Proof of Lemma~\ref{lem:firstMjk}}
\label{sec:prooffirstMjk}
  Let $\hat\cT$ be the set of all 4-tuples $(K,C,w,a)$ that might form a copy of $\Mjkzerohat$ in
  $\cG_\tim$ (i.e.\ all sizes and containment relations are
  correct, but we make no assumptions about which simplices are present
  or absent), and let $\hat T = (K,C,w,a)\in \hat\cT$.
  Property~\ref{Mjkhat:simplex} holds with probability 
\begin{equation}\label{eq:pk0}
p_{k_0}=\frac{\eps}{n}\bar p_{k_0}
\stackrel{\eqref{eq:parameters}}{=}
\Theta \left( \frac{\log n}{n^{k_0-j+1}}\right),
\end{equation}  
 since the choice of $k_0$ is such that $\bar \alpha_{k_0}\neq 0$
  (see Definition~\ref{def:indices}~\ref{def:indices:kzero}).
 By Proposition~\ref{prop:localisedprob} the probability that
  \ref{Mjkhat:flower} holds is $(1+o(1))
  \bq^{\;\tim (k-j+1)}$, where  since $\tim = \eps/n$ we have
  \begin{equation*} 
    \bq^{\;\tim} \stackrel{\eqref{eq:q_computing}}{=}
   (1+o(1)) \exp \left( - \frac{\eps}{n}
   \cdot \sum\nolimits_{k=j+1}^d
   \left( \bar\alpha_k \log n +
   \frac{\bar \beta_k}{n^{\bar\gamma_k}} \right)
   \right) = 1+o(1),
  \end{equation*}
and therefore~\ref{Mjkhat:flower}
holds---independently of whether \ref{Mjkhat:simplex} holds---with
probability $1+o(1)$.

  In order to calculate the probability that~\ref{Mjkhat:shell} also holds, first observe that
  if~\ref{Mjkhat:flower} holds, then no simplex can contain more than
  one side of the (potential) $j$-shell $C\cup\{w\}\cup\{a\}$. Thus,
  conditioned on the event that~\ref{Mjkhat:simplex} and~\ref{Mjkhat:flower} hold,
  each of the $j+1$ sides of $C\cup\{w\}\cup\{a\}$ forms a $j$-simplex independently
  with probability
  \begin{equation*}
    1-\prod\nolimits_{k=j}^{d}(1-p_k)^{\binom{n-j-1}{k-j}+O(n^{k-j-1})}
    = (1+o(1))r,
  \end{equation*}
  where
  \begin{equation}\label{eq:r}
    r \ :=\  \sum\nolimits_{k=j}^{d}\frac{p_k n^{k-j}}{(k-j)!} =
    \Theta\left(\frac{\log n}{n}\right).
  \end{equation}
  Combining all the probabilities, we obtain
  \begin{equation}\label{eq:probMjkzerohat}
    \Pr(\hat T \text{ forms a copy of }\Mjkzerohat) = (1+o(1))p_{k_0}r^{j+1}.
  \end{equation}

  Recall that $\varMjkzerohat$ denotes the number of copies of $\Mjkzerohat$ in
  $\cG_{\tim}$. Now~\eqref{eq:probMjkzerohat} implies that
  \begin{align}\label{eq:expecMjkzerohat}
    \EE(\varMjkzerohat)
    & = (1+o(1))\binom{n}{k_0+1}\binom{k_0+1}{j}(k_0-j+1)(n-k_0-1)p_{k_0}r^{j+1} \nonumber \\
    & = (1+o(1))\frac{p_{k_0}r^{j+1}n^{k_0+2}}{j!(k_0-j)!}
    \stackrel{\eqref{eq:pk0},\eqref{eq:r}}{=} 
    \Theta((\log n)^{j+2}).
  \end{align}
  
  We now aim to calculate the second moment $\EE((\varMjkzerohat)^2)$.
  Given two $4$-tuples $\hat T_1=(K_1,C_1,w_1,a_1)$ and
  $\hat T_2=(K_2,C_2,w_2,a_2)$, we define
  \begin{itemize}
  \item $I = I(\hat T_1,\hat T_2) \ :=\ 
    (K_1 \cup \{a_1\}) \cap (K_2 \cup \{a_2\})$ and $i\ :=\ |I|$;
  \item $s = s(\hat T_1,\hat T_2) \ :=\ 
    \begin{cases}
      1 & \mbox{if }K_1 = K_2,\\
      2 & \mbox{otherwise};
    \end{cases}$
  \item $\cJ_x$ to be the set of all $(j+1)$-subsets of
    $\{C_x\cup\{a_x\}\cup\{w_x\}\}$ for $x=1,2$ and 
    \begin{equation*}
      t=t(\hat T_1,\hat T_2) \ :=\ 
      |(\cJ_1\cup\cJ_2)\setminus\{C_1\cup\{w_1\},C_2\cup\{w_2\}\}|,
    \end{equation*}
    i.e.\ the number of $(j+1)$-sets that are sides of the (potential)
    $j$-shells of $\hat T_1$ and $\hat T_2$, but not a base of either
    $j$-shell. 
  \end{itemize}
  If $s=2$ and the intersection of the two simplices contains a petal,
  then $\hat T_1$ and $\hat T_2$ cannot both form a copy of $\Mjkzerohat$,
  because~\ref{Mjkhat:flower} would be violated. In the following, we
  therefore assume that this is not the case.

  Clearly,~\ref{Mjkhat:simplex} holds for both $\hat T_1$ and
  $\hat T_2$ simultaneously with probability $(p_{k_0})^s$,
  while conditioned on~\ref{Mjkhat:simplex},
  by Proposition~\ref{prop:localisedprob}, the probability that~\ref{Mjkhat:flower} holds
  for both $\hat T_1$ and $\hat T_2$ simultaneously
  is (at least) $(1+o(1))\bq^{\;\tim 2(k_0-j+1)}=1+o(1)$. Conditioned
  on~\ref{Mjkhat:simplex} and~\ref{Mjkhat:flower} holding, observe
  that each of the $t$ sides of the (potential) $j$-shells lies in some
  $k$-simplex (and hence forms a $j$-simplex) with probability $r$.
  Moreover, no simplex in $\cG_{\tim}$ can contain more than two of
  those sides (at most one from each potential shell since otherwise
  it would contain a petal, which is ruled out by the conditioning on~\ref{Mjkhat:flower}).
  Furthermore, the
  probability of a side lying in any $k$-simplex that contains two
  distinct sides is
  \begin{equation*}
    1-\prod\nolimits_{k=j+1}^{d}(1-p_k)^{O(n^{k-j-1})} = O\left(\frac{\log n}{n^2}\right)
    \stackrel{\eqref{eq:r}}{=} o(r^2).
  \end{equation*}
  Therefore, the probability that all $t$ sides form $j$-simplices is
  $(1+o(1))r^t$ and thus
  \begin{equation}\label{eq:psqt}
    \Pr(\hat T_1,\hat T_2 \mbox{ both form copies of }\Mjkzerohat) = (1+o(1))(p_{k_0})^s r^t.
  \end{equation}

  Define $\hat \cT^2(i,s,t)$ to be the set of pairs
  $(\hat T_1,\hat T_2)
  \in \hat\cT \times \hat\cT$ with parameters $i,s$ and
  $t$. With this
  notation,~\eqref{eq:psqt} implies that
  \begin{equation*}
    \EE((\varMjkzerohat)^2) = (1+o(1))\sum_{(i,s,t)} \,
    \sum_{(\hat T_1,\hat T_2) \in \hat\cT^2(i,s,t)} (p_{k_0})^s r^t.
  \end{equation*}
  Observe that $|\hat\cT^2(i,s,t)| = O(n^{2k_0+4-i})$.We can now
  estimate the contributions of all the summands, distinguishing according to the
  possible values of $s$ and $i$. 

  \emph{Case 1:} s=1. This means that $K_1 = K_2$, and thus $k_0+1\le i\le k_0+2$.
  \begin{itemize}
  \item $i=k_0+1$. In this case $a_1 \neq a_2$ and thus the sets of
    sides of the two $j$-shells would be disjoint, i.e.\ $t=2j+2$.
    Therefore we get a contribution of order
    \begin{equation*}
      O\left(p_{k_0}r^{2j+2}n^{2k_0+4-(k_0+1)}\right)
      \stackrel{\eqref{eq:expecMjkzerohat}}{=}
      O\left(\frac{\EE(\varMjkzerohat)^2}{p_{k_0}n^{k_0+1}}\right)
      \stackrel{\eqref{eq:pk0}}{=} o\left((\EE(\varMjkzerohat)^2\right).
    \end{equation*}
    Indeed, in order to prove the final property of the lemma, that
    the associated copies of $\Mjkzero$ are distinct,
    we observe something even stronger:
    we have
    \begin{equation*}
      p_{k_0}r^{2j+2}n^{2k_0+4-(k_0+1)} = \Theta\left(\frac{(\log n)^{2j+3}}{n^j}\right)
      \stackrel{\eqref{eq:pk0},\eqref{eq:r}}{=} o(1).
    \end{equation*}
    Thus by Markov's inequality, whp there are no two copies of $\Mjkzerohat$ that share the same
    $k_0$-simplex but have distinct apex vertices.
  \item $i=k_0+2$. The two $j$-shells have the same apex vertex and thus
    the $j$-shells coincide if and only if they have the same base.
    This means that $t\geq j+1$, which gives a contribution of order
    \begin{equation*}
      O(p_{k_0}r^{j+1}n^{2k_0+4-(k_0+2)}) \stackrel{\eqref{eq:expecMjkzerohat}}{=}
      O\left( \EE(\varMjkzerohat) \right) =
      o\left( \EE(\varMjkzerohat)^2 \right).
    \end{equation*}
  \end{itemize}

  \emph{Case 2:} s=2.
  \begin{itemize}
  \item $i=0$. We show that this case represents the dominant
    contribution to $\EE((\varMjkzerohat)^2)$. The two $j$-shells are
    disjoint, hence $t=2j+2$. Observe that we have
    \begin{equation*}
      \binom{n}{k_0+1}\binom{k_0+1}{j}(k_0-j+1)(n-k_0-1) = (1+o(1))\frac{n^{k_0+2}}{j!(k_0-j)!}
    \end{equation*}
    choices for $\hat T_1$. For any fixed $\hat T_1$, the number of choices
    for $\hat T_2$ that yield $i=0$ is
    \begin{equation*}
      \binom{n-k_0-2}{k_0+1}\binom{k_0+1}{j}(k_0-j+1)(n-2k_0-3) =
      (1+o(1))\frac{n^{k_0+2}}{j!(k_0-j)!}.
    \end{equation*}
    Thus, the contribution of all such pairs is
    \begin{equation*}
      (1+o(1))\frac{p_{k_0}^2
      r^{2j+2}n^{2k_0+4}}{(j!(k_0-j)!)^2}
      \stackrel{\eqref{eq:expecMjkzerohat}}{=} (1+o(1)) \EE(\varMjkzerohat)^2.
    \end{equation*}
  \item $1\le i\le j$. In this case $\hat T_1$ and $\hat T_2$ cannot share a
    $j$-simplex of their shells, i.e.\ $t=2j+2$. Therefore the
    contribution is
    \begin{equation*}
      O\left(p_{k_0}^2
      r^{2j+2}n^{2k_0+4-i}\right)
      \stackrel{\eqref{eq:expecMjkzerohat}}{=}
      O\left(\frac{\EE(\varMjkzerohat)^2}{n^{i}}\right) = o(\EE(\varMjkzerohat)^2).
    \end{equation*}
  \item $i=j+1$. Here, $\hat T_1$ and $\hat T_2$ can share at most one
    $j$-simplex of their shells, which means $t \geq 2j+1$ and we have
    a contribution of order
    \begin{equation*}
      O\left(p_{k_0}^2
      r^{2j+1}n^{2k_0+4-(j+1)}\right)
      \stackrel{\eqref{eq:expecMjkzerohat}}{=}
      O\left(\frac{\EE(\varMjkzerohat)^2}{r 
      n^{j+1}}\right)
      \stackrel{\eqref{eq:r}}{=} o(\EE(\varMjkzerohat)^2).
    \end{equation*}
  \item $j+2 \le i \le k_0+2$. In this case $t \ge j$, because $\hat T_1$
    and $\hat T_2$ may share their $j$-shells, meaning that at least $j+2$ many $j$-simplices
    must be present, but have different bases,
    i.e.\ up to two sides of the (potential) $j$-shells may be
    automatically present as $j$-simplices because of $K_1$ and $K_2$. Therefore the
    contribution is
    \begin{equation*}
      O\left(p_{k_0}^2
      r^jn^{2k_0+4-i}\right) \stackrel{\eqref{eq:expecMjkzerohat}}{=}
      O\left(\frac{\EE(\varMjkzerohat)^2}{r^{j+2}n^i}\right)
      \stackrel{\eqref{eq:r}}{=} o(\EE(\varMjkzerohat)^2).
    \end{equation*}
  \end{itemize}
  Summing over all cases shows that $\EE((\varMjkzerohat)^2) =
  (1+o(1))\EE(\varMjkzerohat)^2$, as desired. Thus, Chebyshev's inequality
  implies that $\varMjkzerohat = (1+o(1))\EE(\varMjkzerohat)$ whp.

  Finally, recall that
  in the case $s=1$, $i=k_0+1$, we observed that
  whp there are no two copies of $\Mjkzerohat$
  that contain a common $\Njkzero$, in which case all copies of
  $\Mjkzerohat$ must have distinct associated copies of $\Mjkzero$, as claimed.
\qed

\smallskip
\subsection{Proof of Proposition~\ref{prop:scaledparameters}}
\label{sec:proofofscaledparam}
Observe that for each $k \ge j$ we have
\begin{align*}
p_k = \frac{\alpha_k\log n + \beta_k}{n^{k-j+\gamma_k}}(k-j)!
& = \left(1 + \xi \right)\frac{\bar \alpha_k\log n + \bar \beta_k}{n^{k-j+\bar \gamma_k}} (k-j)!\\
& = \frac{\bar \alpha_k \log n + \left(1+\xi \right)
\bar \beta_k + \bar \alpha_k \xi \log n}{n^{k-j+\bar \gamma_k}}
(k-j)!,
\end{align*}
and therefore the first three statements follow directly.
Furthermore, since $\logp_k$ and $\constp_k$ are dependent only on $\alpha_k$ and $\gamma_k$, and not on $\beta_k$,
the fourth and sixth statements also follow.

For $k\ge j+1$, recall that $p_k \le 1$ by Remark~\ref{rem:pkprob} and therefore we have
\[ \sublogp_k  = -(k-j+1)\sum\nolimits_{i=j+1}^d \frac{\beta_i}{n^{\gamma_i}} + 
\begin{cases}
\log\log n  &\text{if }  \alpha_k\not=0,\\
\log(\beta_k) & \text{if }  \alpha_k=0.
\end{cases}\]
We have that 
\begin{align}
\sum_{i=j+1}^d \frac{\beta_i}{n^{\gamma_i}} 
& \stackrel{\phantom{\text{\ref{parameters:subpoly}-\ref{parameters:sublog}}}}{=}
 \sum_{i=j+1}^{d}\frac{\left(1+\xi\right)
\bar \beta_i+ \bar \alpha_i \xi \log n}{n^{\bar\gamma_i}}
\nonumber \\ &
\stackrel{\text{\quad \ref{parameters:zero} \quad}}{=} \sum_{i=j+1}^{d}  \frac{\bar \beta_i}{n^{\bar\gamma_i}} + 
\xi \cdot \left(\sum_{i=j+1}^{d} \bar\alpha_i \log n +
\sum_{i=j+1}^{d} \frac{\bar\beta_i}{n^{\bar\gamma_i}} \right)\nonumber \\
& \stackrel{\text{\ref{parameters:subpoly}-\ref{parameters:sublog}}}{=} 
\sum_{i=j+1}^{d}  \frac{\bar \beta_i}{n^{\bar\gamma_i}} + 
\xi \cdot \left(\sum_{i=j+1}^{d} \bar\alpha_i \log n +
\sum_{i:\bar\gamma_i \neq 0} \frac{\bar\beta_i}{n^{\bar\gamma_i}} + o(\log n)\right) \nonumber \\
& \stackrel{\phantom{\text{\ref{parameters:subpoly}-\ref{parameters:sublog}}}}{=}
 \sum_{i=j+1}^{d}  \frac{\bar \beta_i}{n^{\bar\gamma_i}} 
+ (1+o(1)) \xi \sum_{i=j+1}^{d} \bar\alpha_i \log n.
\label{eq:betaterm}
\end{align}
Observe that \eqref{eq:betaterm} does not depend on $k$.
For all $k$ with $\alpha_k=0$, note that $\sublogp_k$ contains the 
additional term
\[ \log(\beta_k) = \log((1+\xi)\bar\beta_k) = \log(\bar\beta_k) + O(\xi), \]
while for all $k$ with $\alpha_k \not=0$, we have the same additional
term $\log\log n$ in both $\sublogp_k$ and $\bar{\sublogp}_k$.
Thus in total we have
\[ \sublogp_k = \bar\sublogp_k - (1+o(1))
(k-j+1) \xi \sum\nolimits_{i=j+1}^d \bar \alpha_i \log n.\]

On the other hand, for $k=j$, observe that~\ref{parameters:zero} implies
that if $p_j\le 1$, then $\alpha_j=\bar\alpha_j = 0$. Furthermore,
\begin{equation}\label{eq:crossone}
\mbox{if } p_j\le 1 \le \bar p_j, \mbox{ then } \bar \alpha_j=\bar \gamma_j=0 \mbox{ and } \bar p_j=\bar \beta_j = 1+O(\xi).
\end{equation}
Therefore, we have
\begin{align*}
\sublogp_j & \stackrel{\phantom{\eqref{eq:betaterm},\eqref{eq:crossone}}}{=}
 -\sum\nolimits_{i=j+1}^d \frac{\beta_i}{n^{\gamma_i}} + 
\begin{cases}
0 & \text{if }  p_j>1,\\
\log(\beta_j) & \text{if }  p_j\le 1
\end{cases}\\
& \stackrel{\phantom{\eqref{eq:betaterm},\eqref{eq:crossone}}}{=}
 - \sum\nolimits_{i=j+1}^{d}\frac{\left(1+\xi\right)\bar \beta_i + \bar \alpha_i \xi \log n}{n^{\bar\gamma_i}} +
\begin{cases}
0 & \text{if }  p_j>1,\\
\log(\bar \beta_j) + \log (1+\xi) & \text{if }  p_j\le 1
\end{cases}\\
& \stackrel{\phantom{\eqref{eq:betaterm},\eqref{eq:crossone}}}{=} \bar \sublogp_j - \xi \sum\nolimits_{i=j+1}^d \frac{\bar \beta_i
 + \bar \alpha_i \log n }{n^{\bar \gamma_i}} + O(\xi)
+ \begin{cases}
\log(\bar \beta_j) & \text{if }  p_j\le 1 \le \bar p_j,\\
0 & \text{otherwise}
\end{cases}\\
& \stackrel{\eqref{eq:betaterm},\eqref{eq:crossone}}{=}  \bar \sublogp_j - (1+o(1)) \
\xi \sum\nolimits_{i=j+1}^d 
 \bar \alpha_i \log n ,
\end{align*}
as required.
\qed

\smallskip
\subsection{Proof of Corollary~\ref{cor:expectatcritwind}}
\label{sec:proofexpectcritwind}
For any $j\le k\le d$, Proposition~\ref{prop:scaledparameters} applied 
with $\xi = \frac{c_n}{\log n}$ tells us that
\begin{align} \label{eq:paramcritwind}
\logp_k &= \bar\logp_k, &
\sublogp_k &= \bar\sublogp_k - (1+o(1)) (k-j+1)
c_n \sum\nolimits_{i=j+1}^d \bar\alpha_i, &
\constp_k &= \bar\constp_k.
\end{align}

If $k$ is not a critical dimension, by
Definition~\ref{def:criticaldirection} and
Lemma~{\ref{lem:expectedNjk}} 
we have
\[ \log \left(\EE(\bar X_{j,k}) \right) = \bar \logp_k \log n + \bar \sublogp_k + 
\bar\constp_k  + o(1) \to - \infty,\]
where $\bar X_{j,k}$ denotes the number of copies of $\Njk$ in $\cG(n,\bar\allp)$
(i.e.\ for $\tim=1$)
and thus $\EE(\bar X_{j,k})  = o(1)$.
Hence, by applying Lemma~\ref{lem:expectedNjk}  at time $\tim$ we have
\[ \EE(\varNjk) \stackrel{\eqref{eq:paramcritwind}}{=}
(1+o(1)) \EE(\bar X_{j,k}) = o(1), \]
as required.

If $k$ is a critical dimension,
we have $\bar\logp_k = 0$ and $\bar\sublogp_k = O(1)$ 
(see Definition~\ref{def:criticaldim}).
Thus by Lemma~\ref{lem:expectedNjk} we have
\begin{align*}
\EE(\varNjk) &\stackrel{\phantom{\eqref{eq:paramcritwind}}}{=}  \exp(\logp_k \log n + \sublogp_k +
\constp_k + o(1) )\\
& \stackrel{\eqref{eq:paramcritwind}}{=} (1+o(1)) \frac{\exp(\bar\sublogp_k +
\bar\constp_k)}{\exp\left((k-j+1)c_n \sum_{i=j+1}^d 
\bar\alpha_i\right)}\\
& \stackrel{\phantom{\eqref{eq:paramcritwind}}}{=} 
(1+o(1))\exp (\bar\sublogp_k + \bar\constp_k + c(\bar\gamma_k - j -1) ),
\end{align*}
where we are using 
that $0=\bar\logp_k = j+1 - \bar\gamma_k -
(k-j+1)\sum_{i=j+1}^d \bar\alpha_i$.
\qed

\smallskip
\subsection{Proof of Lemma~\ref{lem:manyNjk}}
\label{sec:proofmanyNjk}

We will prove the lemma with $c= \bar \alpha_{k_0}/4$, where
$k_0$ is as defined in Definition~\ref{def:indices}~\ref{def:indices:kzero}.

First observe that by Proposition~\ref{prop:scaledparameters} applied with
$\xi=-1/\omega_0$, for
any $j\le k\le d$ we have 
$\logp_k = \bar\logp_k$, $\constp_k = \bar\constp_k$, and
\begin{align*}
\sublogp_k  &= \bar \sublogp_k + (1+o(1))
(k-j+1) \frac{1}{\omega_0} \sum\nolimits_{i=j+1}^d \bar \alpha_i \log n\ge \bar \sublogp_k + \frac{\bar\alpha_{k_0} \log n}{
2\omega_0},
\end{align*}
where we are using that $k-j+1 \ge 1$.
Thus we have
$$
\logp_k\log n + \sublogp_k + \constp_k \ge \bar \logp_k\log n + \bar \sublogp_k + \bar \constp_k + \frac{\bar \alpha_{k_0}\log n}{2\omega_0},
$$
therefore Lemma~\ref{lem:expectedNjk} and the fact that $\bar \logp_k\log n + \bar \sublogp_k + \bar \constp_k =O(1)$ (because 
$k$ is a critical dimension) imply that
\begin{equation}\label{eq:boundexpectatomega0}
\EE(\varNjk) \ge \exp\left(\frac{\bar \alpha_{k_0}\log n}{2\omega_0} + O(1)\right)
\ge \exp\left(\frac{\bar \alpha_{k_0}\log n}{3\omega_0}\right),
\end{equation}
and thus in particular $\EE(\varNjk) = \omega(1)$.

In order to apply a second moment argument, we will show that 
\[\EE((\varNjk)^2) = (1+o(1))\EE(\varNjk)^2 ,\] 
implying that whp $\varNjk$ is concentrated around its expectation.
We first consider the case when $k\ge j+1$.

Let $\mathcal{T}_k$ denote the family of pairs $T=(K,C)$, where $K \subset [n]$ with $|K|=k+1$ and $C$ is a $j$-subset of $K$.
Each of these pairs may form a copy of $\Njk$ with $K$ as $k$-simplex and $C$ as
centre of the flower $\mathcal{F}(K,C)$.

Given two pairs $T_1=(K_1,C_1)$ and $T_2=(K_2,C_2)$, we define
\begin{itemize}

\item $s = s(T_1,T_2) \ :=\  \begin{cases}1 & \mbox{if } K_1 = K_2,\\
2 & \mbox{otherwise}; \end{cases}$

\item $\mathcal{F}_h \ :=\  \cF(K_h,C_h)$ for $h =1,2$;

\item $t=t(T_1,T_2) \ :=\  |\mathcal{F}_1 \cup \mathcal{F}_2|$, 
i.e.\ the total number of (potential) petals. 

\end{itemize}
By Proposition~\ref{prop:localisedprob}, the probability that two pairs in $\mathcal{T}_k$ both form a copy 
of $\Njk$ is $(1+o(1))p_k^s \bq^{\;\tim t}$. With this observation, we can
determine the contribution to $\EE((\varNjk)^2)$ made by those pairs with a fixed value of $s$.
\begin{itemize}
\item $s=1$. Petals can be shared, but certainly $t\geq k-j+1$ and the contribution is at most of order
\[ O\left(n^{k+1} p_k \bq^{\;\tim(k-j+1)}\right) \stackrel{\eqref{eq:expXjk}}{=} O(\EE(\varNjk))
  = o(\EE(\varNjk)^2).\]

\item $s=2$. By definition, a petal cannot lie in any other $k$-simplex
  and thus only the pairs with $t=2(k-j+1)$ have a positive probability
  of both forming a copy of $\Njk$. The number of such pairs is
  \begin{equation*}
    (1+o(1))\binom{n}{k+1}^2\binom{k+1}{j}^2
  \end{equation*}
  and thus these pairs provide a contribution of
  \begin{equation*}
    (1+o(1))\binom{n}{k+1}^2
    \binom{k+1}{j}^2
    p_k^2
    \bq^{\;\tim 2(k-j+1)}
    \stackrel{\eqref{eq:expXjk}}{=} (1+o(1))\EE(\varNjk)^2.
  \end{equation*}
\end{itemize} 
In total, we therefore have $\EE((\varNjk)^2) = (1+o(1))\EE(\varNjk)^2$.

We now consider the case $k=j$. The proof is similar but simpler, since for a pair
$(K,C)$ to form a copy of $\Njj$
we only require $K$ to be an isolated $j$-simplex,
and $C$ to be the canonical choice (see Definition~\ref{def:Njj}).
On the other hand, we need to be careful if $p_j>1$,
since then $p_j$ must be replaced by $1$ in any probability calculations.

Recall that since $j$ is a critical dimension we have that 
$\bar\logp_j\log n + \bar\sublogp_j + \bar\constp_j = O(1)$
(see Definition~\ref{def:criticaldim}).
For the second moment of $\varisol$, we count pairs of isolated $j$-simplices according to 
the size of their intersection $i$.
Applying Proposition~\ref{prop:localisedprob}, we obtain
\begin{align*}
\EE((\varisol)^2) & \stackrel{\phantom{\eqref{eq:expXjj}}}{=}\EE(\varisol) + \sum_{i=0}^j \binom{n}{j+1} \binom{j+1}{i}\binom{n-j-1}{j+1-i} \min\{p_j,1\}^2  \cdot (1+o(1))\bq^{\;2\tim}\\
& \stackrel{\eqref{eq:expXjj}}{\le} \EE(\varisol)^2 \left(1+o(1)+\sum_{i=1}^j O(n^{-i})\right)\\
&  \stackrel{\phantom{\eqref{eq:expXjj}}}{=} \EE(\varisol)^2 \left(1+o(1)\right).
\end{align*}

Thus in both cases we have
$\EE((\varNjk)^2) = (1+o(1))\EE(\varNjk)^2$ and so 
by Chebyshev's inequality  whp
\[\varNjk = (1+o(1))\EE(\varNjk)\stackrel{\eqref{eq:boundexpectatomega0}}{>} 
\exp\left(\frac{\bar \alpha_{k_0}\log n}{4\omega_0}\right),\]
as required.
\qed

\smallskip
\subsection{Proof of Claim~\ref{claim:firstbirthtime}}
\label{sec:proofoffirstbirthtime}
We split the proof into two cases,
according to which of $\bar p_k$ and
$\bar p_{k_0}$
is larger. 
In both cases, we will use the
fact that, since
$k_0 \ge j+1$, 
by Remark~\ref{rem:pkprob} we have 
$\frac{1}{\bar p_{k_0}} = \omega(1)$, and thus
\begin{equation} \label{eq:approx1/pk0}
\frac{1}{\bar p_{k_0}} - \tim^{-} = (1+o(1)) \frac{1}{\bar p_{k_0}}.
\end{equation}

\vspace{0.3cm}

\noindent \emph{Case 1:} $\frac{1}{\bar{p}_k} \ge \frac{1}{\bar{p}_{k_0}}$.
Let $S$ be the event that $\tim_K \le 
\frac{1}{\bar p_{k_0}}$.
Note that $L_3' \subseteq S$, hence
\begin{equation} \label{eq:L3'inS}
\Pr \big(L_3' \;|\; (L_1 \land L_2) \big)
= \Pr \big( S \;| \; (L_1 \land L_2) \big)
\Pr \big( L_3' \; | \; ( L_1 \land L_2 \land 
S) \big).
\end{equation}

Recall that conditioned on $L_1$, the birth time $\tim_K$ is uniformly distributed in $[\tim^-, \frac{1}{\bar p_k}]$. 
Therefore, since $S$ is independent of $L_2$, we have
\begin{equation}\label{eq:SgivenL1L2}
\Pr \big( S \;| \; (L_1 \land L_2) \big)
= \Pr ( S \;| \;  L_1)
= \frac{1/\bar{p}_{k_0} - \tim^-}{1/\bar{p}_{k} - \tim^-}
\stackrel{\eqref{eq:approx1/pk0}}{=}
(1+o(1)) \frac{\bar p_k}{\bar{p_{k_0}} (1 -
\tim^- \bar p_k)}.
\end{equation}
Moreover, conditioned on $S$, the set $K$ has the same
birth time distribution as the $z$ many bad 
$(k_0+1)$-sets in $\cZ$ and all these birth times are independent, thus
\begin{equation}\label{eq:L3'givenS}
\Pr \big( L_3' \; | \; ( L_1 \land L_2 \land 
S) \big) = \frac{1}{1+z} 
=
(1+o(1)) \frac{1}{z}\,,
\end{equation}
because $z=\omega(1)$ by~\eqref{eq:badk0+1sets} and the fact that
$p_{k_0} = o(1)$.

Putting \eqref{eq:L3'inS}, \eqref{eq:SgivenL1L2}, and \eqref{eq:L3'givenS}
together, we have
\begin{equation*}
\Pr \big(L_3' \;|\; (L_1 \land L_2) \big)
= (1+o(1)) \frac{\bar p_k}{z \bar p_{k_0}
(1 -  \tim^- \bar  p_k  )},
\end{equation*}
as claimed.

\vspace{0.3cm}

\noindent \emph{Case 2:} 
$\frac{1}{\bar{p}_k} < \frac{1}{\bar{p}_{k_0}}$.
First observe that if $k=j$, by \eqref{eq:condlocalobst} we have
$\bar p_j < 1 + \frac{\log \log n}{9d \log n}$, while if $k\ge j+1$, then
$\bar p_k <1$  by Remark~\ref{rem:pkprob}.
Thus by the definition of $\tim^-$,
for any $k\ge j$ we have
\begin{equation}\label{eq:boundon1-timpk}
1 - \tim^- \bar p_k > \frac{\log \log n}{12 d \log n}.
\end{equation}

Let $\cZ_{\bar{p}_k}$ be
the set of bad $({k_0}+1)$-sets 
in $\cZ$ with birth times in the interval $[\tim^-,\frac{1}{\bar{p}_k}]$
and let $\zeta_k\ :=\  |\cZ_{\bar{p}_k}|$.
Since the birth times of the sets in $\cZ$
are uniformly distributed in 
$[\tim^-, \frac{1}{\bar p_{k_0}}]$,
the random variable
$\zeta_k$ has binomial distribution
$\text{Bi}\left(z, \frac{1/\bar{p}_k - \tim^-}{1/\bar{p}_{k_0} - \tim^-} \right)$
and observe that
\begin{equation}\label{eq:expzetak_1}
\EE(\zeta_k) = z\cdot \frac{1/\bar{p}_k - \tim^-}{1/\bar{p}_{k_0} - \tim^-}
 \stackrel{\eqref{eq:approx1/pk0}}{=}(1+o(1))\frac{z\bar p_{k_0}}{\bar p_k} 
 (1 - \tim^- \bar p_k).
\end{equation} 
Since $z \bar p_{k_0} = \Theta(\log n)$
by \eqref{eq:badk0+1sets}
and $\bar p_k < 1 + \frac{\log \log n}{9d \log n}$, we obtain
\begin{equation*}
\EE(\zeta_k) 
\stackrel{\eqref{eq:boundon1-timpk}}{=}
\Omega \left( \log n  \cdot \frac{\log \log n}{12d \log n}  \right)= \Omega
 \left(\log \log n\right) \to 
 \infty.
\end{equation*}

By the Chernoff bound,
the probability that $\zeta_k$ is not within a multiplicative factor
$1\pm \frac{1}{\EE(\zeta_k)^{1/4}}$
of the mean is 
$\exp(-\Omega(\EE(\zeta_k)^{1/2}))$.

Furthermore, conditioned on the value of 
$\zeta_k$ (and the events $L_1$, $L_2$), 
the probability
of $L_3'$ 
is $\frac{1}{1+\zeta_k}$, 
because the birth times of $K$ and 
of the bad sets in $\cZ_{\bar{p}_k}$ all have
the same (conditional) distribution.
Thus, since $\EE(\zeta_k) \to \infty$,
we have
\begin{align*}
\Pr\big(L_3' \; | \; (L_1 \land L_2) \big) 
&\stackrel{\phantom{\eqref{eq:expzetak_1}}}{=} \frac{1+o(1)}{1 + \left(1 \pm \frac{1}{
\EE(\zeta_k)^{1/4}}\right) \EE(\zeta_k)}
+ 
\exp(-\Omega(\EE(\zeta_k)^{1/2}))
\\
&\stackrel{\phantom{\eqref{eq:expzetak_1}}}{=} \frac{1+o(1)}{\EE(\zeta_k)}\\
& \stackrel{\eqref{eq:expzetak_1}}{=}
(1+o(1)) \frac{\bar{p}_k}{z\bar{p}_{k_0}
(1 - \tim^- \bar p_k) },
\end{align*} 
as claimed.
\qed

\smallskip
\subsection{Proof of Claim~\ref{largesupp:numberpairs}}
\label{sec:proofclaim:numberpairs}
Consider the $j$-simplex in step $i$ of the exploration process
described in Remark~\ref{rem:exploration}: there are
at most $\binom{n}{k-j}$ many $(k+1)$-sets which we could potentially discover from this
$j$-simplex, and of these we must choose $b_{i,k}$. 
From each of the chosen $k$-simplices,
we find at most $\binom{k+1}{j+1}-1<\binom{k+1}{j+1}$
undiscovered $j$-simplices of $\cS$. 
This holds for every $k \in \{j+1,\ldots,d\}$, thus this can happen in at most
$\prod_{k=j+1}^{d} \binom{\binom{n}{k-j}}{b_{i,k}} 2^{\binom{k+1}{j+1}b_{i,k}}$ different ways. 
Hence, considering the choices for the initial $j$-simplex,
the number of pairs $(\cS,\cT(\cS))$ with $\cS$ traversable and with exploration matrix $B$
is bounded from above by
\[ \binom{n}{j+1} \prod\nolimits_{i,k} \binom{\binom{n}{k-j}}{b_{i,k}} 2^{\binom{k+1}{j+1}b_{i,k}} \le n^{j+1} \frac{\prod_k \left( \binom{n}{k-j} 2^{\binom{k+1}{j+1}}\right)^{t_k}}{\prod_{i,k} b_{i,k}!},\]
using that $\sum_{i=1}^s b_{i,k} = t_k$, for each $k \in \{j+1,\ldots,d\}$.\qed

\smallskip
\subsection{Proof of Claim~\ref{largesupp:meshwal}}
\label{sec:proofclaim:meshwal}
The expected number of $j$-sets that do not form a $(j-1)$-simplex in $\cG_\tim$ is bounded from above by
\begin{align*}
\binom{n}{j} (1 - p_{k_0})^{\binom{n-j}{k_0-j+1}}
& \le \binom{n}{j} (1 - \tim_0 \bar p_{k_0})^{\binom{n-j}{k_0-j+1}} \\
& \le (1+o(1)) n^j
\exp\left( - \tim_0 n \cdot \frac{\bar\alpha_{k_0}\log n + \bar\beta_{k_0}}{k_0-j+1} \right) = o(1),
\end{align*}
thus by Markov's inequality whp $\cG_\tim$ has a complete $(j-1)$-dimensional skeleton.\qed

\smallskip
\subsection{Proof of Claim~\ref{largesupp:badsets}}
\label{sec:proofclaim:badsets}
Recall that we condition on the high probability event in
Claim~\ref{largesupp:meshwal} and that this implies that
Proposition~\ref{prop:meshwal} can be applied to $f_\tim$ (if it
exists). Suppose that $f_\tim$ exists and write $s := |\cS_\tim|$.
Then $\cD(f_\tim)$ comprises at least $\frac{sn}{j+2}$ many
$(j+2)$-sets by Proposition~\ref{prop:meshwal}. Each $A\in\cD(f_\tim)$ is
not allowed to be part of $k$-simplices of $\cG_\tim$, for every $k
\in \{j+1,\ldots,d\}$. There are $\binom{n-j-2}{k-j-1}$ many
$(k+1)$-sets in $[n]$ that contain $A$, each of which contains
$\binom{k+1}{j+2}$ many $(j+2)$-sets. Thus we have
\begin{equation*}
  |\cD_k(f_\tim)|
  \ge \frac{sn \binom{n-j-2}{k-j-1}}{(j+2) \binom{k+1}{j+2}}
  \ge \largecon_0 s n^{k-j}
\end{equation*}
for some positive constant $\largecon_0$ and for every $k=j+1,\ldots,d$.\qed

\smallskip
\subsection{Proof of Claim~\ref{largesupp:rB}}
\label{sec:proofclaim:rB}
Let a pair $(\cS,\cT(\cS))$ with $\allt(\cS)=(t_{j+1},\ldots,t_d)$ be
given and recall that $s \ge \sum_{k=j+1}^{d} t_k =: Y_{\allt}$ by
\eqref{eq:simpltrav}. By Claim~\ref{largesupp:badsets}, the
probability that $\cS_\tim=\cS$ is at most
\begin{equation}\label{eq:xtim}
  \prod\nolimits_{k=j+1}^d  p_k^{t_k} (1-p_k)^{\largecon_0 n^{k-j} s} \le \tim^{Y_{\allt}} \prod\nolimits_{k=j+1}^d \bar{p}_k^{t_k} (1-\tim \bar{p}_k)^{\largecon_0 n^{k-j} Y_{\allt}} =: x(\tim).
\end{equation}
The function $x(\tim)$ is positive and its derivative (with respect to
$\tim$) is
\[ \frac{dx}{d\tim} = \frac{x(\tim) Y_{\allt}}{\tim} \left(1 - \tim \largecon_0  \sum\nolimits_{k=j+1}^d \frac{n^{k-j} \bar{p}_k}{1 - \tim \bar{p}_k} \right). \]
Recalling that the index $k_0$ is such that $\bar\alpha_{k_0} \neq 0$ and $\bar\gamma_{k_0} = 0$, we deduce that
\[ \frac{dx}{d\tim} \le \frac{x(\tim) Y_{\allt}}{\tim} \left( 1 - \tim \largecon_0 \frac{\bar\alpha_{k_0} \log n + \bar\beta_{k_0}}{1 - \tim \bar{p}_{k_0}} \right)<0\]
for $\tim =\omega(1/\log n)$. 
Thus, since the derivative of $x(\tim)$ is negative throughout the whole range $\tim\ge \tim_0=1-o(1)$,
we have $x(\tim) \le x(\tim_0)$ for all $\tim \ge \tim_0$, and therefore
in the following calculations we may substitute $\tim_0 $ for $\tim$.

Now Claim~\ref{largesupp:numberpairs} implies that
\begin{align*}
  r_B \prod\nolimits_{i,k} b_{i,k}!
  &\stackrel{\phantom{\eqref{eq:xtim}}}{\le} x(\tim_0) n^{j+1} \prod\nolimits_{k=j+1}^d \left( \binom{n}{k-j} 2^{\binom{k+1}{j+1}} \right)^{t_k}\\
  &\stackrel{\eqref{eq:xtim}}{\le} n^{j+1} \prod\nolimits_{k=j+1}^d \left( \binom{n}{k-j} 2^{\binom{k+1}{j+1}} \tim_0 \bar{p}_k\right)^{t_k} \left( 1 - \tim_0 \bar{p}_k \right)^{\largecon_0 n^{k-j} Y_{\allt}}\\
  &\stackrel{\,\eqref{eq:parameters}\,}{\le} n^{j+1} \prod\nolimits_k \Bigg( \left( \Theta(1) \frac{\bar\alpha_k \log n + \bar\beta_k}{n^{\bar\gamma_k}} \right)^{t_k}  \\
  & \hspace{3cm} \cdot \exp \left(-(1+o(1)) \largecon_0 Y_{\allt} \left( \frac{\bar\alpha_k \log n + \bar\beta_k}{n^{\bar\gamma_k}} (k-j)! \right) \right) \Bigg)\\
  &\stackrel{\phantom{\eqref{eq:xtim}}}{\le} n^{j+1} \left( O(\log n)\right)^{Y_\allt} n^{-(1+o(1)) \largecon_0 \bar\alpha_{k_0}(k_0-j)! Y_\allt} \\
  &\stackrel{\phantom{\eqref{eq:xtim}}}{\le} n^{j+1} n^{-\hat\largecon_0 Y_{\allt}},
\end{align*}
where $\hat\largecon_0\ :=\ \frac{\largecon_0\bar \alpha_{k_0}(k_0-j)!}{2}$.
Now suppose that $s \ge \tilde{\largecon}_2 \ge \frac{2(j+1)\sum_{k}\binom{k+1}{j+1}}{\hat\largecon_0}$.
Since $Y_{\allt} \ge \frac{s}{\sum_k \binom{k+1}{j+1}}$, we can find another positive constant $\largecon_1$ such that
\begin{equation*}
  r_B \le \frac{n^{- \largecon_1 s}}{\prod_{i,k}b_{i,k}!},
\end{equation*}
as desired.\qed

\smallskip
\subsection{Proof of Claim~\ref{largesupp:singleh}}
\label{sec:proofclaim:singleh}
For any exploration matrix $B=(b_{i,k})$, define $\allu(B) = (u_{j+1}, \ldots, u_d)$, where
\[ u_k \ :=\  | \{ i :  b_{i,k} \ge n^{\largecon_1/(d-j+1)}\}|.\]
Conversely, given $\allu=(u_{j+1},\ldots,u_d)$ with $u_k \ge 0$, let
$\cB_{\allu}$ be 
the set of all matrices $B$ such that $\allu(B) = \allu$.
Observe that each $B\in \cB_\allu$ is an $(s\times (d-j))$-matrix. There
are $\prod_{k}\binom{s}{u_k}$ choices for which entries are large
(i.e.\ which contribute to $u_k$),
at most $n^{\largecon_1/(d-j+1)}$ possibilities for each of the small entries and,
since the sum of all the entries is $\sum_k t_k \le s$, at most $s$ possibilities
for each of the large entries.
Thus we obtain the (rather crude) upper bound
\begin{align}
|\mathcal{B}_{\allu}| & \le \left(\prod\nolimits_k \binom{s}{u_k}\right) \left( n^{\largecon_1/(d-j+1)}\right)^{s(d-j) - \sum_k u_k} s^{\sum_k u_k}\nonumber\\
&\le s^{2\sum_k u_k} \left( n^{\largecon_1/(d-j+1)}\right)^{s(d-j) - \sum_k u_k}.\label{eq:Bu}
\end{align}
Moreover, for $B \in \mathcal{B}_{\allu}$ 
\begin{equation}\label{eq:prodbik}
\prod\nolimits_{i,k} b_{i,k}! \ge \left( \left( n^{\frac{\largecon_1}{d-j+1}}\right)!\right)^{\sum_k u_k} \ge n^{n^{\frac{\largecon_1}{d-j+2}} \sum_k u_k}.
\end{equation}
Putting everything together, the probability $r_s$ that $\cS_\tim$ of fixed size $s\ge \tilde{\largecon}_1$ exists (together with the collection $\cT(\cS_\tim)$ of simplices) satisfies
\begin{align*}
  r_s &\stackrel{\phantom{\eqref{eq:Bu},\eqref{eq:prodbik}}}{\le} \sum\nolimits_{\allu} \sum\nolimits_{B \in \mathcal{B}_{\allu}} r_B
  \stackrel{\mbox{\footnotesize (Cl.~\ref{largesupp:rB})}}{\le} \sum\nolimits_{\allu} |\mathcal{B}_{\allu}| \frac{n^{-\largecon_1 s}}{\prod_{i,k} b_{i,k}!}\\
  &\stackrel{\eqref{eq:Bu},\eqref{eq:prodbik}}{\le} \sum\nolimits_{\allu} \left[\prod\nolimits_k \left( \frac{s^2}{n^{\frac{\largecon_1}{d-j+1}} n^{n^{\frac{\largecon_1}{d-j+2}}}}\right)^{u_k}\right] \frac{n^{\frac{\largecon_1}{d-j+1}s(d-j)}}{n^{\largecon_1 s}}\\
  &\stackrel{\phantom{\eqref{eq:Bu},\eqref{eq:prodbik}}}{\le} \sum\nolimits_{\allu} 1 \cdot n^{-\frac{\largecon_1}{d-j+1}s}
  \le (s+1)^{d-j}\cdot n^{-\frac{\largecon_1}{d-j+1}s} \le n^{-\largecon_2 s},
\end{align*}
for some positive constant $\largecon_2$.

Thus, the probability that $\cS_{\tim}$ exists with $|\cS_{\tim}|\ge
\tilde{\largecon} \ge \tilde{\largecon}_1$ is at most
\[\sum\nolimits_{s\ge\tilde{\largecon}} r_s \le \sum\nolimits_{s\ge\tilde{\largecon}}
n^{-\largecon_2 s} \le \frac{n^{-\largecon_2 \tilde{\largecon}}}{1-n^{-\largecon_2 }} \le n^{-\largecon_2 \tilde{\largecon} / 2},\]
which concludes the proof of Claim~\ref{largesupp:singleh}.\qed

\newpage

\section{Further standard proofs}\label{sec:proofaux2}

In this appendix we provide for completeness some further proofs which are simply applications of standard ideas and techniques,
or which are obvious generalisations of previously existing proofs.

\subsection{Proof of Lemma~\ref{lem:topconn}}
\label{sec:prooftopconn}
\ref{topconn:no}
Since topological connectedness is a monotone property, it is enough to prove the statement
in the case when $p_k=\frac{c^-\log n}{n^k}$ for all $k\in [d]$, which will be convenient in the proof.

Let $U$ be the number of isolated vertices in $\cG(n,\allp)$. 
Every vertex is contained in $\binom{n-1}{k}$ many $(k+1)$-sets, each of which does not form an $i$-simplex with probability $(1-p_k)$, all independently. 
Hence each vertex is isolated with probability $\prod_{k=1}^{d}(1-p_k)^{\binom{n-1}{k}}$, and therefore we have
\begin{align*}
\EE(U) = n \prod\nolimits_{k=1}^{d}(1-p_k)^{\binom{n-1}{k}} 
&\geq n \cdot \exp \left( - \sum\nolimits_{k=1}^{d} \frac{n^k}{k!} \left(p_k+ O(p_k^2)\right) \right) \\
&\geq n \cdot \exp\left( - \sum\nolimits_{k=1}^{d} \frac{n^k}{k!} \cdot \frac{c^- \log n}{n^k} + o(1) \right)\\
& = (1+o(1)) n^{1-\tilde dc^{-}},
\end{align*}
where $\tilde d\ :=\  \sum_{k=1}^d 1/k!$.
Moreover, the probability that two fixed distinct vertices are both isolated is
\begin{align*}
\prod\nolimits_{k=1}^d (1-p_k)^{2 \binom{n-1}{k} - \binom{n-2}{k-1}} &\leq \exp \left(-\sum\nolimits_{k=1}^d p_k \binom{n-1}{k} \left( 2-\frac{k}{n-1} \right)\right) \\
&\leq\exp \left( -2\sum\nolimits_{k=1}^d \frac{c^-\log n}{k!} \left(1+O\left(\frac{1}{n}\right)\right)\right)\\
&\leq (1+o(1)) n^{-2\tilde dc^-}.
\end{align*}
By choosing $c^-$ such that $c^-<1/\tilde d$ (so in particular $\EE(U) \to \infty$), we obtain 
\begin{align*}
 \EE(U^2) &\leq \EE(U) + n(n-1) (1+o(1)) n^{-2\tilde dc^-} \\
 &=\EE(U) + (1+o(1)) n^{2(1-\tilde dc^-)} = (1+o(1))\EE(U)^2,
\end{align*} so by Chebyshev's inequality whp there are isolated vertices, implying that whp $\cG(n,\allp)$ is not topologically connected.\newline
\ref{topconn:yes}  Consider $\tilde\allp$ obtained from $\allp$ by replacing all probabilities except $p_k$ by zero,
where $k\in [d]$ is an index such that $p_k\ge \frac{c^+\log n}{n^k}$.
If $k=1$, then \ref{topconn:yes} follows from the corresponding results
for graphs. For $k \ge 2$, \ref{topconn:yes} holds because of the fact
that we can choose $c^+$ such that whp $\cG(n,\tilde\allp)$ is
topologically connected by \cite[Lemma~4.1]{CooleyDelGiudiceKangSpruessel20}.
\qed

\smallskip
We note that the proof idea of Lemma~\ref{lem:topconn} is a standard generalisation
of the very well-known hitting time result for graphs: whp the random graph process becomes
connected at exactly the moment its last isolated vertex disappears. Indeed,
Theorem~\ref{thm:main} is also a generalisation of this result, albeit a far more
complex one.

The vertex-connectedness threshold for uniform random hypergraphs, which we quoted from~\cite{CooleyDelGiudiceKangSpruessel20}
for the proof of \ref{topconn:yes}, also follows as a special case of earlier and
much stronger results from~\cite{CooleyKangKoch16}
and from~\cite{Poole15}. The proof in~\cite{CooleyDelGiudiceKangSpruessel20} has the advantage that it
is a simple and elementary extension of the standard graph argument.

\smallskip
\subsection{Proof of Proposition~\ref{prop:localisedprob}}
\label{sec:proofoflocalisedprob}
We first observe that for $j+1\le k \le d$, the number of $(k+1)$-sets which
contain at least two distinct $(j+1)$-sets of $\cJ$ is at most
$\binom{|\cJ|}{2}\binom{n}{k-j-1} = O(n^{k-j-1})$, and therefore
the number of $(k+1)$-sets that must \emph{not} be $k$-simplices in order for $A$ to hold
is
$$
|\cJ|\binom{n-j-1}{k-j} - O(n^{k-j-1}) - |\cS| = |\cJ|\binom{n-j-1}{k-j} - O(n^{k-j-1}).
$$
Thus, since $p_k n^{k-j-1}=o(1)$, we have
\begin{align*}
\Pr(A) &  \stackrel{\phantom{\eqref{eq:q_computing}}}{=} \prod\nolimits_{k=j+1}^d (1-p_k)^{|\cJ|\binom{n-j-1}{k-j} - O(n^{k-j-1})}\\
& \stackrel{\phantom{\eqref{eq:q_computing}}}{=} (1+o(1)) \prod\nolimits_{k=j+1}^d (1-\tim \bar p_k)^{|\cJ|\binom{n-j-1}{k-j}} \\
& \stackrel{\phantom{\eqref{eq:q_computing}}}{=} (1+o(1)) \exp \left( - |\cJ| \tim
 \sum\nolimits_{k=j+1}^{d}  
 \binom{n-j-1}{k-j} \bar p_k 
 \right)\\
& \stackrel{\eqref{eq:q}}{=} (1+o(1)) \bq^{\; \tim |\cJ|},
\end{align*}
as claimed.
\qed

\smallskip
\subsection{Proof of Proposition~\ref{prop:meshwal}}
\label{sec:proofmeshwal}
Given an ordered $m$-simplex $\Phi=[v_0,\ldots,v_m]$ and a vertex $u \notin \Phi$, define the ordered $(m+1)$-simplex
\[ [u,\Phi]\ :=\  [u,v_0,\ldots,v_{m}],\]
for any $m\in[n-2]_0$.

Let $v \in [n]$ and consider a $j$-cochain $f$ as in the statement. We define the $(j-1)$-cochain $f_v$ that maps every $(j-1)$-simplex $\rho$ to the value 
\[ f_v(\rho) = 
\begin{cases}
f([v,\rho]) & \mbox{if } v \notin \rho; \\
0_R & \mbox{otherwise}.
\end{cases}\]
For any ordered $j$-simplex $\sigma=[v_0,\ldots,v_j]$ we have
\begin{equation}\label{eq:meshwall1}
 f(\sigma) - (\delta^{j-1} f_v)(\sigma) = 
f(\sigma) - \sum\nolimits_{i=0}^j (-1)^i f_v \left([v_0,\ldots,\hat{v_i},\ldots,v_j]\right). 
\end{equation}

If $v \notin \sigma$, then
\begin{equation*}
f(\sigma) - (\delta^{j-1} f_v)(\sigma) \stackrel{\eqref{eq:meshwall1}}{=} f(\sigma) - \sum\nolimits_{i=0}^j (-1)^i f \left([v,v_0,\ldots,\hat{v_i},\ldots,v_j]\right) = (\delta^jf)([v,\sigma]),
\end{equation*}
by definition of the operator $\delta^{j}$.

If $v \in \sigma$ then $v=v_\ell$ for some $\ell \in [j]_0$, implying that $f_v([v_0,\ldots,\hat{v_i},\ldots,v_j])=0$ for every $i\neq \ell$ and $f_v([v_0,\ldots,\hat{v_\ell},\ldots,v_j])=(-1)^{\ell} f(\sigma)$. Thus
\[ f(\sigma) - (\delta^{j-1} f_v)(\sigma) \stackrel{\eqref{eq:meshwall1}}{=} f(\sigma) - (-1)^{2\ell} f(\sigma) = 0_R.\]

Putting everything together
\begin{equation}\label{eq:meshwall2}
f(\sigma) - (\delta^{j-1} f_v)(\sigma) = 
\begin{cases}
(\delta^j f)([v,\sigma]) & \mbox{if } v \notin \sigma; \\
0_R & \mbox{otherwise}.
\end{cases}
\end{equation}

Recalling that every $j$-cochain of the form $f+g$ with $g$ a $j$-coboundary
has support of size at least $|\cS|$, we have 
\begin{equation} \label{eq:meshwall3}
n|\cS| \le \sum_{v\in[n]}|\text{supp}(f - \delta^{j-1} f_v)|  = |\{(v,\sigma) : v\in [n],\text{ } \sigma \in \text{supp}(f - \delta^{j-1} f_v) \}|.
\end{equation}
For a pair $(v,\sigma)$, by \eqref{eq:meshwall2} it holds that $\sigma$ is in the support of $f - \delta^{j-1} f_v$ if and only if $v \notin \sigma$ and the $(j+1)$-simplex $[v,\sigma]$ is in the support of $(\delta^{j}f)(u\sigma)$. Hence
\begin{align*}
n |\cS| &\stackrel{\eqref{eq:meshwall3}}{\le} |\{(v,\rho) : v\in \rho, \text{ } \rho \in \text{supp}(\delta^j f) \}|\\
&\stackrel{\phantom{\eqref{eq:meshwall3}}}{=} (j+2) |\text{supp}(\delta^j f)| \\
&\stackrel{\phantom{\eqref{eq:meshwall3}}}{=} (j+2) |\cD(f)|,
\end{align*}
as required.
\qed

\smallskip
\subsection{Multivariate Poisson approximation}
\label{sec:multipoisson}
In order to prove Lemma~\ref{lem:Njkdistribution}, 
we will use a method from \cite{BarbourHolstJansonBook} that we briefly explain in this section.

Given a discrete set $H$, the \emph{total variation distance} between the distributions
of two $H$-valued random variables $Y$ and $Z$ is defined by
\[d_{TV} \left(\mathcal{L}(Y),\mathcal{L}(Z) \right)  \ :=\  
\frac{1}{2} \sum\nolimits_{h \in H} | \Pr(Y = h) - \Pr(Z = h) |.\]

\begin{lem}[{\cite[Theorem 10.J]{BarbourHolstJansonBook}}] \label{lem:multpoisson}
Given a set $\Gamma$ with a partition 
$\Gamma = \dot\cup_{k=1}^r \Gamma_k$ and a collection $(I_a)_{a \in \Gamma }$ of indicator random variables
defined on a common probability space, let
\begin{itemize}
\item $\pi_a \ :=\  \Pr(I_a = 1)$, for every $a \in \Gamma$;
\item $W_k \ :=\  \sum_{a \in \Gamma_k} I_a$, for $k \in [r]$;
\item $\mathbf{W} \ :=\  (W_1,\ldots,W_r)$;
\item $m_k \ :=\  \EE(W_k) = \sum_{a \in \Gamma_k} \pi_a$, for $k \in [r]$;
\item $\mathbf{m}\ :=\ (m_1,\ldots,m_r)$.
\end{itemize}
Suppose that for each $a \in \Gamma$
there exist random variables $(J_{b a})_{b \in \Gamma}$ 
defined on the same probability space as $(I_b)_{b \in \Gamma}$ with
\[ \mathcal{L} \left( (J_{b a})_{b \in \Gamma}\right) = \mathcal{L} \left( (I_b )_{b \in \Gamma} | I_a = 1\right) .\]
Then
\[ d_{TV} \Big( \mathcal{L} \big( \mathbf{W} \big) ,
\Po(\mathbf{m})  \Big) 
\le 
\sum\nolimits_{a \in \Gamma} \pi_a \left( \pi_a + 
\sum\nolimits_{b \neq a} \EE|J_{b a} - I_b | \right),   \]
where $\Po(\mathbf{m})$ denotes the joint Poisson
distribution $(\Po(m_1)\ldots,\Po(m_r))$ and $\Po(0)\equiv 0$.
\end{lem}

It is easy to see that if there exists $\tilde{\mathbf{m}}\ :=\ (\tilde{m}_1,\ldots,
\tilde{m_r})$
 such that for every $k \in [r]$, 
 $\tilde{m}_k \in \mathbb{R}$ and 
$m_k=m_k(n) \xrightarrow{n \to \infty} \tilde{m}_k$, 
then
\[\mathcal{L} \big( \mathbf{W} \big) 
\xrightarrow{\enskip \dist \enskip}
\Po(\tilde{\mathbf{m}})  \]
if and only if
\[d_{TV} \left( \mathcal{L} \big( \mathbf{W} \big) ,
 \Po(\mathbf{m})  \right) \xrightarrow{n \to \infty} 0.\]

\smallskip
\subsection{Proof of Lemma~\ref{lem:Njkdistribution}}
\label{sec:proofofNjkdistr}

We will first show that we can apply Lemma~
\ref{lem:multpoisson} with $W_k = \varNjk$ and $m_k = 
\EE(\varNjk)$ for $k=j,\ldots,d$.
Subsequently, we show the bound on the total variation 
distance is indeed $o(1)$ and that $\lim_{n\to \infty}
\EE(\varNjk) = \critm_k$.

We want to define the set $\Gamma$ 
of potential copies of $\Njk$ 
in $\cG_\tim$ for each $j\le k\le d$.
As in the proof of Lemma~\ref{lem:manyNjk}, we consider
the sets
\begin{equation*}
  \cT_k = \left\{ (K,C) : K \in \binom{[n]}{k+1}, C \in \binom{K}{j} \right\},
\end{equation*}
for each $j+1\le k\le d$. Furthermore we define the set
$\cT_j$ analogously but with the additional condition  
that given a $(j+1)$-set $K$, then the set $C$ 
consists of the first $j$ vertices of $K$ according
to the increasing order of $[n]$ (cf.\
Definition~\ref{def:Njj}).
Following the notation of Lemma~\ref{lem:multpoisson},
we set $\Gamma \ :=\  \dot\cup_{k=j}^d \cT_k$ and we use 
$a=(K_a,C_a)$ to denote an element of $\Gamma$.

For any $a \in \cT_k \subseteq \Gamma$, we define the following quantities:
\begin{itemize}
\item $k_a \ :=\  k$;
\item $I_a$ is the indicator random variable of the event that $a$ forms a copy of $\Njk$;
\item $\pi_a \ :=\  \Pr(I_a = 1) = \EE(I_a)$;
\item $\mathcal{B}_a$ is the collection
 of \emph{forbidden sets} for $a$, i.e.\ 
\[\mathcal{B}_a = \{ B \subset [n] : |B|\le d+1 , B \not\subseteq K_a ,  B \supset P \mbox{ for some } P \in \mathcal{F}(K_a,C_a) \},\]
where
$\mathcal{F}(K_a,C_a) = 
\{C_a \cup \{w\} \mid w \in K_a \setminus C_a\}$ is 
the $j$-flower in $K_a$ with centre $C_a$ 
(see Definition~\ref{def:flower} 
and \eqref{eq:flowerinKwithcentreC}).
In other words, $\mathcal{B}_a$ is the collection 
of subsets of $[n]$ that are not allowed to be simplices 
in $\cG_\tim$ in order for $a$ to form a copy of $\Njk$
(cf.\ \ref{Njk:flower} in Definition~\ref{def:Njk}). 
\end{itemize}
Observe that if $a = (K_a,C_a) \in \cT_j$, then 
$\cF(K_a,C_a)=\{K_a\}$, therefore the set $\mathcal{B}_a$
consists of all subsets of $[n]$ (of cardinality
at most $d+1$) that contain $K_a$, except for $K_a$ itself.

Given a family $\mathcal{D}$ of sets of vertices,
we say that the indicator random variable of an event 
$E$ \emph{depends only on} 
$\mathcal{D}$ if $E$ only depends on whether the sets in 
$\mathcal{D}$ are simplices in $\cG_\tim$ or not.
Observe that by Definitions~\ref{def:Njk} and \ref{def:Njj} we can write
\[ I_a = \mathbbm{1}\; \big\{\{K_a \in \cG_\tim\}  \land 
\{B \not\in \cG_\tim, \forall B \in \mathcal{B}_a \}\big\},\]
therefore the random variable $I_a$ depends only on the family of sets 
\[\mathcal{D}_a \ :=\  \{ K_a\} \cup \mathcal{B}_a.\]

We now aim to define the random variables
$(J_{ba})_{b\in \Gamma}$ needed to apply
Lemma~\ref{lem:multpoisson}.
Given $a,b \in \Gamma$, we define the events
\begin{itemize}
\item $E^1_{ba} = \{ K_b \in \cG_\tim\} \lor \{ K_b = K_a\}$,
\item $E^2_{ba} = \{B \not\in \cG_\tim, \forall B \in \mathcal{B}_b \setminus \mathcal{B}_a \}$,
\item $E^3_{ba} = \{ K_b \not\in \mathcal{B}_a \} \land
\{ K_a \not\in \mathcal{B}_b \}$,
\end{itemize}
and the indicator random variable
\begin{equation}\label{eq:indicatorcoupling}
J_{ba} = \mathbbm{1} \; \{E^1_{ba} \land
E^2_{ba} \land E^3_{ba} \}.
\end{equation}
We claim that 
\begin{equation} \label{eq:conditioningonIalpha}
\mathcal{L} \left( (J_{b a})_{b \in \Gamma}\right) = \mathcal{L} \left( (J_{b a} )_{b \in \Gamma} | I_a = 1\right).
\end{equation}
To see this, let $\mathcal{D}_{ba}$ be the family of 
sets of vertices which $J_{ba}$ depends only on.
If $K_b \in \mathcal{B}_a$ or $K_a \in \mathcal{B}_b$, by~\eqref{eq:indicatorcoupling} 
and the definition of $E_{ba}^3$ we have
 $J_{ba}=0$ deterministically 
and we set $\mathcal{D}_{ba} \ :=\  \emptyset$. 
Otherwise, by~\eqref{eq:indicatorcoupling} we have
\[ \mathcal{D}_{ba} \ :=\  \left(\{ K_b\} \setminus \{K_a\}
 \right) \cup \left( \mathcal{B}_b \setminus 
 \mathcal{B}_a  \right) \]
and if $K_b = K_a$ then $K_a$ is not in 
$\mathcal{B}_b$ because the event
$E_{ba}^3$ holds, hence we have 
$\mathcal{D}_{b a} = \mathcal{D}_b \setminus 
\mathcal{D}_a$. In particular, this implies that
$\mathcal{D}_{ba}$ and $\mathcal{D}_{a}$ are 
always disjoint 
and this holds for every $b \in \Gamma$, 
thus the joint distribution of 
$(J_{b a})_{b \in \Gamma}$ does not change 
if we condition on $I_a = 1$, yielding 
\eqref{eq:conditioningonIalpha}.

We further claim that for every $b \in \Gamma$
\begin{equation} \label{eq:JiffI}
\big( (I_a=1) \land (J_{ba} = 1) \big) \quad \Longleftrightarrow
\quad \big( (I_a=1) \land (I_b = 1) \big).
\end{equation}
Suppose $I_a = J_{b a} = 1$. 
We have $K_b\in \cG_\tim$ by $E_{ba}^1$ and the fact that 
$K_a\in\cG_\tim$, since $I_a=1$.
Moreover, $I_a=1$
yields that none of the sets in $\mathcal{B}_a$ is in
$\cG_\tim$ and 
by definition of $E_{ba}^2$ also none of the sets in 
$\mathcal{B}_b \setminus \mathcal{B}_a$  
is in $\cG_\tim$, therefore in particular 
every set in $\mathcal{B}_b$ is not in $\cG_\tim$. 
Thus, by definition of
$I_b$ we have that $I_b=1$. 

Vice versa, suppose that $I_a=I_b=1$. 
By definition of $I_b$, clearly the events 
$E_{ba}^1$ and $E_{ba}^2$ hold. Moreover, $I_a$ and
$I_b$ can only both be equal to $1$ if $K_b$ is not
forbidden for $a$ and $K_a$ is not forbidden for
$b$, i.e.\ the event $E^3_{ba}$ must hold. Thus, 
it follows that $J_{b a}=1$. 
This proves \eqref{eq:JiffI}.

Hence, conditioned on $I_a=1$, for every $b\in \Gamma$
\eqref{eq:JiffI} yields that
$J_{ba}$ and $I_b$ are the same random variable,
and thus in particular
\begin{equation} \label{eq:J=1iffI=1}
 \mathcal{L} \big( (J_{b a} )_{b \in \Gamma} | I_a = 1\big)
=
 \mathcal{L} \big( (I_b )_{b \in \Gamma} | I_a = 1\big).
\end{equation}
In total, we have
\begin{equation} \label{eq:condtionpoissonapprox}
\mathcal{L} \big( (J_{b a})_{b \in \Gamma}\big) 
\stackrel{\eqref{eq:conditioningonIalpha}}{=}
 \mathcal{L} \big( (J_{b a} )_{b \in \Gamma} | I_a = 1\big)
 \stackrel{\eqref{eq:J=1iffI=1}}{=}
 \mathcal{L} \big( (I_b )_{b \in \Gamma} | I_a = 1\big).
\end{equation}
Since $\varNjk=\sum_{a \in \mathcal{T}_k} I_a$ for any $j\le k\le d$, we can therefore apply Lemma~\ref{lem:multpoisson}.
Setting $Z_k\ :=\ \Po(\EE(\varNjk))$ independently for each $k$, we obtain
\begin{equation} \label{eq:distancedistributions}
 d_{TV} \bigg( \mathcal{L} (\mathbf{X}) ,
\big(Z_j,\ldots,Z_d \big)   \bigg)
\le 
\sum_{a \in \Gamma} \pi_a \left( \pi_a + 
\sum_{b \neq a} \EE|J_{b a} - I_b | \right).
\end{equation}

We want to show that the right-hand side of
\eqref{eq:distancedistributions} is $o(1)$. 
Recall that for every $b \in \Gamma$ by \eqref{eq:q} and Proposition~\ref{prop:localisedprob}  we have 
\begin{equation} \label{eq:expIb}
\EE(I_b)= \pi_b = (1+o(1)) p_{k_b}\bq^{\;\tim
(k_b-j+1)},
\end{equation}
and therefore (cf.~\eqref{eq:expXjk})
\begin{equation} \label{eq:expsumIb}
 \EE(\varNjk) = \sum\nolimits_{b \in \cT_k } \EE(I_b) =
\Theta( n^{k+1} p_k \bq^{\;\tim(k-j+1)}) = O(1),
\end{equation}
where the last equality holds
because we are considering $\cG_\tim$ 
within the critical window.
Furthermore, by \eqref{eq:condtionpoissonapprox} we have
\begin{equation}\label{eq:expJba}
 \EE(J_{b a}) = \Pr(I_b = 1 | I_a = 1).
\end{equation}
We now fix $a \in \Gamma$ and estimate the sum 
$\sum_{b \neq a} \EE|J_{b a} - I_a |$, by distinguishing some cases.

\emph{Case 1:} $K_b = K_a$.
First observe that since $b \neq a$, 
this case can only be possible if $k_b=k_a \ge j+1$ and $C_b 
\neq C_a$. 
Moreover, conditioned on $I_a=1$,
i.e.\ $a$ forming a copy of $M_{j,k_a}$, 
there are $\binom{k_a +1}{j} - 1 = O(1)$ ways to choose 
$b$ such that $K_b=K_a$ and $C_b 
\neq C_a$. 
Furthermore, 
$b$ forms a copy of $M_{j,k_b}=M_{j,k_a}$
 with probability
 $O(\bq^{\;\tim(k_a - j)})$, because  $K_b=K_a$ already exists in $\cG_\tim$ as simplex (and so there is
  no $p_{k_b}=p_{k_a}$ term) and because the flower
 $\mathcal{F}(K_b,C_b)$ 
can share at most one petal with the flower 
$\mathcal{F}(K_a,C_a)$ (and so we lose at most one 
factor $\bq^{\;\tim}$). Thus if we set
\[B_1 = B_1(a)\ :=\ \{b\in\Gamma: b\neq a, K_b=K_a\}\]
we have
\begin{align*}
\sum\nolimits_{b \in B_1} \EE|J_{b a} - I_b| 
& \stackrel{\phantom{\eqref{eq:expIb}}}{\le} \sum\nolimits_{b \in B_1}\big( \EE(J_{b a}) + \EE(I_b) \big)\\
& \stackrel{\eqref{eq:expJba}}{=} 
\sum\nolimits_{b \in B_1}\big( \Pr(I_b = 1 |I_a=1) + \pi_b \big)\\
& \stackrel{\eqref{eq:expIb}}{=}
O\left( 1 \right) \cdot \big( O(\bq^{\;\tim
(k_a - j)}) 
+ O(p_{k_a} \bq^{\;\tim(k_a-j+1)}) \big)\\
&\stackrel{\phantom{\eqref{eq:expIb}}}{=} O(\bq^{\;\tim(k_a -j)}) = o(1),
\end{align*}
where the last equality follows from the facts that 
$\tim(k_a -j)> 1 + o(1)$ and $\bq=o(1)$ (cf.\ \eqref{eq:qissmall}).

\emph{Case 2:} $K_b \neq K_a$, but $K_b \in \mathcal{B}_a$ or $K_a \in \mathcal{B}_b$. 
This means that the event $E^3_{ba}$  does not happen, 
thus $J_{ba}=0$ deterministically by \eqref{eq:indicatorcoupling}.

\emph{Case 2.1:} $K_b \in \mathcal{B}_a$.
Given $K_a$, the set $K_b$ 
must contain at least $j+1$ vertices of $K_a$ 
in order to be forbidden for $K_a$,
because $K_b$ contains at least one petal
(i.e.\ $(j+1)$-set) of
the flower $\mathcal{F}(K_a,C_a)$.
Hence, there are 
$O(n^{k_b -j})$ possible choices for $b$, and thus
if we set
\[ B_{2.1} = B_{2.1}(a) \ :=\ \{b \in\Gamma: b\neq a, K_b\neq K_a,
K_b \in \mathcal{B}_a\},\]
we have
\begin{align*}
\sum\nolimits_{b\in B_{2.1}} \EE|J_{b a} - I_b| 
& \stackrel{\phantom{\eqref{eq:expIb}}}{\le}  \sum\nolimits_{b\in B_{2.1}}\EE(I_b) \\
& \stackrel{\eqref{eq:expIb}}{=} 
\sum\nolimits_{k=j}^d O\left(n^{k-j} p_k \bq^{\;\tim(k-j+1)}  \right) \\
& \stackrel{\eqref{eq:expsumIb}}{=} \sum\nolimits_{k=j}^d 
O\left( \frac{\EE(\varNjk)}{n^{j+1}}\right) =o(1).
\end{align*}

\emph{Case 2.2:} $K_a \in \mathcal{B}_b$. Set
\[ B_{2.2} = B_{2.2}(a) \ :=\  \{b \in\Gamma: b\neq a, K_b\neq K_a,
K_a \in \mathcal{B}_b\}.\]
By exchanging
the roles of $K_a$ and $K_b$ in Case 2.1, with the same argument
we have
\[\sum\nolimits_{b \in B_{2.2}} \EE|J_{b a} - I_b| = o(1). \]

\emph{Case 3:} $K_b \neq K_a$, $K_b \notin \mathcal{B}_a$, and $K_a \notin \mathcal{B}_b$. 
This case contains almost all the summands of 
$\sum_{b \neq a} \EE|J_{b a} - I_a |$, thus we need 
the main terms in the sum to cancel.
The event $E^3_{ba}$ holds,
yielding that if $I_b=1$ then also $J_{ba} = 1$, 
that is $J_{b a} \ge I_b$ deterministically 
and therefore 
\begin{equation} \label{eq:noabsolutevalue}
  \EE|J_{b a} - I_b| = \EE(J_{b
a}) - \EE(I_b).
\end{equation}
There are $(k_b - j +1)$
(potential) petals in $b$ each contained in 
$\binom{n-j-1}{k-j}$ many $(k+1)$-sets that must not form 
$k$-simplices in $\cG_\tim$ 
in order for $b$ to form a copy of $\Mjk$, for each 
$j+1\le k\le d$. 
However some of these forbidden $(k+1)$-sets 
might be double-counted because they contain
more than one petal in $b$, and additionally some of 
these forbidden $(k+1)$-sets
might be forbidden for both $a$ and $b$, and therefore
we already know that they are not simplices if we
condition on $I_a = 1$. 
In either case, any of these $(k+1)$-sets
contains at least two petals and 
so at least $j+2$ vertices are already fixed, thus there
are $O \left(\binom{n-(j+2)}{k+1-(j+2)}\right) = 
O(n^{k-j-1})$ many $(k+1)$-sets that we have to exclude
when counting the sets of size $k+1$
that are forbidden for $b$. In other words, 
the number of $(k+1)$-sets that must
not be simplices is 
$(k_b-j+1)\binom{n-j-1}{k-j} - O(n^{k-j-1})$, 
yielding 
\begin{align*}
\EE(J_{b a}) &\stackrel{\eqref{eq:expJba}}{=} 
\Pr(I_b = 1 | I_a = 1)  \\
&\stackrel{\phantom{\eqref{eq:expJba}}}{=} p_{k_b} \prod\nolimits_{k=j+1}^d (1-\tim \bar p_k)^{(k_b-j+1)
\binom{n-j-1}{k-j} - O(n^{k-j-1})}  \\
& \stackrel{\eqref{eq:q}}{=} 
p_{k_b} \bq^{\;\tim(k_b-j+1)}
\prod\nolimits_{k=j+1}^d (1-\tim \bar p_k)^{-O(n^{k-j-1})}
\\
& \stackrel{\phantom{\eqref{eq:expJba}}}{=} p_{k_b} \bq^{\;\tim(k_b-j+1)} 
\exp \left( \sum\nolimits_{k=j+1}^{d} 
O \left( \frac{\log n}{n^{k-j}} n^{k-j-1} \right) 
\right) \\
&\stackrel{\phantom{\eqref{eq:expJba}}}{=} (1+o(1)) p_{k_b} \bq^{\;\tim(k_b-j+1)}, 
\end{align*}
and thus
\begin{equation} \label{eq:expect_case3}
\EE|J_{b a} - I_b| 
\stackrel{\eqref{eq:noabsolutevalue}}{=} 
\EE(J_{b a}) - \EE(I_b)
= \EE(J_{ba}) - \pi_b 
 \stackrel{\eqref{eq:expIb}}{=} 
 o( p_{k_b} \bq^{\;\tim(k_b-j+1)}).
\end{equation}
Given $a$, the number of $b \in \cT_k$ satisfying the conditions of Case $3$ is $O(n^{k+1})$, hence if
we set
\[  B_3 = B_3(a) \ :=\  \{b \in\Gamma: b\neq a, K_b\neq K_a,
K_b \not\in \mathcal{B}_a,
K_a \not\in \mathcal{B}_b \},\]
we have
\begin{align*}
\sum\nolimits_{b\in B_3} 
\EE|J_{b a} - I_b| & \stackrel{\eqref{eq:expect_case3}}{=}
\sum\nolimits_{k=j}^{d} O(n^{k+1}) \cdot o( p_{k} 
\bq^{\;\tim(k-j+1)}) \\
& \stackrel{\eqref{eq:expsumIb}}{=} 
o(1)\cdot  \sum\nolimits_{k=j}^{d} O\left( \EE(\varNjk) \right) = o(1).
\end{align*}
Since 
$\{ b\in \Gamma : b \neq a\} =
B_1 \cup B_{2.1} \cup B_{2.2} \cup B_3$, putting all the cases together we have that 
$\sum_{b\neq a} \EE|J_{b a} - I_b| = o(1)$ for any fixed $a \in \Gamma$,
as required.

Observe that for symmetry reasons, the quantity $\sum_{b\neq a} \EE|J_{b a} - I_b|=o(1)$ remains the same if the sum is over $b\neq a'$ with $k_{a'}=k_a$. Thus we have
\begin{align}
\sum\nolimits_{a \in \Gamma} \pi_a \sum\nolimits_{b\neq a} 
  \EE|J_{b a} - I_b|
&\stackrel{\phantom{\eqref{eq:expsumIb}}}{=} \sum\nolimits_{k=j}^d o(1) \cdot \sum\nolimits_{a\in \Gamma_k } \pi_a \nonumber\\
& \stackrel{\phantom{\eqref{eq:expsumIb}}}{=} \sum\nolimits_{k=j}^d o(1) \cdot \EE(\varNjk) \nonumber \\
& \stackrel{\eqref{eq:expsumIb}}{=} \sum\nolimits_{k=j}^d o(1)
\cdot O(1) = o(1). \label{eq:unifo(1)}
\end{align}
The right-hand side of \eqref{eq:distancedistributions} is therefore
\begin{align*}
\sum\nolimits_{a \in \Gamma} \pi_a \left( \pi_a + 
\sum\nolimits_{b \neq a} \EE|J_{b a} - I_b | \right) & 
\stackrel{\eqref{eq:unifo(1)}}{=} 
\left(\sum\nolimits_{a \in \Gamma} \pi_a^2 \right) +o(1)
 \\
& \stackrel{\phantom{\eqref{eq:expsumIb}}}{\le}   \left( \max_{a \in \Gamma} \pi_a\right)
\left( \sum\nolimits_{a \in \Gamma} \pi_a \right)+ o(1)\\
& \stackrel{\phantom{\eqref{eq:expsumIb}}}{=}  \left( \max_{a \in \Gamma} \pi_a\right)
\left( \sum\nolimits_{k=j}^d \EE(\varNjk) \right) + o(1)\\ 
& \stackrel{\eqref{eq:expsumIb}}{=}
 \left( \max_{a \in \Gamma} \pi_a\right) 
 \cdot O(1) + o(1), \\
\end{align*}
and 
\begin{align*}
\max_{a \in \Gamma} \pi_a & \stackrel{\eqref{eq:expIb}}{=}
\max_{j\le k \le d} \left( (1+o(1)) p_k 
\bq^{\;\tim(k-j+1)}
\right) \\
& \stackrel{\eqref{eq:expsumIb}}{=} 
\max_{j\le k \le d} 
\frac{\EE(\varNjk)}{\Theta(n^{k+1})}  = O\left(\frac{1}{n^{j+1}} \right) = o(1).
\end{align*}

In conclusion, we have
\begin{equation*}
 d_{TV} \bigg( \mathcal{L} (\mathbf{X}) ,
\big(Z_j,\ldots,Z_d \big)   \bigg)
\xrightarrow{n \to \infty} 0.
\end{equation*}
Since by Corollary~\ref{cor:expectatcritwind}, 
$\lim_{n\to\infty} \EE(\varNjk) = \critm_k$
for every
$j\le k\le d$, we have  
\[ \mathcal{L}(\mathbf{X}) 
\xrightarrow{\enskip \dist \enskip}  \left(
\Po \big(\critm_j \big), \ldots, \Po \big(\critm_d \big) \right),\]
as required.
\qed

\newpage
\section{Glossary}\label{ap:glossary}

For the reader's convenience, we include a glossary of some of the most important terminology and notation defined in the paper.

\vspace{0.3cm}

\subsection{Combinatorial terminology}

\ 

\vspace{0.3cm}

\begin{center}
\begin{tabular}{|l|m{10cm}|l|}
\hline
\textbf{Term} \raisebox{0.2cm}{${\phantom{\big(\big)}}$}
 & \textbf{Informal description} & \textbf{First}\\
 & & \textbf{defined}\\
\hline
$j$-shell & all $j$-simplices on $j+2$ vertices & p.~\pageref{comb:shell}${\phantom{\bigg(\bigg)}}$ \\
\hline
$j$-flower & $j$-simplices within a $k$-simplex containing a common centre & p.~\pageref{comb:flower}${\phantom{\bigg(\bigg)}}$ \\
\hline
copy of $\Mjkhat$ & $K$-localised $j$-flower with $j$-shell containing one petal & p.~\pageref{comb:Mjkhat}${\phantom{\bigg(\bigg)}}$ \\
\hline
copy of $\Mjk$ & $K$-localised $j$-flower & p.~\pageref{comb:Mjk}${\phantom{\bigg(\bigg)}}$ \\
\hline
$K$-localised $j$-simplex & all simplices containing 
the $j$-simplex are contained in $K$ & p.~\pageref{comb:Klocal}${\phantom{\bigg(\bigg)}}$\\
\hline
local $j$-obstacle & $(k+1)$-set $K$ containing $k-j+1$
many $K$-localised $j$-simplices & p.~\pageref{comb:localobst}${\phantom{\bigg(\bigg)}}$\\
\hline
traversability & notion of connectedness on $j$-simplices & p.~\pageref{comb:traversability}${\phantom{\bigg(\bigg)}}$\\
\hline
\end{tabular}
\end{center} 

\vspace{1cm}

\subsection{Cohomology terminology}

\begin{center}
\begin{tabular}{|l|m{10cm}|l|}
\hline
\textbf{Term} \raisebox{0.2cm}{${\phantom{\big(\big)}}$}
 & \textbf{Informal description} & \textbf{First}\\
 & & \textbf{defined}\\
\hline
$j$-cochain & function on ordered $j$-simplices & p.~\pageref{cohom:cochain}${\phantom{\bigg(\bigg)}}$\\
\hline
$C^j(\cG;R)$ & group of $j$-cochains & p.~\pageref{cohom:Cj}${\phantom{\bigg(\bigg)}}$\\
\hline
coboundary operator $\delta^j$ & generates $(j+1)$-cochain from $j$-cochain & p.~\pageref{cohom:coboperator}${\phantom{\bigg(\bigg)}}$ \\
\hline
$j$-cocycle & $j$-cochain in $\ker \delta^j$, i.e.\ all $(j+1)$-simplices have zero boundary & p.~\pageref{cohom:cocycle}${\phantom{\bigg(\bigg)}}$\\
\hline
$j$-coboundary & $j$-cochain in $\im \delta^{j-1}$, i.e.\ generated from $(j-1)$-cochain  & p.~\pageref{cohom:coboundary}${\phantom{\bigg(\bigg)}}$\\
\hline
$H^j(\cG;R)$ & $j$-th cohomology group over $R$: $\ker\delta^j / \im \delta^{j-1}$ & p.~\pageref{cohom:cohomgroup}${\phantom{\bigg(\bigg)}}$\\
\hline
\end{tabular}
\end{center}

\vspace{1cm}

\newpage
\subsection{Probabilities and birth times}

\ 

\vspace{0.3cm}

\begin{center}
\begin{tabular}{|l|m{6cm}|l|}
\hline
\textbf{Symbol} \raisebox{0.2cm}{${\phantom{\big(\big)}}$}
 & \textbf{Informal description} & \textbf{First}\\
 & & \textbf{defined}\\
\hline
$\barallp$ & vector defining process direction & p.~\pageref{prob:barallp}${\phantom{\bigg(\bigg)}}$\\
\hline
$\tim_{\max}$ & $\tim$ at end of process: $1/\bar p_d$ & p.~\pageref{prob:timmax}${\phantom{\bigg(\bigg)}}$\\
\hline
$\tim_j^*$ & time when last $\Mjkhat$ disappears & p.~\pageref{prob:timstar}${\phantom{\bigg(\bigg)}}$\\
\hline
$\tim'$ & $1-\frac{\log \log n}{10 d \log n}$ & p.~\pageref{prob:timprime}${\phantom{\bigg(\bigg)}}$\\
\hline
$\tim''$ & first time after $\tim'$ that no $\Mjk$ exists & p.~\pageref{prob:timsecond}${\phantom{\bigg(\bigg)}}$\\
\hline
\end{tabular}
\end{center}

\vspace{1cm}

\subsection{Random variables}

\ 

\vspace{0.3cm}

\begin{center}
\begin{tabular}{|l|m{6cm}|l|}
\hline
\textbf{Variable} \raisebox{0.2cm}{${\phantom{\big(\big)}}$}
 & \textbf{Informal description} & \textbf{First}\\
 & & \textbf{defined}\\
\hline
$\varNjk$ & number of copies of $\Mjk$ in $\cG_\tim$ & p.~\pageref{var:Njk}${\phantom{\bigg(\bigg)}}$\\
\hline
$\varMjkhat$ & number of copies of $\Mjkhat$ in $\cG_\tim$ & p.~\pageref{var:Mjkhat}${\phantom{\bigg(\bigg)}}$\\
\hline
\end{tabular}
\end{center}

\newpage

\subsection{Parameters}

\ 

\vspace{0.3cm}

\begin{center}
\begin{tabular}{|l|m{10cm}|l|}
\hline
\textbf{Parameter} \raisebox{0.2cm}{${\phantom{\big(\big)}}$}
 & \textbf{Informal description} & \textbf{First}\\
 & & \textbf{defined}\\
\hline
$\bar\alpha_k$ & logarithmic term in $\bar p_k$ & p.~\pageref{param:alphbetgamm}${\phantom{\bigg(\bigg)}}$\\
\hline
$\bar\beta_k$ & sublogarithmic term in $\bar p_k$ & p.~\pageref{param:alphbetgamm}${\phantom{\bigg(\bigg)}}$\\
\hline
$\bar\gamma_k$ & polynomial correction exponent in $\bar p_k$ & p.~\pageref{param:alphbetgamm}${\phantom{\bigg(\bigg)}}$\\
\hline
$\bar p_k$ & probability of $k$-simplex existing at $\tim =1$: $\displaystyle \frac{\bar \alpha_k \log n + \bar \beta_k}{n^{k-j+\bar \gamma_k}}(k-j)!$ & p.~\pageref{param:barpk}${\phantom{\bigg(\bigg)}}$\\
\hline
$\bar\logp_k$ & $j+1-\bar\gamma_k-(k-j+1)\sum_{i=j+1}^{d}\bar\alpha_i$ & p.~\pageref{param:lambdmunu}${\phantom{\bigg(\bigg)}}$\\
\hline
$\bar\sublogp_k$ & $- (k-j+1)\sum_{i=j+1}^{d}\frac{\bar\beta_i}{n^{\bar\gamma_i}} +
  \begin{cases}
        \log\log n & \text{if }\bar\alpha_k\not=0,\\
        \log(\bar\beta_k) & \text{if }\bar\alpha_k=0
      \end{cases}$ & p.~\pageref{param:lambdmunu}${\phantom{\bigg(\bigg)}}$\\
\hline
$\bar\constp_k$ & 
$\begin{cases}
-\log((j+1)!) & \text{if }k=j, \\
-\log(j!)-\log(k-j+1) +\log(\bar\alpha_k) & \mbox{if } \bar\alpha_k\neq 0,\\
-\log(j!)-\log(k-j+1) & \text{otherwise}
\end{cases}$
& p.~\pageref{param:lambdmunu}${\phantom{\bigg(\bigg)}}$\\
\hline
$\bar k$ & index $j \le \bar k \le d$ such that $\bar\logp_{\bar k}\log n + \bar\sublogp_{\bar k} + \bar\constp_{\bar k} = 0$ & p.~\pageref{param:bark}, (see also \raisebox{0.2cm}{${\phantom{\big(\big)}}$}\\
& & Def.~\ref{def:indices}, p.~\pageref{def:indices})\\
\hline
$k_0$ & index $j+1 \le k_0 \le d$ such that $\bar\alpha_{k_0} \not= 0$ & p.~\pageref{param:alphbetgamm}, (see also \raisebox{0.2cm}{${\phantom{\big(\big)}}$}\\
& & Def.~\ref{def:indices}, p.~\pageref{def:indices})\\
\hline
$\ell$ & index $j\le \ell \le d$ such that at time $\tim_j^*$, a copy of $\Mjlhat$ vanishes & p.~\pageref{param:ell}${\phantom{\bigg(\bigg)}}$\\
\hline
\end{tabular}
\end{center}

\end{document}